\documentclass[10pt]{amsart}

\usepackage[T1]{fontenc}
\usepackage{fouriernc}

\usepackage[cal=boondoxo]{mathalfa}
\DeclareSymbolFontAlphabet{\mathcalorig} {symbols}

\usepackage{amssymb}
\usepackage{amsmath}
\usepackage{amsthm}
\usepackage{amssymb}
\usepackage{amsfonts}
\usepackage{latexsym}
\usepackage{float}

\usepackage{longtable}
\usepackage{afterpage}
\usepackage{array}
\usepackage{verbatim}
\usepackage{graphicx} %% for pictures
\usepackage{color}%% for color text
\usepackage{mysects}  %% for section headings
\usepackage{fancyhdr} %% for top of pages
\usepackage[labelformat=simple]{caption}
\usepackage{setspace}
\usepackage{bm}

\usepackage{enumitem}

\usepackage{microtype}
\usepackage{lineno}

\usepackage{appendix}
\newtheorem{theorem}{Theorem}[section]
\newtheorem{lemma}[theorem]{Lemma}
\newtheorem{proposition}[theorem]{Proposition}
\theoremstyle{definition}

\newtheorem{Note}[theorem]{Remark}
\newtheorem{Corollary}[theorem]{Corollary}
\newtheorem{conjecture}[theorem]{Conjecture}
\theoremstyle{remark}
\newtheorem{remark}[theorem]{Remark}
\newtheorem{notation}[theorem]{Notation}
\numberwithin{equation}{section}

\def\ve {\varepsilon}

\let\LaTeXStandardTableOfContents\tableofcontents
\renewcommand{\tableofcontents}{
\begingroup
\renewcommand{\bfseries}{\relax}
\LaTeXStandardTableOfContents
\endgroup
}

\begin{document}

\title{Integral points on Markoff type cubic surfaces}

\author{Amit Ghosh}
\address{Oklahoma State University\\ Stillwater, OK 74078, USA}
\curraddr{}
\email{ghosh@okstate.edu}
\thanks{}

\author{Peter Sarnak}
\address{Institute for Advanced Study and  Princeton University \\Princeton, NJ 08540, USA}
\curraddr{}
\email{sarnak@math.ias.edu}
\thanks{}
\date{}

\begin{abstract} For integers $k$, we consider the affine cubic surface $V_{k}$ given by $M({\bf x})=x_{1}^2 + x_{2}^2 +x_{3}^2 -x_{1}x_{2}x_{3}=k$. We show that for almost all $k$ the Hasse Principle holds, namely  that $V_{k}(\mathbb{Z})$ is non-empty if $V_{k}(\mathbb{Z}_p)$ is non-empty for all primes $p$, and that there are   infinitely many $k$'s for which it fails. The Markoff morphisms act on  $V_{k}(\mathbb{Z})$ with finitely many orbits and a numerical study points to some basic conjectures about these ``class numbers'' and Hasse failures. Some of the analysis may be extended to less special affine cubic surfaces.
\end{abstract}

\maketitle

\setcounter{secnumdepth}{5}
\setcounter{tocdepth}{3}
\setcounter{section}{0}
\setcounter{subsection}{4}
\setcounter{subsubsection}{5}
\setcounter{tocdepth}{1}
\tableofcontents

\section{Introduction}
\label{intro}
Little is known about the values at integers assumed by \textit{affine cubic forms} $F$ in three variables. Unless otherwise stated, by an \textit{affine form}   $f$ in $n$-variables we mean   $f\in \mathbb{Z}[x_1,\ldots ,x_n]$ whose leading homogeneous term $f_0$ is non-degenerate\footnote{that is it cannot be transformed to a polynomial of fewer than $n$ variables by a linear change of variables.} and such that $f -k$ is (absolutely) irreducible for all constants $k$. For $k\in \mathbb{Z}$ set
\begin{equation}\label{Intro2}
V_{k,F}=\{{\bf x}= (x_1,x_2,x_3):\ F({\bf x})=k\},
\end{equation}
and $\mathfrak{v}_{F}(k):=\vert V_{k,F}(\mathbb{Z})\vert$.
The basic question is for which $k$ is $V_{k,F}(\mathbb{Z})\neq \emptyset$, or more generally infinite or Zariski dense in $V_{k,F}$ ?

 A prime example is $F=S$, the sum of three cubes:
\begin{equation}\label{Intro1}
S(x_1,x_2,x_3)=x_{1}^{3}+x_{2}^{3} + x_{3}^{3}.
\end{equation}
 There are obvious local congruence obstructions, namely that $V_{k,S}(\mathbb{Z})=\emptyset$ if $k \equiv 4\ \text{or}\ 5\, ({\rm  mod}\, 9)$, but beyond that it is possible that the answers to all three questions is yes for all the other $k$'s, which we call the \textit{admissible} values (see  \cite{MorBook}, \cite{CV94}). It is known that strong approximation in its strongest form fails for $V_{k,S}(\mathbb{Z})$; the global obstruction coming from an application of cubic reciprocity (\cite{Cas85}, \cite{HB92}, \cite{CThW12}). Moreover, \cite{Le} and \cite{Be} show that $V_{1,S}(\mathbb{Z})$ is Zariski dense in $V_{1,S}$.

The case when the cubic polynomial $F(x_1,x_2,x_3)$ factors into linear factors can be studied algebraically using divisor theory, and is apparently quite different to our irreducible $F$. If $F$ is the split norm form  $N({\bf x})=x_{1}x_{2}x_{3}$, then every $V_{k,N}$ is non-empty, and  for $k$ non-zero, $\mathfrak{v}_{N}(k)$ is finite and is a divisor function.

For a $\mathbb{Q}$-anistropic torus given by $N({\bf x})={Nm}_{K\slash\mathbb{Q}}(\alpha_{1}x_{1}+\alpha_{2}x_{2}+\alpha_{3}x_{3})$, where $\alpha_{1},\,\alpha_{2},\,\alpha_{3}$ is a $\mathbb{Z}$-basis of an order in a cubic number field $K$, the Dirichlet Unit Theorem coupled with the action ${\bf w} \rightarrow u{\bf w}$ of the unit group on the homogeneous space and the theory of divisors, allows for the study of $V_{k,N}({\mathbb Z})$. It consists of a finite number  $\mathfrak{h}_{N}(k)$ of orbits (putting $\mathfrak{h}_{N}(k) =\mathfrak{v}_{N}(k)=0$ if $V_{k,N}=\emptyset$), is infinite if it is non-empty and is Zariski dense if $K$ is totally real. The dependence of $\mathfrak{h}_{N}(k)$ on $k$ is subtle, especially if the class number $H$ of the order is not one. Most $k$'s are not represented; in fact   \cite{Odo75} shows that 
\begin{equation}\label{Intro3}
\big{|}\{|k|\leq X: \mathfrak{v}_{N}(k)\neq 0\}\big{|} \sim \frac{1}{H}\big{|}\{|k|\leq X:  k\ \text{admissible}\}\big{|} \sim C X(\log X)^{-\frac{2}{3}}\, ,
\end{equation}
as $X \to \infty$ .  The question of the density of Hasse failures for norms of elements in a number field $K$ is studied in \cite{bn16}.

To measure the richness of representations by $f$, we say that $f$ is \textit{perfect} if $V_{k,f}(\mathbb{Z})$ is Zariski dense in $V_{k,f}$ for all but finitely many admissible $k$'s; we say it is \textit{almost perfect} if the same holds for almost all admissible $k$ (in the sense of natural density); and $f$ is \textit{full} if $\mathfrak{v}_{f}(k) \to \infty$ as $k \to \infty$ for almost all admissible $k$'s.  For an affine form, it
follows from \cite{LW} and \cite{SchBook} that the admissible $k$'s are given
in terms of a congruence condition as in the case of $S$.

Much more is known about cubic forms in the ``subcritical'' case of  forms in four or more variables or diagonal forms $f= x_1^{a_1} + \ldots + x_b^{a_b}$ with $\sum_{j=1}^{b}a_j^{-1} >1$ and $b\geq 3$ (see \cite{Hoo16}, \cite{VW}, \cite{bro15} for example)  and in the ``super-critical" case of two variables (\cite{CFZ}). The basic analytic feature in the subcritical case is that the average number of representations of $k$ is $k^{\delta}(\log k)^A$ for some $\delta >0$, while in the critical case, $\delta =0$.  If $f$ is a cubic polynomial, $n\geq 10$ and $f_0$ is nonsingular  then $f$ is perfect (\cite{B-HB})\footnote{They show that $\vert V_{k,f}(\mathbb{Z})\vert=\infty$ for $k$ admissible from an asymptotic count which is flexible enough to deduce that $V_{k,f}(\mathbb{Z})$ is Zariski dense in $V_{k,f}$.}. In a recent paper \cite{Hoo16}, it is shown that if $f=f_0$  and is nonsingular with $n\geq 5$, then $f$ is full, while conditional on the Riemann Hypothesis for certain Hasse-Weil $L$-functions, the same is true for $n\geq 4$. Moreover, it is conjectured there that  any such $f$ with $n\geq 4$ is perfect. For cubic $f$ in two variables (supercritical case) the celebrated theorems (\cite{Thue}, \cite{Siegel}) assert that $V_{k,f}(\mathbb{Z})$ is finite and moreover only for very few of the admissible $k$'s is $V_{k,f}(\mathbb{Z})$ non-empty (\cite{Sch}).

Returning to the critical dimension $n=3$  for affine cubic forms, there are well-known examples of $F$ which are not perfect, see (\cite{Mor53}, \cite{CG66})\footnote{The projective cubic surface for \cite{CG66}, namely $F(x_1,x_2,x_3)=10x_4^3$ with $F(x_1,x_2,x_3)= 5x_1^3 + 12x_2^3 + 9x_3^3$, fails the Hasse principle over $\mathbb{Q}$; from which it follows that $V_{k,F}(\mathbb{Z})$ fails the Hasse principle over $\mathbb{Z}$ for $k=10w^3$. There are many other such projective cubic surfaces over $\mathbb{Q}$ (see Sec.\,4 of \cite{Br}).}  and also our example of $M$ below; however it is possible that $F$ is always full (see the discussion at the end of the Introduction).

This paper is concerned with $F=M$ where
\begin{equation}\label{Intro4}
M({\bf x})= x_{1}^2 + x_{2}^2 +x_{3}^2 -x_{1}x_{2}x_{3}.
\end{equation}
The affine cubic surface $V_{0,M}(\mathbb{Z})$ was studied by Markoff (\cite{Markoff1}, \cite{Markoff2}); the points $(x_1,x_2,x_3)\in V_{0,M}(\mathbb{Z})$ with $x_j \in \mathbb{N}$ being essentially the ``Markoff triples'' . The reason that one can study $V_{0,M}(\mathbb{Z})$, or more generally $V_{k,M}(\mathbb{Z})$ is that there is a descent group action albeit non-linear. The Vieta involutions $\mathcal{V}_j$ with $\mathcal{V}_{1}(x_1,x_2,x_3)=(x_2x_3-x_1,x_2,x_3)$ and similarly for $\mathcal{V}_2$, $\mathcal{V}_3$, preserve $M$, as do permutations of the $x_j$'s and switching the signs of two of the $x_j$'s. We denote by $\Gamma$ the group of polynomial affine transformations generated as above. Then, $\Gamma$ preserves $V_{k,M}(\mathbb{Z})$ and except for the case of the Cayley cubic with $k=4$ (see Sec.\,\ref{subsec4b}), $V_{k,M}(\mathbb{Z})$ decomposes into a finite number  $\mathfrak{h}_{M}(k)$ of $\Gamma$-orbits. For example, if $k=0$, then $\mathfrak{h}_{M}(0)=2$ corresponds to the orbits of $(0,0,0)$ and $(3,3,3)$ (\cite{Markoff2}). If $V_{k,M}(\mathbb{Z})\neq \emptyset$ (so that $\mathfrak{h}_{M}(k) >0$) and $k\geq 5$ or $k< 0$ with $k$ not a square, which will be our cases of interest, then each $\Gamma$-orbit is infinite and even Zariski dense in $V_{k,M}$ (see \cite{CZ}, \cite{CL} and   Sec. \ref{param}). In particular, for $k\geq 5 $ and $k$ not a square, or $k\leq 0$
\begin{equation}\label{zariski}
V_{k,M}(\mathbb{Z}) \neq \emptyset \quad \text{iff} \quad V_{k,M}(\mathbb{Z}) \text{ is Zariski dense in}\ V_{k,M}.
\end{equation}
Moreover, $V_{k,M}(\mathbb{Z})$ contains polynomial  parametric solutions $\left(x_1(m),x_2(m),x_3(m)\right)$ if and only if $k=4+\nu^2$, in which case it contains a line (see Sec. \ref{param} for a direct proof). In \cite{BGS16} and \cite{BGS17}, it is shown that these affine cubic surfaces  with $V_{k,M}(\mathbb{Z})\neq \emptyset$ satisfy a form of strong approximation\footnote{In its strongest form this fails as is shown using quadratic reciprocity in Section \ref{sec7}, see \eqref{7a}.}, after taking into account the possible finite orbits of $\Gamma$ in $V_{k,M}(\bar{\mathbb{Q}})$.  Our goal in this paper is to study the set of $k$'s for which $\mathfrak{h}_{M}(k)>0$.

The first issue is to determine the congruence obstructions for $k$. This is elementary and in Section \ref{sec5} we show that $V_{k,M}(\mathbb{Z}\slash p^{n}\mathbb{Z})\neq \emptyset$ unless $k \equiv 3\ ({\rm mod}\, 4)$ or $k \equiv \pm 3\ ({\rm mod}\, 9)$.  Recall that $k$  is admissible means $k$ does not satisfy any of these congruences. The number of $0 < k \leq K$ (or $0 < -k \leq K$) which are admissible is $\frac{7}{12}K + O(1)$. Any admissible $k$ for which $\mathfrak{h}(k)=0$ is called a Hasse failure (since in this case $V_{k,M}(\mathbb{Z})$ is empty but there is no congruence obstruction).

In order to study $\mathfrak{h}_{M}(k)$ both theoretically and numerically, we give an explicit reduction (descent) for the action of $\Gamma$ on $V_{k,M}(\mathbb{Z})$. For this purpose, it is convenient to remove an explicit set of special admissible $k$'s, namely those for which there is a point in $V_{k,M}(\mathbb{Z})$ with $|x_j|=0,\,1$ or $2$. These $k$'s take the form (i) $k=u^2 + v^2$ or (ii) $4(k-1)=u^2 + 3v^2$ or (iii) $k=4+ u^2$. The number of these special $k$'s (which we refer to as {\it exceptional}) with $0\leq k\leq K$ is asymptotic to $C'\frac{K}{\sqrt{\log K}}$. The remaining admissible $k$'s are called {\it generic} (all negative admissible $k$'s are generic). For them, we have the following elegant reduced forms

\begin{theorem}\label{Thm1}\ 
\begin{enumerate}[wide,labelindent=0pt,label=(\roman*).]
\item Let $k\geq 5$ be  generic and consider the compact set 
\[
\mathfrak{F}_{k}^{+} = \{{\bf u}\in \mathbb{R}^3 : 3\leq u_1\leq u_2\leq u_3\, ,\  u_1^2 +u_2^2+u_3^2 +u_1u_2u_3 = k\}\, .
\]
The points in  $\mathfrak{F}_{k}^{+}(\mathbb{Z})= \mathfrak{F}_{k}^{+} \cap \mathbb{Z}^{3}$ are $\Gamma$-inequivalent, and any ${\bf x}\in V_{k,M}(\mathbb{Z})$ is $\Gamma$-equivalent to a unique point ${\mathbf u}'=(-u_1,u_2,u_3)$ with ${\mathbf u}=(u_1,u_2,u_3) \in \mathfrak{F}_{k}^{+}(\mathbb{Z})$ .
\item Let $k<0$ be admissible and consider the compact set 
\[
\mathfrak{F}_{k}^{-} = \{{\bf u}\in \mathbb{R}^3 : 3\leq u_1\leq u_2\leq u_3\leq \frac{1}{2}u_1u_2\,  ,\  u_1^2 +u_2^2+u_3^2 - u_1u_2u_3 = k\}\, .
\]
The points in  $\mathfrak{F}_{k}^{-}(\mathbb{Z})= \mathfrak{F}_{k}^{-} \cap \mathbb{Z}^{3}$ are $\Gamma$-inequivalent, and any ${\bf x}\in V_{k,M}(\mathbb{Z})$ is $\Gamma$-equivalent to a unique point  ${\mathbf u}=(u_1,u_2,u_3) \in \mathfrak{F}_{k}^{-}(\mathbb{Z})$ .
\end{enumerate}
\end{theorem} 

 The Theorem is illustrated for $k>5$ in  Figs. \ref{fig:fset1} and \ref{fig:fset2} with $k=3685$ where $\mathfrak{h}_{M}(3685)=6$, and for $k<0$ in  Figs. \ref{fig:fset1a} and \ref{fig:fset2a} with $k= -3691$, where  $\mathfrak{h}_{M}(-3691)=9$.  The lattice points $V_{k,M}(\mathbb{Z})$ are highlighted and the  fundamental sets  indicated in a polygonal region. 

\begin{figure}[!ht]
\centering
  \includegraphics[width=.65\linewidth]{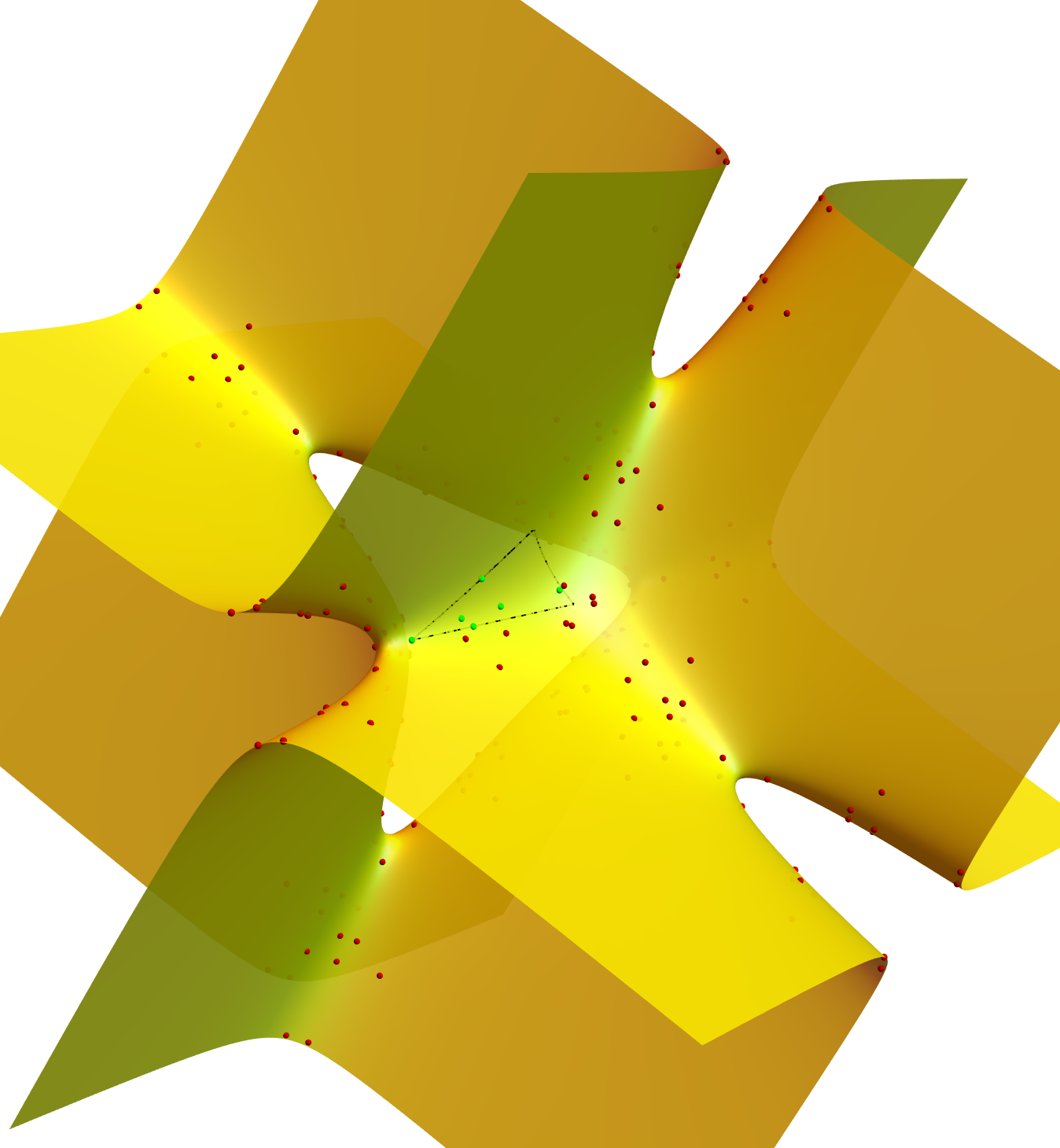}
  \caption{Lattice points and fundamental set (triangular) for $k=3685$.}
  \label{fig:fset1}
\end{figure}

\begin{figure}[!ht]
\centering
  \includegraphics[width=.65\linewidth]{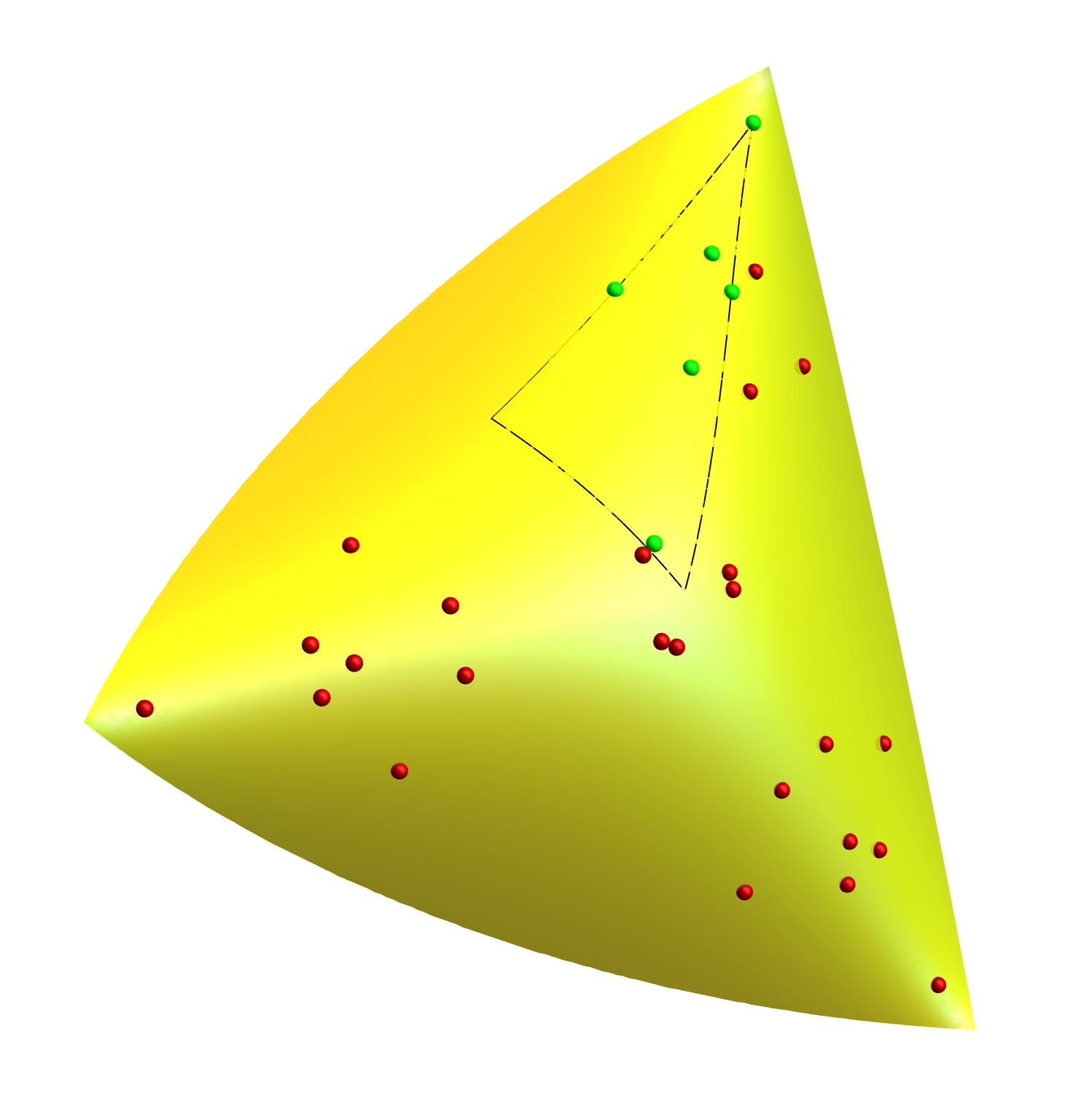}
  \caption{Closeup of fundamental set (triangular) for $k=3685$.}
  \label{fig:fset2}
\end{figure}

\begin{figure}[!ht]
\centering
  \includegraphics[width=.70\linewidth]{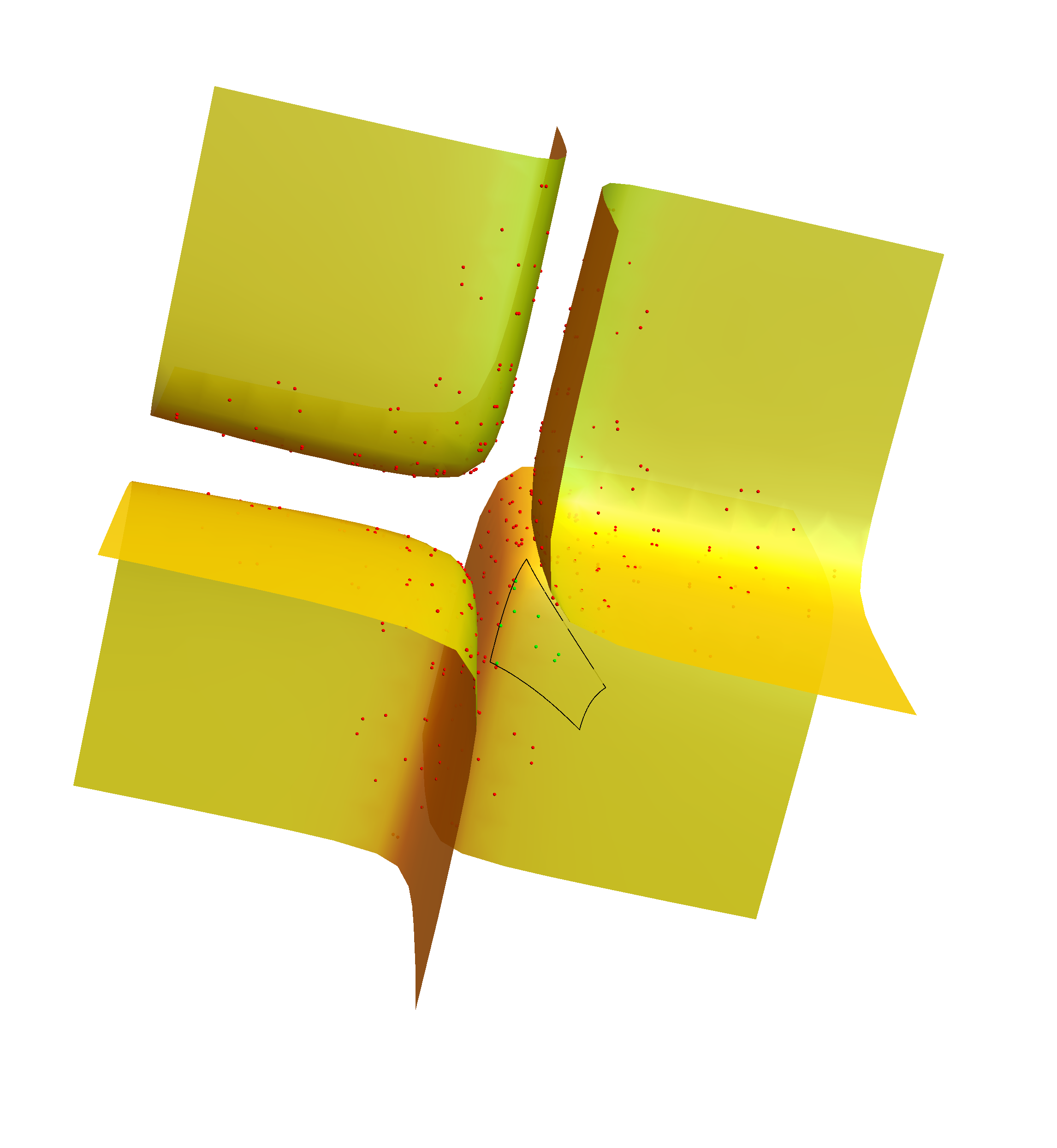}
  \caption{Lattice points and fundamental set  for $k=-3691$.}
  \label{fig:fset1a}
\end{figure}

\begin{figure}[!ht]
\centering
  \includegraphics[width=.70\linewidth]{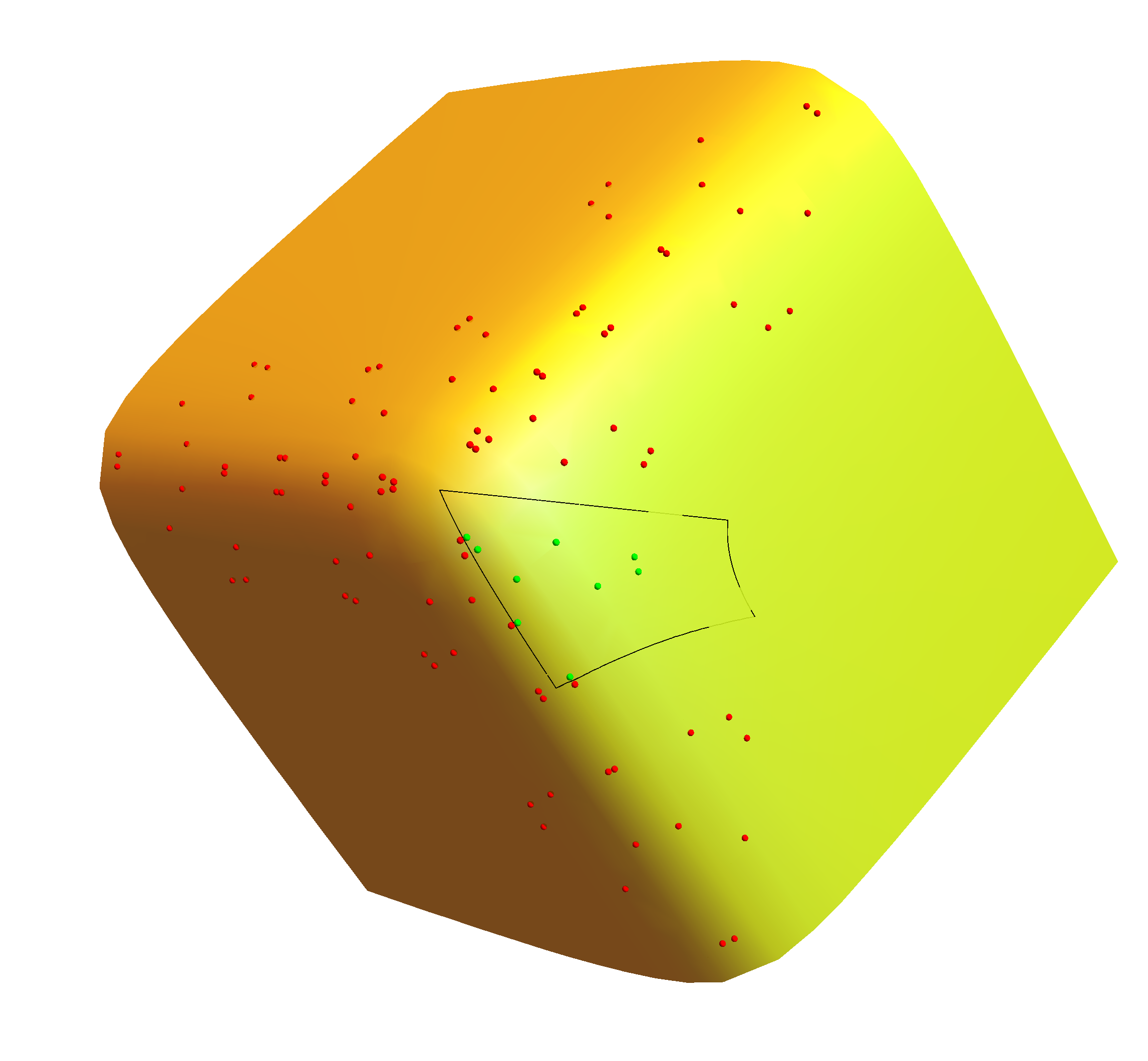}
  \caption{Closeup of fundamental set  for $k=-3691$.}
  \label{fig:fset2a}
\end{figure}
Some simple consequences of Theorem \ref{Thm1} are (see the discussion in Sec.\,\ref{sec2} and also Secs.\,\ref{sec6} and \ref{sec7})\, :
\begin{enumerate}[wide,labelindent=0pt,label=(\alph*).]
\item $V_{46}(\mathbb{Z})=\emptyset$, that is $\mathfrak{h}_{M}(46)=0$, this being the first positive Hasse failure\, .
\item $\mathfrak{h}_{M}(-2)=1$ with all solutions equivalent to the point $(3,3,4)$; while $k=-4$ is the first negative Hasse failure\, .
\item $\mathfrak{h}_{M}(k)\ll_{\ve} |k|^{\frac{1}{3}+ \ve}$ as $k \to \pm \infty$. This follows from the fact that when considering the values taken by the corresponding indefinite quadratic form in the $y$ and $z$ variable, for each fixed $x$, the units are bounded in number due to the restrictions imposed by the fundamental sets\, .
\item Let $\mathfrak{h}^{\pm}_{M}(k)= |\mathfrak{F}^{\pm}_{k}(\mathbb{Z})|$ where $\pm = \text{sgn}({k})$, this being defined for any $k$.  Then Theorem \ref{Thm1} implies that for generic $k$, $\mathfrak{h}_{M}^{\pm}(k)= \mathfrak{h}_{M}(k)$ while otherwise $\mathfrak{h}_{M}(k) \leq \mathfrak{h}_{M}^{\pm}(k)$. We have 
\begin{equation}\label{avg}
\sum_{\substack{k\neq 4\\ |k|\leq K}}\mathfrak{h}_{M}^{\pm}(k)\sim C^{\pm}K(\log K)^2 , 
\end{equation}
 where $C^{\pm}>0$ and  $K \to \infty$ (see Lemmas \ref{average} and \ref{average2}).
\end{enumerate}
So, as expected in this case of critical dimension 3, the numbers $\mathfrak{h}_{M}(k)$ are small on average. On the other hand  the fact that this average grows, albeit very slowly, is a key feature as it suggests that $\mathfrak{h}_{M}(k)$ might be non-zero for many $k$'s. In Section \ref{sec9}, we report on some numerical experiments using Theorem 1.1 to find the Hasse failures among the generic $k$'s when $0<  k < 6\times 10^8$. These suggest that
\begin{equation}\label{Intro5}
\sum_{\substack{0<k\leq K \\ k\,\text{admissible} \\ \mathfrak{h}_{M}(k)=0}} 1 \sim C_0K^\theta ,
\end{equation}
with $C_0 >0$ and $\theta \approx 0.8875..$\,. We also provide results concerning other statistics for the $\mathfrak{h}_{M}(k)$'s for $k$ in this range (see Sec.\,\ref{sec9} for the numerics concerning the numbers $\mathfrak{h}_{M}(k)$ and some conjectures that these support).

Our main result concerns the values assumed by $M$ and the Hasse failures in \eqref{Intro5}; we prove that $M$ is almost perfect but not perfect.

\begin{theorem}\label{Thm2}\ 
\begin{enumerate}[wide,labelindent=0pt,label=(\roman*).]
\item There are infinitely many Hasse failures. More precisely, the number of  $0<k \leq K$ and $-K \leq k <0$ for which the Hasse Principle fails is at least   $\sqrt{K}(\log K)^{-\frac{1}{2}}$ for $K$ large.
\vspace{3pt}
\item $M$ is almost perfect, that is
\[
\#\big\{|k|\leq K :\,k \ \text{admissible},\ \mathfrak{h}_{M}(k)=0\big\} = o(K)\, ,
\]
as $K \to \infty$\, and for almost all admissible $k$, $V_{k}(\mathbb{Z})$ is Zariski dense in $V_{k}$\ . 
\end{enumerate}
\end{theorem}

\begin{remark}\ 
\begin{enumerate}[wide,labelindent=0pt,label=(\alph*).]
\item The proof of $(i)$ is based on quadratic reciprocity and a global factorization that arises for special $k$'s connected to the singular Cayley cubic  $V_{4,M}$. If $k=4 +\beta\nu^2$, with  $\beta$ carefully chosen and $\nu$'s having its prime factors in certain arithmetic progressions, we show that $V_{k,M}(\mathbb{Z})=\emptyset$  even though $k$ is generic. Explicit examples are given in Sec.\,\ref{sec7}. Some of these obstructions to integer solutions are similar to ones found by  Mordell \cite{Mor53} for similar cubic equations, and also to the ``Integer Brauer-Manin obstructions'' in \cite{CThW12}. Following our posting of an earlier version of this paper, \cite{LM} and \cite{CTWX} computed explicitly the Brauer groups of these affine Markoff surfaces, as well as the corresponding integral Brauer-Manin obstructions. They find that the Hasse failures in (i) and (ii) of Prop. \ref{HF1} are accounted for by their obstructions. However, the analysis leading to Hasse failures in part (iii) of the Proposition uses both reciprocity and Markoff descent, and they are not accounted for by the integral Brauer-Manin obstruction alone. In any case, all of these algebraic obstructions are far fewer (they are of order of magnitude $\sqrt{K}$) than the Hasse failures that we found numerically, indicating that any simple description of the latter is perhaps not possible.

\item In the recent paper \cite{GMS}, the Hasse failure (i) is exploited to give failures of profinite local to global principles for commutator equations in $SL_2(\mathcal{O})$ for $\mathcal{O}$ a ring of $S$-integers.

\item The proof of $(ii)$, when combined with Theorem \ref{Thm1} yields further information about the $\mathfrak{h}_{M}(k)$'s for generic $k$'s. If $t\geq 0$ is fixed, then
\[
\#\big\{0\leq |k|\leq K :\, \mathfrak{h}_{M}(k)=t,\  k\ \text{generic}\big\} = o(K),
\]
as $K \to \infty$\,. So for  generic $k$, $\mathfrak{h}_{M}(k)\to \infty$ for almost all $k$. 

\par
\item Our approach to proving that $M$ is full is to look for points in $V_{k,M}(\mathbb{Z})$ with $\vert k\vert\leq K$ in a region $\mathcal{R}$ where $x_1$ is small (roughly of size a power of $\log{K}$) and $x_2,\ x_3$ vary in a sector (so they are of the same size). $\mathcal{R}$ is contained in the fundamental domains $\mathfrak{F}_k^{\pm}$ and retains the \textit{tentacles} (cusps) of the latter, this being critical to ensuring that the average for $\vert k\vert\leq K$, of the number of points in $V_{k,M}(\mathbb{Z})\cap\mathcal{R}$ grows with $k$. For a given $x_1$, $M$ is a (indefinite) binary quadratic form in $x_2,\ x_3$ and this allows one to use the methods developed in \cite{BF11} and \cite{BG06} to show that $M$ assumes a positive proportion of the $k$'s. Our proof that $M$ is full is much more delicate. As with the proofs that cubic forms in many variables (starting with the case of a sum of four cubes \cite{Dav}) represents almost all admissible numbers, we compare the number of points in $V_{k,M}(\mathbb{Z})\cap\mathcal{R}$, to an arithmetic function $\delta^{(m)}(k)$ (see Sec.\,\ref{sec8}; here $m$ is a secondary parameter) which is a product of local densities of solutions. While this heuristic for the count can be way off for certain $k$'s (e.g. for the Hasse failures), we show that its variance from the actual count when averaged over $k$, is small enough to conclude that for almost all $k$'s, $\delta^{(m)}(k)$ is a good approximation. The fullness then follows after showing that $\delta^{(m)}(k)$ is large for most $k$'s\,. That $M$ is almost perfect then follows from \eqref{zariski} and that $M$ is full. The proof of the vanishing of the variance boils down to examining the ``diagonal'' and ``off-diagonal'' terms in \eqref{mo15}. For the first, we make use of the divisor analysis for varying quadratic forms \cite{BG06}, while for the second a modern treatment of Kloosterman's method for ternary quadratic forms \cite{Ni} allows for uniform control of the contributions of the varying forms.
\end{enumerate}
\end{remark}

To end the introduction, we return to a discussion of the general  affine cubic form $F$ in three variables. The study of the level sets $V_{k,M}(\mathbb{Z})$, for example \eqref{zariski} using the Markoff group is very special. It applies to $F$'s of the form $F=F_0 + G$, where 
\begin{equation}\label{cform}
F_0 = cx_1x_2x_3 \quad  {\rm and}\quad G=\sum_{i,j} a_{ij}x_ix_j  + \sum_{i} a_i x_i+ a,
\end{equation}
with $a_{jj}=\pm1$ for $j=1,\; 2,\; 3$ and $c,\; a,\; a_{ij},\; a_{i} \in \mathbb{Z}$, as well as $F$'s obtained from these via integral affine linear substitutions (see Appendix \ref{invariants}). Among these special affine forms are ones  for which $V_k$ carry explicit integral points and even parametric curves, for every $k$. This coupled with the action of the corresponding Markoff group leads to $V_k(\mathbb{Z})$ being Zariski dense for every $k$. Thus, the form is both perfect and `universal' in the sense that it represents every $k$. Explicit examples are 
\begin{equation}\label{You1}
U_1(x_1,x_2,x_3)= x_1 + x_1^2 + x_2^2 + x_3^2 -x_1x_2x_3,
\end{equation}
and
\begin{equation}\label{You2}
U_2(x_1,x_2,x_3)= x_2(x_3 -x_1) + x_1^2 + x_2^2 + x_3^2 -x_1x_2x_3.
\end{equation}
 See Section \ref{param} for an analysis of these forms. The only perfect $F$'s that we are aware of are of the form \eqref{cform}.

On the other hand, our treatment of the fullness of $M$ applies more generally. We leave the precise details and proofs of the following comments to a forthcoming paper. If $F_0$ is reducible in $\mathbb{Q}[x_1,x_2,x_3]$, then $F$ is full. In this case $F_0$ has a linear factor, which is the condition that $F$ has $\mathcal{h}$-invariant (\cite{D-L})  equal to 1 (see Appendix \ref{invariants} for a discussion of these arithmetic invariants of $F$). The linear factor yields a rational plane in $F_{0}(x_1,x_2,x_3)=0$ which can be used as the small variable and  to generate a family of planes and of  binary quadratic forms and a tentacled region. If $F_0$ is irreducible in $\mathbb{Q}[x_1,x_2,x_3]$ then our moving plane  method fails. Nevertheless one can still create tentacled regions $\mathcal{R}$ in $\mathbb{R}^3$ using neighborhoods at infinity of the  curve $F_0({\bf x})=0$ in $\mathbb{P}^2(\mathbb{R})$. As before, on average over $k$ with $\vert k\vert \leq K$, the number of points $r_{\mathcal{R}}(k)$ in $V_{k,F}(\mathbb{Z})\cap\mathbb{R}$ grows slowly with $k$. The study of the variance of $r_{\mathcal{R}}(k)$ from its expected number (i.e. a product of local densities) reduces to counting points on the hypersurface $F({\bf x})-F({\bf y})=0$ with $({\bf x},{\bf y})\in\mathcal{R}\times\mathcal{R}$. While this is well beyond the available tools from the circle method, a natural hypothesis in this context along the lines of ([Hoo16]) would lead to $F$ being full. In particular, this applies to $F=S$ in \eqref{Intro1}. The much stronger suggestion that $S$ is perfect (\cite{HB92})\footnote{Very recently, $V_{k,S}(\mathbb{Z})$ for $k=33$ and $42$ were shown to be nonempty, completing the list of such for $1\leq k\leq 100$ (see Booker \cite{Booker19} and Booker-Sutherland \cite{wiki19}).}, which was mentioned at the start is a fascinating one, as is the question of the existence of any  perfect homogeneous $F$. All $k$'s are admissible for the homogeneous form $x_1^3 + x_2^3 + 2x_3^3$ and it is a candidate for being both perfect and universal. It is interesting to note that this form is universal when considered over the $S$-integer ring $\mathbb{Z}[\frac{1}{6}]$, and has infinitely many solutions for each $k$.

We point out  that  the analogous problem for quadratic polynomials in  two variables is very different in that $f$ is never absolutely  irreducible, and indeed the typical such $f$ is never full.

Finally, we note that the $V_{k,F}$'s for $F=M$ are the relative character varieties for the representations of $\pi_{1}(\Sigma_{1,1})$ into $SL_2$ (here $\Sigma_{g,n}$ is a surface of genus $g$ and $n$ punctures) and the group $\Gamma$ is essentially the mapping class group action on the $V_{k,M}$'s (see Goldman \cite{Gol03}). As such, many of the questions that we address in  this simplest case make sense with $\Sigma_{1,1}$ replaced by $\Sigma_{g,n}$ (see Whang \cite{Whang}). In particular it is shown there that the key feature that the integral points for these varieties consist of finitely many $\Gamma$-orbits, persists. However both for $\Sigma_{1,1}$ and in this more general setting, this finiteness fails when the integers are replaced by $S$-integers in a general number ring. This makes for a quite different picture and analysis to which we will return in a future work.

\vskip 0.2in
\noindent{\small {\bf Acknowledgements.}}\\ 
We thank  V. Blomer, E. Bombieri, J. Bourgain, T. Browning, C. McMullen, P. Whang and U. Zannier for insightful discussions.\\
AG thanks the Institute for Advanced Study and Princeton University for making possible visits during part of the years 2015-2017  when much of this work took place. He also acknowledges support from the IAS, the Simons Foundation and his home department. He dedicates this article to his family Priscilla, Armand and Saskia.\\
PS was supported by NSF grant DMS 1302952.\\
 The softwares Eureqa$^\text{\textregistered}$ and  Mathematica$^{\copyright}$ were used on a PC running Linux to generate some of the data. Additional computations were done at the OSU-HPCC at Oklahoma State University, which is supported in part through the NSF grant OCI-1126330.
\vskip 0.2in
\begin{notation}
For the remainder of the paper we suppress the reference to the Markoff equation. So for example $V_k$ would mean $V_{k,M}$. We also use $(\frac{*}{p})_L$ to denote the Legendre symbol $(\frac{*}{p})$ to avoid any confusion with fractions.
\end{notation} 
%%%%%%%%%%%%%%%%%%%%%%%%%%%%%%%%%%%%%%%%%%%%%%%%%%%%%%%%%%%%%%%%%%%%%%%%%%%
%%%%%%%%%%%%%%%%%%%%%%%%%%%%%%%%%%%%%%%%%%%%%%%%%%%%%%%%%%%%%%%%%%%%%%%%%%%
%%%%%%%%%%%%%%%%%%%%%%%%%%%%%%%%%%%%%%%%%%%%%%%%%%%%%%%%%%%%%%%%%%%%%%%%%%%%%%%%
%%%%%%%%%%%%%%%%%%%%%%%%%%%%%%%%%%%%%%%%%%%%%%%%%%%%%%%%%%%%%%%%%%%%%%%%%%%%%%%%
\section{The descent argument revisited}
\label{sec2}
The descent argument was first considered by Markoff in \cite{Markoff2}, and later extended by Hurwitz \cite{Hur07} and Mordell \cite{Mor53} (see also \cite{Baragar} for a study of fundamental solutions associated with a special case of these several variable hypersurfaces) . In particular, Hurwitz used a ``height'' function given by $h(x_1,x_2,x_3)= |x_1|+|x_2|+|x_3|$, which was then utilized subsequently in the literature. The descent argument led to a finite number of points plus those with minimal height. Our initial analysis is a revisit of this descent argument but without the use of the height function (we later use a new function for a finer analysis).
 
For $k \in \mathbb{Z}$, consider the set of integral points on the Markoff surface
\begin{equation}\label{1a}
V_{k}:\quad\quad x_{1}^2 + x_{2}^2 + x_{3}^2 - x_{1}x_{2}x_{3} =k\, .
\end{equation}

After invariance by permutations and also changing two signs but leaving out Vieta involutions (which we call narrow equivalence), we see that (i) if  $k< 0$, we may consider only solutions $0\leq x_1 \leq x_2 \leq x_3$, and (ii) if $k\geq 0$, there are two types of solutions namely those with all variables non-negative and so $0\leq x_1\leq x_2\leq x_3$; and those in the compact  set $\mathfrak{S}^{+}(k)$ with exactly one negative variable and two positive.

For $k<0$ we note that $x = 0$, $1$ or $2$ are not possible (since they give $k= x_{2}^2 + x_{3}^2$, $4(k-1) = (2x_{2}- x_{3})^2 + 3x_{3}^2$ and $(x_2 - x_3)^2 = k-4$ respectively) so that we assume $3\leq x_1\leq x_2\leq x_3$ in this case. 

When $k \geq 0$, $x=0$ and $x=1$ give at most finitely many triples $(x_{1},x_{2},x_{3})$. and we denote this set by $\mathfrak{T}(k)$.  Thus in this case, $(x_{1},x_{2},x_{3})$ is a solution implies it is equivalent (narrowly) to one in $\mathfrak{S}^{+}(k) \cup \mathfrak{T}(k)$ or it satisfies $2\leq x_1 \leq x_2 \leq x_3$.

We now consider the Vieta involution acting on $(x_{1},x_{2},x_{3}) \rightarrow (x_{1},x_{2},x_{1}x_{2}- x_3)$. If $x_{1}x_{2}- x_3<0$, so that  $ k\geq 0$, then $(x_{1},x_{2},x_{3})$ is equivalent to a solution in $\mathfrak{S}^{+}(k)$. Next suppose $x_{1}x_{2}- x_3 \geq x_3$, so that $2x_3 \leq x_{1}x_{2}$. Solving for $x_3$ in \eqref{1a} gives $2x_3 = x_{1}x_{2} \pm \sigma$ where $\sigma = \sqrt{x_{1}^2 x_{2}^2 -4(x_{1}^2 + x_{2}^2 -k)}$, so that necessarily $\sigma = x_{1}x_{2} - 2x_3 \leq (x_1-2)x_2$. Squaring and simplifying gives $(x_1 -2)x_{2}^2 \leq (x_{1}^2 -k)$. 

If $x_1 \geq 3$ and $k>0$, we conclude that $x_{2}^2 < x_{1}^2$, a contradiction. If $x_1 =2$, we conclude that $k\leq 4$. Thus we derive a contradiction for all $k>4$, so that in this case we have $0\leq x_{1}x_{2}- x_3 < x_3$. But more is true, namely $0 \leq x_{1}x_{2} - x_3 < x_2$  shown as follows: if $x_2\leq x_{1}x_{2}- x_3 < x_3$, then $x_{1}x_{2} < 2x_3 \leq 2(x_1 -1)x_2$, so that necessarily $2x_3 = x_{1}x_{2} + \sigma$. Then $\sigma \leq (x_1 -2)x_2$ and the argument above gives a contradiction.   Hence we have

\begin{lemma}
\label{lem2.1}
If $k>4$ and if $(x_{1},x_{2},x_{3})$ is a lattice  point on  $V_k$ in \eqref{1a}, it is equivalent to one in the compact set $\mathfrak{S}^{+}(k) \cup \mathfrak{T}(k)$ where
\[
\mathfrak{S}^{+}(k) = \left\{(-x_{1},x_{2},x_{3}) : \ 3\leq x_1 \leq x_2 \leq x_3\,;\ x_1^2 + x_2^2 +x_3^2 + x_1x_2x_3 = k\right\}\cap \mathbb{Z}^3\, ,
\]
 or if not then it is equivalent to $(x_{1},x_{1}x_{2}-x_{3},x_2)$, with $3 \leq x_{1}x_{2}- x_3 < x_2 \leq x_3$ and $x_1\geq 3$.
\end{lemma}
The special cases $1\leq k\leq 3$ are settled as follows: (i) there are no solutions when $k=3$ since there are none modulo 4; (ii) for $k=2$, we can use the descent argument  above and conclude that we need only look for solutions to $x^2 + y^2 + z^2 +xyz=2$ with all variables non-negative or we solve the Markoff equation with $x_1 \in \{ 0, \pm 1, \pm 2\}$, giving us the point $(0,1,1)$ and its infinite orbit under $\Gamma$; and for $k=1$, the same analysis results in the point $(0,0,1)$, for which there is only a finite orbit under $\Gamma$.  The cases $k=0$ and $4$ we consider in the next sections (they correspond to the original Markoff surface in Sec. \ref{subsec3a} and the singular Cayley surface in Sec. \ref{subsec4b}).

For $k <0$, the estimate $(x_1 -2)x_{2}^2 \leq (x_{1}^2 -k)$ given above is still valid when we assume $x_{1}x_{2}- x_3 \geq x_3$, with $3\leq x_1\leq x_2\leq x_3$. Then, if $x_1\geq 4$, it follows that $2x_{2}^2 \leq x_{1}^2 + |k|$, which then implies $x_2\leq \sqrt{|k|}$, so that $x_{3} \leq \frac{x_{1}x_{2}}{2} \leq \frac{|k|}{2}$. If $x_1 =3$, then clearly $x_2 \leq \sqrt{9+|k|}$, and so $x_3 \leq \frac{3}{2}\sqrt{9+|k|}$. The same argument shows that for large values of $|k|$, $x_1 \ll |k|^{\frac{1}{3}}$, $x_2 \ll \sqrt{\frac{|k|}{x_1}}$ and $x_3 \ll \sqrt{|k|x_1}$. Next, supposing $x_2 \leq x_{1}x_{2}- x_3 < x_3$, we see that the point $(x_1,x_2,x_3)$ is $\Gamma$-equivalent to $(y_1,y_2,y_3)= (x_1,x_2,x_1x_2-x_3)$, where now $y_1y_2 -y_3 \geq y_3$, the same inequality considered above.   Thus we have

\begin{lemma}
\label{lem2.2}
For $k<0$, if $(x_{1},x_{2},x_{3})$ is a lattice point on $V_k$ in  \eqref{1a}, it is then equivalent to one in the compact set $\mathfrak{S}^{-}(k)\subset \mathfrak{U}(k)$, where
\[
\mathfrak{S}^{-}(k):=\left\{(x_{1},x_{2},x_{3}): 3 \leq x_1  \leq x_2 \leq x_3 \leq \frac{1}{2}x_1x_2\right\}\cap V_{k}(\mathbb{Z})\, ,
\]
and
\[
\mathfrak{U}(k):=\left\{(x_{1},x_{2},x_{3}): 3 \leq x_1  \leq x_2 \leq \sqrt{|k|+9}; \  3\leq x_3 \leq \frac{3}{2}(|k|+9)\right\}\, ,
\] or if not  it is equivalent to $(x_{1},x_{1}x_{2}-x_{3}, x_2)$ with $3 \leq x_{1}x_{2}- x_{3} < x_2 \leq x_3$ and $x_1 \geq 3$.
\end{lemma}

The lemmas above form the basis of the descent argument with repeated application of the Vieta involution so that ultimately any integral solution is equivalent to one in a corresponding finite set.

%%%%%%%%%%%%%%%%%%%%%%%%%%%%%%%%%%%%%%%%%%%%%%%%%%%%%%%%%%%%%%%%%%%%%%%%%%%%%%%%%%%%%%%%%%%%%%%%%%%%%%%%%%
%%%%%%%%%%%%%%%%%%%%%%%%%%%%%%%%%%%%%%%%%%%%%%%%%%%%%%%%%%%%%%%%%%%%%%%%%%%%%%%%%%%%%%%%%%%%%%%%%%%%%%%%%
%%%%%%%%%%%%%%%%%%%%%%%%%%%%%%%%%%%%%%%%%%%%%%%%%%%%%%%%%%%%%%%%%%%%%%%%%%%%%%%%%%%%%%%%%%%%%%%%%%%%%%%%%%%
\section{Bhargava cubes and Markoff}
\label{sec3}
To construct the fundamental sets in the next section, we utilize a function $\Delta({\bf x})$ given in \eqref{1b}, that  proves useful in tracking the images of points under the action of the group $\Gamma$. While we could define $\Delta$ without comment, we give here our original construction using Bhargava cubes.

Let $x_1$, $x_2$ and $x_3$ be arbitrary integers and consider the Bhargava cube (\cite{Bha04}) as shown in Fig. \ref{img1}.

\begin{figure}[!ht]
\centering
  \includegraphics[width=.5\linewidth]{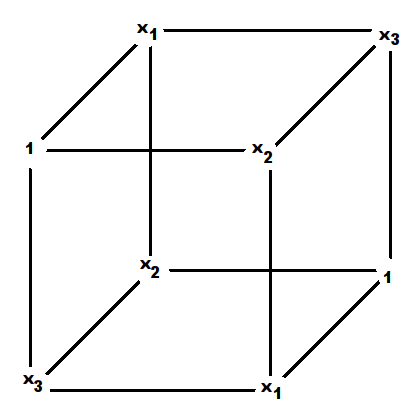}
\caption{Bhargava cube associated with $M$.}
  \label{img1}
\end{figure}

The Bhargava slicings give rise to the three matrix pairs:
\[
M_{1} = \left(\begin{array}{cc}
1 & x_{2} \\ x_{3} & x_{1} \end{array}\right) \ , \quad N_{1} = \left(\begin{array}{cc}
x_{1} & x_{3} \\ x_{2} & 1 \end{array}\right) \ ;
\]
\[
M_{2} = \left(\begin{array}{cc}
1 & x_{3} \\ x_{1} & x_{2}\end{array}\right) \ , \quad N_{2} = \left(\begin{array}{cc}
x_{2} & x_{1} \\ x_{3} & 1 \end{array}\right) \ ;
\]
\[
M_{3} = \left(\begin{array}{cc}
1 & x_{1} \\ x_{2} & x_{3} \end{array}\right) \ , \quad N_{3} = \left(\begin{array}{cc}
x_{3} & x_{1} \\ x_{2} & 1 \end{array}\right) . 
\]
These in turn give the following three quadratic forms $Q_{i}(u,v)$, where
\begin{align*}
Q_{1} &= (x_2x_3 - x_1)u^2 + (1 + x_{1}^{2} - x_{2}^{2} - x_{3}^{2})uv + (x_{2}x_{3} -x_1)v^2\, ,\\
Q_{2} &= (x_{1}x_{3} - x_2)u^2 + (1 + x_{2}^{2} - x_{1}^{2} - x_{3}^{2})uv + (x_{1}x_{3} - x_2)v^2\, ,\\
Q_{3} &= (x_{1}x_{2} -x_3)u^2 + (1 + x_{3}^{2} - x_{1}^{2} -x_{2}^{2})uv + (x_1x_2 -x_3)v^2\, .
\end{align*}
All three quadratic forms have the same discriminant $\Delta = \Delta(x_1,x_2,x_3)$ which also factorizes to give  
\begin{align}\label{1b}
\begin{split}
\Delta &= (1 + x_{2}^{2} - x_{1}^{2} -x_{3}^{2})^2 -4(x_1 x_3 - x_2)^2\, ,\\
&=(1+x_1 + x_2 + x_3)(1+x_2 -x_1 - x_3)(1+x_3 -x_1 - x_2)(1+x_1 -x_2 -x_3)\, .
\end{split}
\end{align}
Note that 
\begin{enumerate}[wide,labelindent=0pt,label=(\alph*).]
\item $\Delta \equiv 0$ or $1\, ({\rm mod}\, 4)$ depending on if $x_{1}^{2} + x_{2}^{2} + x_{3}^{2}$ is odd or even respectively.
\item $\Delta$ is invariant under permutations.
\item $\Delta$ is invariant if one variable is fixed and the sign is changed on the other two variables.
\item If $2\leq x_1\leq x_2 \leq x_3$, then $\Delta <0$ if and only if $x_2 \leq x_3 \leq x_1 + x_2 -2$.
\end{enumerate}

\subsection{The case $k=0$}
\label{subsec3a}\ 

Recall (Markoff \cite{Markoff2}) that the solution set has two orbits with fundamental roots $(0,0,0)$ and $(3,3,3)$. We have $\Delta(0,0,0)= 1$ and $\Delta(3,3,3)= -80$. We show here that 
\begin{equation}\label{eq3.1}
\Delta(x_1,x_2,x_3)<0 \ \ \text{if and only if}\  \ (x_1,x_2,x_3)=(3,3,3).
\end{equation}
 Thus, the two orbits each have a minimal value for $\Delta$, taken at the associated fundamental roots. In other words, there are two components of $V_{0}(\mathbb{Z})$ and in each component $\Delta$ has a minimum value, taken at a unique point, which can then be used as a generator for that component. This phenomenon repeats itself when $k\geq 5$ below.

We prove \eqref{eq3.1} as follows: since $x_{1}^{2} + x_{2}^{2} + x_{3}^{2} = x_{1}x_{2}x_{3}$, it follows that $x_1$, $x_2$ and $x_3$ are all positive or exactly two are negative (we avoid the trivial solution here). By the properties of $\Delta$ itemized above, we may assume that $1\leq x_1 \leq x_2 \leq x_3$. The Markoff equation is equivalent to the equation $(x_{1}^{2} -4)(x_{2}^{2} -4)-16 = (2x_3 -x_{1}x_{2})^2$, from which it follows that $3\leq x_1 \leq x_2\leq x_3$, which we assume. Suppose $\Delta(x_1,x_2,x_3)<0$, so that $x_2 \leq x_3 \leq x_1 + x_2 -2 <2x_2$. Solving for $x_3$ in the Markoff equation gives us $2x_3 =x_{1}x_{2} \pm \sigma$, where $\sigma = \sqrt{(x_{1}x_{2})^2 -4(x_{1}^{2} +x_{2}^{2})}\geq 1$. 
 
 If $x_1\geq 4$ we must discard the positive sign since $x_3 <2x_2$. So  in this case, $  x_{1}x_{2} - \sigma = 2x_3 \geq 2x_2$, from which, by expanding and simplifying, one gets $4x_{2}^{2} \leq x_1 x_{2}^{2} \leq x_{1}^{2} + 2x_{2}^{2} \leq 3x_{2}^{2}$, a contradiction.

For  $x_{1}=3$, we have $x_2\leq x_{3} \leq x_1 +x_2 -2 =x_2 +1$, so that $x_3 = x_2$ or $x_3 =x_2 +1$. If $x_3 = x_2$, we have $9+2x_{2}^{2} -3x_{2}^{2}=0$ so that $x_1=x_2=x_3=3$. Finally if $x_3 =x_2 +1$, we must have $9 + x_{2}^{2} + (x_2+1)^2 = 3(x_{2}^{2} + x_2)$, which is impossible.

%%%%%%%%%%%%%%%%%%%%%%%%%%%%%%%%%%%%%%%%%%%%%%%%%%%%%%%%%%%%%%%%%%%%%%
%%%%%%%%%%%%%%%%%%%%%%%%%%%%%%%%%%%%%%%%%%%%%%%%%%%%%%%%%%%%%%%%%%
%%%%%%%%%%%%%%%%%%  SECTION %%%%%%%%%%%%%%%%%%%%%%%%%%%%%%%%%%%%%%%%%%%
\section{Fundamental sets and Theorem \ref{Thm1}}
\label{sec4} 
The descent arguments of Markoff, Hurwitz and Mordell show that there is a finite set of lattice points from which all lattice points of the Markoff surface \eqref{1a} can be obtained as images under $\Gamma$. This section provides a proof of Theorem \ref{Thm1} by showing the inequivalence of the points in the finite set.
%%%%%%%%%%%%%%%%%%%%%%%%%%%%%%%%%%%%%%%%%%%%%%%%%%%%%%%%%%%%%%%%%%%%%%%%%%%%%%%%%%%%%%%%%%%%%%%%%
\subsection{The case $k\geq 5$}
\label{subsec4a}\ 

Recall from Sec.\,\ref{sec2} that if $k \geq 5$, any solution ${\bf x}=(x_1,x_2,x_3)$ to the Markoff equation \eqref{1a} is equivalent to one in a compact \textit{reduced set} (by Lemma \ref{lem2.1} and descent). We order the coordinates first such that $0\leq |x_1|\leq|x_2|\leq|x_3|$. 

In the next section, we  show that the Markoff equation has no solutions for those $k$'s (positive or negative) satisfying any of the following congruences: $k \equiv 3 \, ({\rm mod}\, 4)$ and $k \equiv \pm 3 \, ({\rm mod}\, 9)$, these then accounting for $\frac{5}{12}K +O(1)$ members in the interval $5\leq k \leq K$, and we call them {\it non-admissible}; the non-admissible $k$'s have local obstructions. The remaining $k$'s we call {\it admissible}, and there are $\mathcalorig{A}(K) = \frac{7}{12}K + O(1)$ of them.

We say that $k$ is  {\it exceptional}\footnote{The removal of the points $\bf{x}$ with one of its coordinates in $\{-2,-1,0,1,2\}$ corresponds to avoiding the region at infinity on which $\Gamma$ acts ergodically (when $k>20$) in \cite{Gol03}, and to the notion of ``small'' in \cite{Auroux} Sec.\,5.} if there is a solution to \eqref{1a} with $|x_j|=  0, 1$ or $2$; these $k$'s satisfy at least one of the equations (i) $u^2 + v^2 = k$ ,  (ii) $u^2 + 3v^2 = 4(k-1)$, or (iii) $u^2 = k-4$. Consequently, for $k$'s in an interval of length $K$, they account for at most $O\left(K(\log{K})^{-\frac{1}{2}}\right)$ members, and we will ignore them in what follows.  The remaining $\frac{7}{12}K + O\left(K(\log{K})^{-\frac{1}{2}}\right)$ numbers $k$ in the interval $5\leq k\leq K$ we shall call {\it generic}.

It follows from Sec.\,\ref{sec2} that every solution ${\bf x}$ associated to a generic $k$ is equivalent to one in the set $\mathfrak{S}^{+}(k)$ given in Lemma \ref{lem2.1}.  We now show that the elements in this set, when non-empty,  are inequivalent under $\Gamma$, so that $\mathfrak{S}^{+}(k)$ is a {\it fundamental set}. 

We will use the $\Delta$-function given in \eqref{1b} to form an ordering on the tree of solutions to the Markoff equation. Given any ${\bf x}=(x_1,x_2,x_3)$, the three Vieta maps are $\mathcal{V}_1: (x_1,x_2,x_3) \mapsto (x_2x_3-x_1,x_2,x_3)$, $\mathcal{V}_2: (x_1,x_2,x_3) \mapsto (x_1,x_1x_3-x_1,x_3)$ and $\mathcal{V}_3: (x_1,x_2,x_3) \mapsto (x_1,x_2,x_1x_2-x_3)$. Recall that the group $\Gamma$ is generated by permutations, double sign-changes and the Vieta maps. The $\Delta$-function is invariant under the first two motions and we denote $\Delta_{i} = \Delta\circ \mathcal{V}_i$.  Then, it is easy to check that when ${\bf x}$ is a solution of the Markoff equation, one has

\begin{equation}\label{gauge}
\begin{aligned}
 &\Delta_1({\bf x})-\Delta({\bf x}) = x_2x_3(x_2x_3 -2x_1)\left[2(k-5)+(x_2^2-4)(x_3^2-4)\right],\\
 &\Delta_2({\bf x})-\Delta({\bf x}) = x_1x_3(x_1x_3 -2x_2)\left[2(k-5)+(x_1^2-4)(x_3^2-4)\right],\\
 &\Delta_3({\bf x})-\Delta({\bf x}) = x_1x_2(x_1x_2 -2x_3)\left[2(k-5)+(x_1^2-4)(x_2^2-4)\right].
\end{aligned}
\end{equation}
The expressions in the square brackets in all three formulae above  are strictly positive when $k$ is generic and if ${\bf x}$ is any solution of the corresponding Markoff equation.

We set up the tree associated with solutions as follows: each solution ${\bf x} =(x_1,x_2,x_3)$ will be a vertex and neighboring vertices are edge connected if they are obtained from ${\bf x}$ by one of the three Vieta maps. As such, we identify coordinates if they are obtained by permutations or double sign changes (noting that $\Delta$ is unchanged under them). By this latter identification, the coordinates are one of two types, namely all positive or exactly one negative. It is then possible to rearrange them into the following canonical forms: $(x_1,x_2,x_3)$ or $(-x_1,x_2,x_3)$ with $3\leq x_1\leq x_2\leq x_3$. We call the former {\it positive nodes} and the latter {\it negative nodes}. By Lemma \ref{lem2.1}, for $k\geq 5$, every positive node is equivalent to a negative node (or otherwise, by descent is equivalent to the node $(3,3,3)$ which corresponds to $k=0$).

We look at the action of the Vieta maps on a positive node. It is clear that $x_2x_3 -2x_1$ and $x_1x_3 -2x_2$ are strictly positive so that $\Delta_1({\bf x}) >\Delta({\bf x})$ and $\Delta_2({\bf x}) >\Delta({\bf x})$. Moreover, the nodes $\mathcal{V}_1({\bf x})$ and $\mathcal{V}_{2}({\bf x})$ are both positive, Next, the argument showing descent in Sec.\,\ref{sec2} shows that $x_1x_2 - 2x_3 \geq 0$ is impossible so that $\Delta_3({\bf x}) <\Delta({\bf x})$. Here  $\mathcal{V}_{3}({\bf x})$ may be either positive or negative. We represent these observations by the images Fig.6a and Fig.6b , where   square nodes are positive nodes,  disc nodes  are negative nodes, dark nodes are the Vieta images while the original point is a light node (the vertical ordering of the nodes is determined by  the signs of the $\Delta$-differences from \eqref{gauge}).

Next, if we begin with a negative node (so that one replaces $x_1$ with $-x_1$ in the formulae above, it is obvious that $\Delta_i({\bf x}) >\Delta({\bf x})$ for all $i$ and (after a double sign change and reordering) that the $\mathcal{V}_{i}({\bf x})$ are all positive. This is represented by Fig. 6c .

It follows now that the tree decomposes into components and each component has a root that is a negative node. Moreover, the negative node occupies the lowest point on the tree, with all other nodes in that component being positive (in other words, $\Delta$ has a minimum on each component and that minimum is determined by a negative node). Thus the negative nodes form a fundamental set, giving us the first  case of Theorem \ref{Thm1}.

\begin{figure}[!ht]\label{nodes}
\centering
  \includegraphics[width=.35\linewidth]{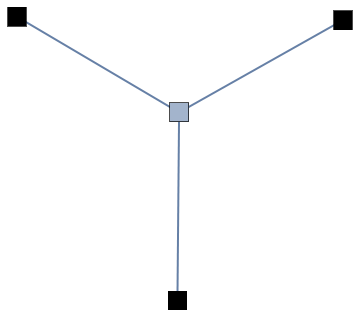}
  \hspace{60pt}
  \vspace{1cm}
  \includegraphics[width=.35\linewidth]{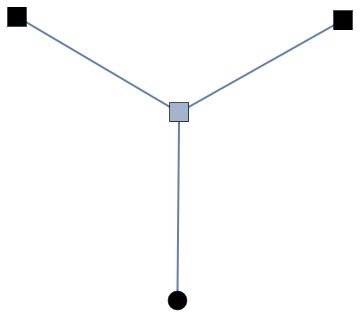}
  \includegraphics[width=.35\linewidth]{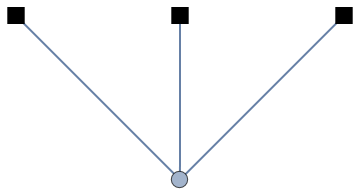}
\caption{Node blocks; top:\ (a), (b); bottom:\ (c).}
\end{figure}
%%%%%%%%%%%%%%%%%%%%%%%%%%%%%%%%%%%%%%%%%%%%%%%%%%%%%%%%%%%%%%%%%%%%%%%%%%%%%%%%%%%%%%%%%%%%%%%
\subsection{The case $k<0$}
\label{subsec4aa}\ 

From Sec.\,\ref{sec2} and Lemma \ref{lem2.2} every lattice point in $V_k$ is equivalent to one in $\mathfrak{S}^{-}(k)$. We show that the points in this set are inequivalent. First using $(x_{1}^2 -4)(x_{2}^2 -4) = (2x_3 -x_1x_2)^2 -4(k-4)$ in \eqref{gauge} and the similar formulae with the variables permuted, we see that the three terms in square brackets in \eqref{gauge} are all positive. Thus the signs of the differences of the $\Delta$-functions in \eqref{gauge} are  determined by the three terms $x_2x_3 -2x_1$, $x_1x_3 -2x_2$ and $x_1x_2 -2x_3$. The first two are obviously positive, and one sees that the last is non-negative if and only if $(x_1,x_2,x_3)\in \mathfrak{S}^{-}(k)$ Thus, in the tree determined by these points one sees that we have nodes of the type shown in Fig.6(c) with two or three black square vertices emanating from points in $\mathfrak{S}^{-}(k)$, while for points in the complementary set, we have nodes of the type in Fig.6(a). It follows that the points in $\mathfrak{S}^{-}(k)$ can serve as the roots of the components of the tree, from which the second case of Theorem \ref{Thm1} follows.
%%%%%%%%%%%%%%%%%%%%%%%%%%%%%%%%%%%%%%%%%%%%%%%%%%%%%%%%%%%%%%%%%%%%%%%%%%%%%%%%%%%%%%%%%%%%%%%%
\subsection{The Cayley surface $k=4$}
\label{subsec4b}\ 

Most of the argument above for $k\geq 5$ can be applied to the case $k=4$, and we indicate the necessary modifications.
First, we consider solutions of the type $(-x_1,x_2,x_3)$ with $x_1,\,x_2,\, x_3 \geq 0$ satisfying $x_{1}^{2} +x_{2}^{2}+x_{3}^{2} +
x_1 x_2 x_3 = 4$. It is obvious that there are only two solutions up to equivalence, namely $(-2,0,0)\sim (0,0,2)$ and $(-1,1,1)\sim (1,1,2)$. Hence we need only consider solutions of the type $(x_1,x_2,x_3)$ with $0\leq x_1 \leq\ x_2 \leq  x_3 $. If $x_1 =0$, the only solution is $(0,0,2)$ while if $x_1 =1$, then the only choice is $(1,1,2)$. Then by the descent argument in Sec.\,\ref{sec2}, if $x_1 \geq 3$, the solution $(x_1,x_2,x_3)$ is equivalent to one with one of the coordinates equal to 2. It is trivial that the only solutions of this kind are one of the type $(2,a ,a)$, with $a \geq 0 $ integers. It suffices now to check the equivalence of these solutions. It is easily checked that the orbits of  $(2,0,0)$, $(2,1,1)$ and $(2,2,2)$ contain no other points of the type $(2,a ,a)$ except themselves, so that we assume $a \geq 3$.

Following the three formulas in \eqref{gauge}, if ${\bf x}=(2,a ,a)$, then two of the Vieta transformations  keep it inert while the third creates a node above it, this new node not being of the same type (we say ``above'' to mean $\Delta_i({\bf x})> \Delta({\bf x})$). Also following the argument used for $k\geq 5$, if ${\bf x}=(x_1,x_2,x_3)$, with $x_i \geq 3$, then two Vieta transformations create nodes above it while a third creates a node below it. It is then easily seen that a tree containing a node of the type  $(2,a ,a)$ cannot contain a different node of the same type. Hence we have
 \begin{proposition}\label{cayley}
 The Cayley surface $V_{4}(\mathbb{Z})$ has infinitely many inequivalent orbits, each determined by a solution of the type $(2,a,a)$, with $a \geq 0$.
 \end{proposition}
 One checks that the $(2,0,0)\sim (-2,0,0)$-component has only 1 element (upto permutation and double sign-change) and so the minimal $\Delta$-value is $\Delta(-2,0,0)= 9$. Next, the $(2,1,1)$-component has only 2 elements namely $(2,1,1)$ and $(-1,1,1)$. The minimal $\Delta$-value is $\Delta(-1,1,1)= -16$ while $\Delta(2,1,1)=5$. Finally, $\Delta(2,a,a)=9 -4a ^2 <0$ for $a \geq 2$. Then the same argument used in Sec.\,\ref{subsec3a} can be used to show that  any lattice point $\bf{x}$ not of these type satisfy $\Delta({\bf x})\geq 0$, so that the minimal $\Delta$-value is uniquely determined. Thus, even here each component has a unique minimal $\Delta$-value, whose point can be used as a generator.
 
 \begin{remark}
One can use the $\Delta$-function and the analysis above to deduce a descent procedure. One concludes that either every positive node descends to a negative node or if not, there is an infinite chain of positive nodes on which $\Delta({\bf x})$ is strictly decreasing. The latter is not possible since $\Delta({\bf x}) \geq 0$ on positive nodes. There are only finitely many negative nodes in  $\mathfrak{S}^{+}(k)$. So we conclude that there are finitely many orbits. Repeating the analysis in the paper also shows that all the negative points are $\Gamma$-inequivalent \underline{and} in each orbit $\Delta$ has a minimum value taken at the root of that orbit, so at the only (modulo double sign-changes) negative point on that orbit.

Using Lagrange multipliers on the region on $V_k$ with $x_j \geq 3$ and $k\to \infty$, one can show that :
 \begin{enumerate}
 \item $$\min_{{\bf x} \in \mathfrak{S}^{+}(k)} \Delta({\bf x}) \ \geq\  k^2 +18 k^{\frac{3}{2}} + 88 k  + \frac{621}{4} k^{\frac{1}{2}} + O(1);$$
 \item $$\max_{{\bf x} \in \mathfrak{S}^{-}(k)} \Delta({\bf x}) \ \leq\  k^2 -18 k^{\frac{3}{2}} + 88 k  - \frac{45}{4} k^{\frac{1}{2}} + O(1);$$
 \item $$\min_{{\bf x} \in \mathfrak{S}^{-}(k)} \Delta({\bf x}) \ \geq\  -3k^{\frac{4}{3}} + O(k).$$
 \end{enumerate}
  
Hence asymptotically, $\Delta$ behaves like a Minkowski gauge-function, with "successive minima" taken at the root of the orbits; that is if $\mathfrak{h}(k)$ is the number of orbits, the first $\mathfrak{h}(k)$ minimal values (counted with multiplicity) of $\Delta$ on the lattice points on $V_k$ occur at the negative points.

\end{remark}
%%%%%%%%%%%%%%%%%%%%%%%%%%%%%%%%%%%%%%%%%%%%%%%%%%%%%%%%%%%%%%%%%%%%%%%%
%%%%%%%%%%%%%%%%%%%%%%%%%%%%%%%%%%%%%%%%%%%%%%%%%%%%%%%%%%%%%%%%%%%%%%
%%%%%%%%%%%%%%%%%%%%%%%%%%%%%%%%%%%%%%%%%%%%%%%%%%%%%%%%%%%%%%%%%%%%%%

\section{Parametric solutions on Markoff-type surfaces and Zariski density. }\label{param}
We show in this section that for generic $k$, the Markoff surface has no parametric integral points and that the solution set is Zariski dense. We also consider the  surfaces given by  $U_1$ and $U_2$ mentioned in the Introduction.
\subsection{Parametric solutions}
\begin{lemma} For any $k\in \mathbb{Z}$, let $M^*_k$ be the  surface given by $M^*({\bf x})=k$, where 
\begin{equation}\label{para1}
M^*({\bf x})=\sum_{j=1}^{3}\alpha_j x_j + \left(\beta_1x_2x_3 + \beta_2x_1x_3 + \beta_3x_1x_2\right) + x_1^2 + x_2^2 + \ve x_3^2 -x_1x_2x_3 ,
\end{equation}
where $\ve = \pm 1$ and $\alpha_j$, $\beta_j \in \mathbb{Z}$ for all $j$. Suppose there are polynomials $P_{j}(t)\in \mathbb{Z}[t]$ each with non-zero degree, such that $M^*(P_1, P_2, P_3) =  k$ identically in $t$. Then there are polynomials $Q_1, Q_2 \in \mathbb{Z}[t]$ of non-zero degree and a constant $q\in \mathbb{Z}$ such that $M^*(q, Q_1, Q_2) =  k$ identically in $t$.
\end{lemma}
\begin{proof}
Let $P_j$ have degree $d_j \neq 0$ for $j=1, 2, 3$ as above. By comparing degrees in \eqref{para1} we cannot have $d_1=d_2=d_3$, so that there is either a unique $d_j$ exceeding the other two or exactly two of the degrees are the same. The latter does not happen as it implies  that at least one of the polynomials is a constant. Hence (comparing degrees in \eqref{para1}) we have that $d'' = d' + d$ for some choice of the degrees. It will not matter which subscript represents the largest degree in what follows, so that we put $d_3 = d_1+d_2$, with $d_1$, $d_2 \geq 1$.

There is a Vieta affine transformation acting on the surface given by $x_3 \mapsto x_4 =x_1x_2 - \alpha_3 - \beta_1x_2 -\beta_2x_1 -\ve x_3$, so that if $P_4(t)$ is the polynomial determined by $x_4$, we have
\[ 
P_3P_4= k - P_1^2 - P_2^2 - \beta_3P_1P_2 -\alpha_1P_1 -\alpha_2P_2,
\]
identically in $t$. If $d_4$ is the degree of $P_4$, we have $d_3 + d_4 \leq 2\max(d_1,d_2)$, so that $d_4 \leq \max(d_1,d_2) - \min(d_1,d_2) < \max(d_1,d_2)$. Thus we have polynomials $P_1, P_2, P_4$ in place of $P_1, P_2, P_3$ representing integral points on the surface, with the maximal degree reduced by at least one and the new maximum degree is determined by $P_1$ or $P_2$. Either $P_4$ has degree zero, in which case we are done, or if not, all the new polynomials have non-zero degree. Repeating this descent argument (with a different Vieta transformation) shows that there must be parametric solutions with at least one polynomial constant, and the other two polynomials of non-zero degree.
\end{proof}

 It is not possible to have parametric solutions to \eqref{para1} with two of the polynomials constant. It follows from the lemma that if parametric solutions exist then there exists $q\in \mathbb{Z}$ and $Q_1, Q_2 \in \mathbb{Z}[t]$ of the same degree $d$ satisfying \eqref{para1} (it is possible to show that $d\leq 2$, if it exists). We now consider some special cases:
 
 \begin{enumerate}
 \item For the Markoff equation we have $Q_1^2 + Q_2^2 - qQ_1Q_2 = k-q^2 $. Comparing the highest degree term shows that there are integers $q_1, q_2$ such that $q_1^2 + q_2^2 - qq_1q_2 = 0$. It follows that $q= \pm 2$ and $k-4 = \square$. Moreover if $k=4+w^2$, then one has a parametric family of solutions $q=2$, $Q_1 = t$ and $Q_2=t+w$. In particular, this means that if $k$ is generic, there are no parametric solutions to the associated Markoff level set.
 \item Consider the Markoff-like surface $x_1^2 + x_2^2 -x_3^2 -x_1x_2x_3 =k$.  If we have parametric solutions as above of the type $(Q_1,Q_2,q)$, then the argument is identical to the Markoff case so that we conclude there are no such parametric solutions except when $k+4 = w^2$, in which case we have the parametric family $(t+w,t,2)$. Next, if either $x_1$ or $x_2$ is $q$, we have the equation $Q_1^2 - Q_2^2 - qQ_1Q_2 = k-q^2 $, so like the case above, we have $q_1^2 - q_2^2 - qq_1q_2 = 0 $. We conclude that $q=0$ so that when $k\neq 0$, $Q_1$ and $Q_2$ have degree zero, a contradiction. When $k=0$, we have the parametric family $(Q_1,0, \pm Q_1)$ for any polynomial $Q_1$.
 \begin{remark} This surface has the following features: (i) there are no local obstructions, (ii) for $k= 4^{\alpha}k'$ with $\alpha \geq 0$ and $k'$ odd, it has the integral points $(0,2^{\alpha}\frac{k' +1}{2}, 2^{\alpha}\frac{k' -1}{2})$, (iii) if $k' \neq 1$ or $\alpha \geq 3$, there are infinitely many integral points, and (iv) there are infinitely many Hasse failures (in particular, $k=94$ is a Hasse failure). This latter statement follows from an analysis similar to that in Prop. \ref{HF1}.
 \end{remark}
 
 \item Consider the linear deformation $U_1$ of the Markoff equation considered in \eqref{You1}, namely $x_1+ x_1^2 + x_2^2 + \ve x_3^2 -x_1x_2x_3=k$. For any integer $k$, and $\ve = \pm 1$, we have the parametric family of integral solutions  $\left(-t^2 +k-4\ve, -t^2 + t + k - 4\ve,2\right)$ .
 \item Consider the  quadratic deformation $U_2$ of \eqref{You2}: $x_2x_3 -x_1x_2 + x_1^2 + x_2^2 + x_3^2 -x_1x_2x_3=k$. For any $k$, we have the parametric solutions $(-t^2 + t + k-1, -t^2 +k-1,1)$.
 \end{enumerate}

\subsection{Zariski density}\ 

\subsubsection{}  We  prove \eqref{zariski} for the Markoff surface for $k$ not a square (this ensures that if $V_{k,M}(\mathbb{Z})\neq \emptyset$, then it has a lattice point with at most one coordinate zero).  First note that if $\hat{{\bf x}}= (\hat{x}_1,\hat{x}_2,\hat{x}_3) \in V_k(\mathbb{Z})$ and $\vert\hat{x}_j\vert\geq 2$ for some $j$, then $\vert V_k(\mathbb{Z})\vert=\infty$. To see this, say $\vert\hat{x}_1\vert\geq 2$; then the composition of the Vieta transformation $\mathcal{V}_3$ with the permutation of $x_2$ and $x_3$ yields the transformation $(x_1,x_2,x_3) \mapsto (x_1,x_1x_2-x_3,x_2)$ in $\Gamma$. This preserves the plane $x_1 = \hat{x}_1$ and $V_k$, and it induces the linear action $\left[\begin{array}{cc}
 \hat{x}_{1}& -1 \\ 1 & \ 0 \end{array}\right]$ on this plane. Since $\vert\hat{x}_1\vert\geq 2$, this element in $SL_2(\mathbb{Z})$ is of infinite order, so that its orbit is infinite (since it is not acting on the origin) and its Zariski closure contains the conic section $\{x_1=\hat{x}_1\}\cap V_k$. We now argue as in \cite{CZ}. If $\overline{V_k(\mathbb{Z})}$, the Zariski closure of $V_k(\mathbb{Z})$ is not $V_k$, then it is contained in a finite union of curves in $V_k$. Hence there can be at most finitely many $\hat{x}_1$'s with $(\hat{x}_1,\hat{x}_2,\hat{x}_3) \in V_k(\mathbb{Z})$ with  $\vert\hat{x}_1\vert\geq 2$ (since otherwise $\overline{V_k(\mathbb{Z})}$ contains infinitely many distinct conic sections as above). The same applies to $\hat{x}_2$ and $\hat{x}_3$, giving $\vert V_k(\mathbb{Z})\vert < \infty$. That is we have shown that $\vert V_k(\mathbb{Z})\vert=\infty$ implies that $\overline{V_k(\mathbb{Z})}= V_k$. To complete the proof of \eqref{zariski} note that if $\vert k\vert >20$ and $(\hat{x}_1,\hat{x}_2,\hat{x}_3) \in V_k(\mathbb{Z})$ then for at least one of the $j$'s, $\vert\hat{x}_j\vert\geq 2$ and so  $\vert V_k(\mathbb{Z})\vert \neq \emptyset$ implies  $\overline{V_k(\mathbb{Z})} = V_k$. For the $k$'s with $\vert k\vert \leq20$ we check directly that \eqref{zariski} holds. One can show that when $k =1,\ 9,\ 49$, for example, $V_{k,M}(\mathbb{Z}) \neq \emptyset$ but has only a finite orbit. On the other hand, when $k=k_1^2$ with $k_1$ having an odd prime factor congruent to one modulo 4, then $V_{k,M}(\mathbb{Z})$ has an infinite orbit, and by the argument above, is Zariski dense.
 
\subsubsection{} We next consider the surface $U_1$  discussed above and in \eqref{You1}. The argument is almost the same as for the Markoff surface except that now we have an affine transformation and a lack of full symmetry in the variables. 
 
As in the case for the Markoff equation, assume that $\overline{V_{k,U_1}(\mathbb{Z})}\neq V_{k,U_1}$ so that it is contained in a finite union of curves.  Consider the two Vieta transformations: $\mathcal{V}_1({\bf x}) = (x_2x_3 -1-x_1,x_2,x_3)$ and  $\mathcal{V}_3({\bf x}) = (x_1,x_2,x_1x_2 -x_3)$, keeping $x_2$ fixed. Put ${\bf w}=(x_1,x_3)^T$ so that $\mathcal{V}_1$ and $\mathcal{V}_3$ act on ${\bf w}$.  By abuse of notation, we have
 \[\mathcal{V}_1({\bf w})=\left[\begin{array}{cc}
 -1 & x_2 \\ 0 & \ 1 \end{array}\right]{\bf w} + \left[\begin{array}{c}  
 -1  \\ 0  \end{array}\right]\quad \text{and} \quad \mathcal{V}_3({\bf w})=\left[\begin{array}{cc}
 1 & 0 \\ -x_2 & -1 \end{array}\right]{\bf w},
 \]
 so that we write $\mathcal{V}_1\mathcal{V}_3({\bf w}) = A{\bf w} + {\bf b}$, with 
 \[
 A= \left[\begin{array}{cc}
 -1- x_2^2 & -x_2 \\ -x_2 & \ -1 \end{array}\right]\in SL(2,\mathbb{Z}),\quad \text{and} \quad {\bf b}=\left[\begin{array}{c}  
 -1  \\ 0  \end{array}\right].
 \]
Hence $(\mathcal{V}_1\mathcal{V}_3)^n{\bf w} = A^n{\bf w} + \sum_{j=0}^{n-1}A^j{\bf b}$ for $n \geq 1$.
If $\mathcal{V}_1\mathcal{V}_3$ has order $n$, it follows that $(A^n - I)[(A-I){\bf w} + {\bf b}]=0$. Now, if $x_2 \neq 0$, then $A$ has infinite order and $A^n - I$ is invertible, so that we have $(A-I){\bf w}=-{\bf b}$. This is impossible since $(A-I)^{-1}{\bf b}$ is not integral. Hence $\mathcal{V}_1\mathcal{V}_3$ has infinite order so that the orbit $\mathcal{V}_1\mathcal{V}_3(x_1,x_2,x_3)$ with $x_2\neq 0$ fixed is infinite. The assumption of Zariski density implies that there are only finitely many $x_2$'s. 

Since the surface given by $U_1$ is symmetric in $x_2$ and $x_3$, it follows that there are only finitely many $x_2$ and $x_3$'s, from which we conclude that there are at most finitely many lattice points (since $x_1$ is determined). Starting with the base point ${\bf p}=(k-4,k-4,2)$ which is on the surface, we see that this is impossible since the orbit $\mathcal{V}_1\mathcal{V}_3({\bf p})$ is infinite if $k\neq 4$. Hence $V_{k,U_1}(\mathbb{Z})$ is Zariski dense in $V_{k,U_1}$ for all $k\neq 4$. For $k=4$, we use instead ${\bf p}=(-1,2,0)$ so that ${\bf w}\neq {\bf 0}$, and the argument above gives an infinite orbit, and Zariski dense.

\subsubsection{} The argument for $U_2$ is almost identical: we use the Vieta transformations  $\mathcal{V}_1({\bf x})  = (x_2x_3 +x_2 -x_1,x_2,x_3)$ and  $\mathcal{V}_3({\bf x}) = (x_1,x_2,x_1x_2 - x_2 -x_3)$, and have the corresponding matrix equation for 
   $\mathcal{V}_1\mathcal{V}_3({\bf w}) = A{\bf w} + {\bf b}$, with 
  \[
 A= \left[\begin{array}{cc}
 -1+x_2^2 & -x_2 \\ x_2 & \ -1 \end{array}\right]\in SL(2,\mathbb{Z}),\quad \text{and} \quad {\bf b}=\left[\begin{array}{c}  
 -x_2^2 + x_2  \\ -x_2  \end{array}\right].
 \]
 The analysis is the same as  for $U_1$ except now we have that $A^n -I$ is invertible if $|x_2|\geq 3$.  As above, we derive a contadiction of the finite order assumption since $(A-I)^{-1}{\bf b}$ is not integral. In particular, taking the base point  ${\bf p}=(k-1,k-1,1)$, we conclude that $V_{k,U_2}(\mathbb{Z})$ is infinite if $|k|\geq 4$. The reasoning above using the Zariski density assumption shows that there are only finitely many $x_2$'s.
 
 Due to the lack of symmetry in the variables, we redo the analysis with $x_1$ fixed, using $\mathcal{V}_2\mathcal{V}_3({\bf w}) = A{\bf w} + {\bf b}$, with $\mathcal{V}_2({\bf x}) = (x_1,x_1x_3 +x_1 -x_3 -x_2,x_3)$, $\mathcal{V}_3$ as before, ${\bf w}=(x_2,x_3)^T$ and
 \[
 A= \left[\begin{array}{cc}
 (x_1 -1)^2 -1 & 1-x_1 \\ x_1 -1 & \ -1 \end{array}\right]\in SL(2,\mathbb{Z}),\quad \text{and} \quad {\bf b}=\left[\begin{array}{c}  
 x_1\\ 0  \end{array}\right].
 \]
 If $|x_1 -1|\geq 3$, we conclude $(A-I){\bf w}=-{\bf b}$, and derive a contradiction regarding the finite order assumption. Thus the Zariski density hypothesis implies that there are only finitely many $x_1$'s. Hence, again since $x_3$ is determined by $x_1$ and $x_2$, $V_{k,U_2}(\mathbb{Z})$ is finite, giving a contradiction. Thus $V_{k,U_2}(\mathbb{Z})$ is Zariski dense in $V_{k,U_2}$ for all $|k|\geq 4$. For $|k|\leq 4$, a direct computation gives many eligible candidates for lattice points that lead to Zariski dense. 
 
 A much stronger theorem concerning $\Gamma$ invariant holomorphic curves and structures for the surfaces corresponding to \eqref{cform} is proved in (\cite{CL}, Theorem D).

%%%%%%%%%%%%%%%%%%%%%%%%%%%%%%%%%%%%%%%%%%%%%%%%%%%%%%%%%%%%%%%%%%%%%%%%%%%%%%%%%%%%%%%%%%%%%%%%%%%%%%%%%%%%%%%%%%%%%%%%%%%%%%%%%%%%%%%%%%%%%%%%%%%%%%%%%%%%%%%%%%%%%%%%%%%%%%%%%%%%%%%%%%
\section{Local solutions in $\mathbb{Z}_{p}$}
\label{sec5}
\begin{proposition}\label{local}
 Given $k\in \mathbb{Z}$, the congruence $x_{1}^{2} +x_{2}^{2} +x_{3}^{2} -x_1x_2x_3 \equiv  k \, ({\rm mod}\,  p^n)$ has solutions for all primes $p$ and $n\geq 1$ except for the following exceptions : $k\equiv 3 \, ({\rm mod}\, 4)$, $k\equiv \pm 3 \, ({\rm mod}\,  9)$.
\end{proposition}

We break up the proof into several cases.

It is particularly easy to verify the Proposition for powers of primes $p\geq 5$ as follows:  recall the Fricke trace identity, namely for any real unimodular matrices $A$ and $B$, 
\begin{equation}\label{fricke1}
\mathfrak{S}(A)^2 + \mathfrak{S}(B)^2 + \mathfrak{S}(AB)^2 - \mathfrak{S}(A)\mathfrak{S}(B)\mathfrak{S}(AB) = \mathfrak{S}([A,B]) +2,
\end{equation}
 where $[A,B]=ABA^{-1}B^{-1}$ is the commutator, and $\mathfrak{S}( )$ denotes the trace of the matrix.
 
 Restricting the matrices to $SL(2,\mathbb{Z})$, one obtains integral solutions to  \eqref{1a}, with $k=t+2$, where we denote $\mathfrak{S}([A,B])$ by $t$. We have
 
 \begin{lemma} For any prime $p \geq 5$, $n \geq 1$ and any integer $t$, there exists matrices $A,\,B\,\in SL(2,\mathbb{Z}\slash p^{n}\mathbb{Z})$ such that 
 \[
 \mathfrak{S}(A)^2 + \mathfrak{S}(B)^2 + \mathfrak{S}(AB)^2 - \mathfrak{S}(A)\mathfrak{S}(B)\mathfrak{S}(AB) \equiv t +2 \, ({\rm mod}\,  p^n)\,.
 \] 
 \end{lemma}
 \begin{proof} For
 \[
A= \begin{bmatrix}
a&b\\c&d
\end{bmatrix}, \quad B= \begin{bmatrix}
e&0\\0&f
\end{bmatrix},\quad  A, B \in SL(2,\mathbb{Z}\slash p^{n}\mathbb{Z}),
\]
we have $\mathfrak{S}([A,B]) = 2 a d e f - b c (e^2 + f^2) \equiv 2 - bc(e-f)^2 \, ({\rm mod}\,  p^n)$. 

Since $p\geq 5$, there exists $e$ and $f$ such that $(e-f,p)=1$ with $ef \equiv 1  \, ({\rm mod}\,  p^n)$. Then, we choose $c$ so  that $c(e-f)^2 \equiv 1 \, ({\rm mod}\, p^n)$.  Finally, we choose $a=1$, $b= 2-t$ and $d = 1 +bc$.\qed
 \end{proof}
 
 \begin{Corollary}\label{fifthpower} For $p \geq 5$ and $n \geq 1$, the Markoff congruence $x_{1}^{2} + x_{2}^{2} + x_{3}^{2} -x_1x_2x_3 \equiv k  \, ({\rm mod}\,  p^n)$ has the solution $x_1 \equiv 2 - (k-4)c$, $x_2 \equiv e+f$ and $x_3 \equiv e -f +fx_1$, with $e,\,f$ and $c$  as in the proof of the Lemma.
 \end{Corollary}

The argument above gives the existence of solutions for powers of $p\geq 5$. It is useful to have a precise count for the number of solutions modulo $p$. For this, it is not any harder to consider  the more general problem in

\begin{lemma}\label{5.4}\ For $p\geq 3$, let $N_{p}$ denote the number of solutions to $x_{1}^{2} +x_{2}^{2}+x_{3}^{2} -\alpha x_1x_2x_3 \equiv \beta$ modulo $p$. Then 
\[
 N_{p} = \left\{
 \begin{array}{ll}
  p^2 + p\left(\frac{-\beta}{p}\right)_{\!\!L} & \text{if}\  p|\alpha, \\
  \\
 p^2 + 1 + \left(\frac{\alpha^2\beta -4}{p}\right)_{\!\!L}\left[3+\left(\frac{\beta}{p}\right)_{\!\!L}\right] p   & \quad \text{otherwise}.
 \end{array} \right.
 \]
\end{lemma}
\proof
It is clear we need only consider the cases $\alpha=0$ and $\alpha=1$, the latter when $p\nmid \alpha$, upon which we multiply through with $\alpha^2$ and change variables.

Write $S_{p}(a)=\sum_{u} e_{p}(au^{2})$ (where $e_{p}(x)=e^{\frac{2\pi ix}{p}}$) so that when $p\nmid a$, one has $S_{p}(a) = \left(\frac{a}{p}\right)_{\!\!L}S_{p}(1)$. When $\alpha =1$, putting $u \equiv 2x_3 -  x_1x_2 \, ({\rm mod}\, p)$ shows that we have the same number of solutions as the congruence
\[
4(x_{1}^{2} + x_{2}^{2}) + u^2 - x_{1}^{2}x_{2}^{2} \equiv 4\beta \, ({\rm mod}\,  p)  ,
\]
so that
\begin{equation}\label{a2}
N_{p} - p^{2} = \frac{1}{p}\sum_{a\not\equiv 0} T_{p}(a)S_{p}(a)e_{p}(-4a\beta)\,;
\end{equation}
here we obtained $p^2$ solutions when $a \equiv 0 \, ({\rm mod}\,  p)$, and we put
\[
T_{p}(a)= \sum_{x_1,x_2} e_{p}\left(a(4x_{1}^{2}+4x_{2}^{2} - x_{1}^{2}x_{2}^{2})\right) = \sum_{x_1}e_{p}(4ax_{1}^{2})S_{p}\left(a(4- x_{1}^{2})\right)\,.
\]
Breaking the sum over $x_1$ in $T_{p}$ above depending on when $x_1 \equiv \pm 2$ or not gives us
\[
T_{p}(a) = 2p\,e_{p}(16a) + \left(\frac{a}{p}\right)_{\!\!L}S_{p}(1)\sum_{x_1}\left(\frac{4- x_{1}^{2}}{p}\right)_{\!\!L}e_{p}(4ax_{1}^{2}).
\]
Summing over $a$ in \eqref{a2} gives $N_{p} = p^2 + \mathcalorig{E}_{1} + \mathcalorig{E}_{2}$, where
\begin{equation}\label{a3}
\mathcalorig{E}_{1}= 2S_{p}(1)\sum_{a}\left(\frac{a}{p}\right)_{\!\!L}\, e_{p}\left(a(16 -4\beta)\right) =2S_{p}(1)^2\left(\frac{4-\beta}{p}\right)_{\!\!L}\, ,
\end{equation}
and
\begin{equation}\label{a4}
\mathcalorig{E}_{2} = \frac{1}{p}S_{p}(1)^2 \sum_{x_1}\left(\frac{4- x_{1}^{2}}{p}\right)_{\!\!L}\ \sum_{a\not\equiv 0}e_{p}\left(4a(x_{1}^{2} - \beta)\right).
\end{equation}
Summing over $a$ in \eqref{a4}, we write $\mathcalorig{E}_{2} = -\mathcalorig{E}_{2,1} + \mathcalorig{E}_{2,2}$ with 
\begin{align*}
\mathcalorig{E}_{2,1} &= \frac{S_{p}(1)^2}{p}\sum_{x_{1}^{2} \not{\equiv} \beta}\ \left(\frac{4-x_{1}^{2}}{p}\right)_{\!\!L}\ ,\\ &=  \frac{S_{p}(1)^2}{p}\left[\sum_{x_1}\ \left(\frac{4-x_{1}^{2}}{p}\right)_{\!\!L} - \left(\frac{4-\beta}{p}\right)_{\!\!L}\left[1 + \left(\frac{\beta}{p}\right)_{\!\!L}\right]\right],
\end{align*}
and
\[
\mathcalorig{E}_{2,2} = \frac{S_{p}(1)^2}{p}(p-1) \left(\frac{4-\beta}{p}\right)_{\!\!L}\left[1 + \left(\frac{\beta}{p}\right)_{\!\!L}\right]. 
\]
Since $\sum_{x} \left(\frac{4-x_{1}^{2}}{p}\right)_{\!\!L} = -\left(\frac{-1}{p}\right)_{\!\!L}$, it follows from \eqref{a3} and \eqref{a4} that 
\[
N_{p} = p^2 + 2S_{p}(1)^2\left(\frac{4-\beta}{p}\right)_{\!\!L}  + \frac{S_{p}(1)^2}{p}\left(\frac{-1}{p}\right)_{\!\!L} + S_{p}(1)^2\left(\frac{4-\beta}{p}\right)_{\!\!L} \left[1+ \left(\frac{\beta}{p}\right)_{\!\!L}\right].
\]
Using $S_{p}(1)^2 = p\left(\frac{-1}{p}\right)_{\!\!L}$ then gives us 
\[
N_{p} = p^2 + \left(\frac{\beta -4}{p}\right)_{\!\!L}\left[3+\left(\frac{\beta}{p}\right)_{\!\!L}\right] p +1. 
\]
It follows that $N_{p} \geq p^2 - 4p +1 = (p-2)^2 -3 >0$ if $p \geq 5$. This is also true of $p=3$  as can be checked with different values of $\beta$.

Next, if  $p|\alpha$,
\[
N_{p} - p^2 = \frac{1}{p}\sum_{a\not\equiv 0} e_{p}(-\beta a)S_{p}(a)^3 = \frac{S_{p}(1)^3}{p}\sum_{a}e_{p}(-\beta a)\left(\frac{a}{p}\right)_{\!\!L}.
\]
If $p|\beta$, then $N_{p}=p^2$. If $p\nmid \beta$, then the right hand side is $p\left(\frac{-\beta}{p}\right)_{\!\!L}$. 

\subsection{Prime powers : $p\geq 5$}
\label{subsec5a}\ 

We have already considered this case in Corollary \ref{fifthpower}, but for completeness we give here the argument using Hensel's lemma.  Let $f = x_{1}^{2} + x_{2}^{2} + x_{3}^{2} -x_1x_2x_3 -k$, considered as three  functions of each variable. Then $f'$ takes the following three  forms: $2x_1 -x_2x_3$, $2x_2 -x_1x_3$ or $2x_3 -x_1x_2$. To obtain solutions modulo $p^{n+1}$ from those modulo $p^n$, it suffices that at least one of these derivatives  not vanish modulo $p^n$. We call such triples non-singular. If $(x_1,x_2,x_3)$ is such a non-singular solution modulo $p^n$ with say $2x_1- x_2x_3 \not\equiv 0 \, ({\rm mod}\,  p^n)$, then Hensel's lemma gives a solution to $f \equiv 0 \, ({\rm mod}\,  p^{n+1})$ of the form $(y_{1},x_2,x_3)$ with $y_{1} \equiv x_1 \, ({\rm mod}\,  p^n)$. This new triple is non-singular modulo $p^{n+1}$ so that by induction a non-singular solution modulo $p$ lifts to one modulo $p^n$ for any $n\geq 1$, for any prime $p\geq 5$. Note that $(3,3,3)$ is a non-singular solution when $p\vert k$, giving solutions modulo $p^n$. 

Next suppose the triple $(x_1,x_2,x_3)$ is a singular solution of  the congruence $f \equiv 0 \, ({\rm mod}\,  p)$ for $p\nmid k$, so that we  have  $2x_1 \equiv x_2x_3$, $2x_2 \equiv x_1x_3$ and  $2x_3 \equiv x_1x_2 \, ({\rm mod}\,  p)$. If we assume $p\nmid x_1x_2x_3$, then necessarily $x_{1}^{2} \equiv x_{2}^{2} \equiv x_{3}^{2} $ and $x_1x_2x_3 \equiv 2x_{1}^{2} \, ({\rm mod}\,  p)$. Substituting into $f\equiv 0 \, ({\rm mod}\,  p)$ gives $x_{1}^{2} \equiv k \,({\rm mod}\, p)$ so that $k$ must be a non-zero quadratic residue modulo $p$, so say $k \equiv u^2 \, ({\rm mod}\,  p)$. But then $(u,0,0)$ is a non-singular solution to $f\equiv 0 \,({\rm mod}\,  p)$, and so by above, lifts to a non-singular solution modulo $p^n$ for all $n\geq 1$.

Finally, suppose $p|x_1x_2x_3$ with $(x_1,x_2,x_3)$ singular. Then $p$ divides $x_1$, $x_2$ and $x_3$, so that $p^{2}|k$. But then $(3,3,3)$ is a non-singular solution modulo $p^2$ for all $p > 3$.  We can now apply Hensel's lemma as above, starting modulo $p^2$ and lifting to solutions modulo $p^n$ for all $n\geq 2$ and $p>3$.

\subsection{Prime powers : $p=3$}
\label{subsec5b}\ 

The congruence $f \equiv 0 \, ({\rm mod}\,  3)$ has the following non-singular solutions : when $k\equiv  1$, take $(1,0,0)$; and when $k \equiv -1$, take $(0,1,1)$. These solutions lift to solutions modulo $3^n$ for $n\geq 1$. 

When $k\equiv 0 \, ({\rm mod}\,  3)$, the only solution is the singular $(0,0,0)$. We now consider this case modulo 9.
Since $3$ divides each of $x_1$, $x_2$ and $x_3$, then necessarily when $k \equiv 3$ or $6$ mod $9$, there are no solutions. So assume $9|k$, in which case $(3,0,0)$ is a non-singular solution modulo 9 and so lifts to solutions modulo $3^n$, $n\geq 2$.

\subsection{Prime powers : $p=2$}
\label{subsec5c}\ 

Modulo 2, $f' \equiv$ $x_1x_2$ or $x_1x_3$ or $x_2x_3$. Thus if $k$ is even, one may use the non-singular solution $(1,1,1)$ to obtain solutions modulo powers of 2. When $k$ is odd, the only solution is the singular $(0,0,1)$. Then necessarily $k\equiv 3 \, ({\rm mod}\,  4)$ has no solutions. So assume $k\equiv 1 \, ({\rm mod}\,  4)$ and we find the non-singular solution $(1,0,0)$ modulo 4 (note that here one uses $f' \equiv 2x_1 -x_2x_3 \not\equiv 0 \, ({\rm mod}\,  4)$). This then lifts to higher powers of 2.

%%%%%%%%%%%%%%%%%%%%%%%%%%%%%%%%%%%%%%%%%%%%%%%%%%%%%%%%%%%%%
%%%%%%%%%%%%%%%%%%%%%%%%%%%%%%%%%%%%%%%%%%%%%%%%%%%%%%%%%%%%%%%%%%%%%%
%%%%%%%%%%%%%%%%%%%%%%%%%%%%%%%%%%%%%%%%%%%%%%%%%%%%%%%%%%%%%%%%%%%%%%

\section{The average of $\mathfrak{h}_{M}^{\pm}(k)$: counting lattice points.}
\label{sec6}
We show here that the average of $\mathfrak{h}_{M}^{\pm}(k)$ is $C^{\pm}(\log{k})^{2}$, by counting lattice points in the domains given in Theorem \ref{Thm1} ((see the paragraph containing \eqref{avg} for definitions). We provide the details for $k>5$. 

Fix $u_{1}=a$ with  $3\leq a \ll K^{\frac{1}{3}}$ and write $u_2=m$ and $u_3=n$. We determine the asymptotics of $N_a(K)$,  the number of pairs $(m,n)$ satisfying the inequality $a^2 + m^2 +n^2 +amn \leq K$ with $a\leq m\leq n$. We have 
\[
m\leq n\leq \frac{1}{2}\left(-am + \sqrt{4(K-a^2) +(a^2 -4)m^2}\right) \, ,
\]
so that $m\leq K_a$, with $K_a = \sqrt{\frac{K-a^2}{a+2}}$. Hence 
\[
N_a(K) = \frac{1}{2}\sum_{a\leq m\leq K_a} \left\{\sqrt{4(K-a^2) +(a^2 -4)m^2} -(a+2)m\right\} + O\left(\sqrt{\frac{K}{a}}\right)\, .
\]
The function in the sum is decreasing in $m$ and the contribution from the endpoints are $O(\sqrt{K})$. Hence
\[
N_a(K) = \frac{1}{2}\int_{a}^{K_a} \left\{\sqrt{4(K-a^2) +(a^2 -4)x^2} -(a+2)x\right\} dx + O\left(\sqrt{K}\right)\, .
\]
Changing variables gives 
\[
N_a(K) = 2\frac{K-a^2}{\sqrt{a^2 -4}}\int_{\alpha}^{\beta} \left\{\sqrt{1+ x^2} - \frac{a+2}{\sqrt{a^2 -4}}x\right\} dx + O\left(\sqrt{K}\right)\, ,
\]
where $\beta = \frac{\sqrt{a-2}}{2}$ and $\alpha = O(aK^{-\frac{1}{2}})$. Replacing $\alpha$ with zero gives an error of $O(\sqrt{K})$ and the integral becomes
\[
\frac{1}{2}\left\{\beta\sqrt{1+\beta^2} + \log{\left(\beta + \sqrt{1+\beta^2}\right)} - \frac{a+2}{\sqrt{a^2 -4}}\beta^2\right\}\, .
\]
Simplifying gives us 

\begin{lemma}
\label{av0}
 For $3\leq a\ll K^{\frac{1}{3}}$, the number of pairs $(m,n)$ satisfying the inequality $a^2 + m^2 +n^2 +amn \leq K$ with $a\leq m\leq n$ is
\[
N_{a}(K) = \log{\left[\frac{\sqrt{a-2} + \sqrt{a+2}}{2}\right]}\frac{K-a^2}{\sqrt{a^2 -4}} + O(\sqrt{K}).
\]
\end{lemma}

\begin{lemma}
\label{average}
 Let  $R^{+}(K)$ be  the number of points $(x_1,x_2,x_3)$ satisfying $x_{1}^2 +x_{2}^2 +x_{3}^2 +x_1x_2x_3 \leq K$, with $3\leq x_1\leq x_2\leq x_3$\, .  Then 
\[
R^{+}(K) = \frac{1}{36}K(\log{K})^2   +O(K\log{K})\, .
\]
\end{lemma}
\proof{} It follows from the previous lemma that
\[{}
R^{+}(K) = \sum_{{3\leq a\leq K^{\frac{1}{3}}}}\, \log{\left[\frac{\sqrt{a-2} + \sqrt{a+2}}{2}\right]}\frac{K}{\sqrt{a^2 -4}} + O\left(K^{\frac{5}{6}}\right).
\]
The main term is asymptotic to $\frac{K}{2}\sum_{a}\frac{\log{a}}{a} \sim \frac{K}{4}(\log{K^{\frac{1}{3}}})^{2}$\, .\qed

We also state, without details, the analogous count for the case of $k<0$ in Theorem \ref{Thm1}(ii).

\begin{lemma}\label{average2}
 Let  $R^{-}(K)$ be  the number of points $(x_1,x_2,x_3)$ satisfying $x_{1}^2 +x_{2}^2 +x_{3}^2 - x_1x_2x_3 = -k$, with $0<k\leq K$  and $3\leq x_1\leq x_2\leq x_3\leq \frac{1}{2}x_1x_2$ .  Then 
\[
R^{-}(K) = \frac{1}{48}K(\log{K})^2   +O(K\log{K})\, .
\]
\end{lemma}
%%%%%%%%%%%%%%%%%%%%%%%%%%%%%%%%%%%%%%%%%%%%%%%%%%%%%%%%%%%%%%%%%%%%%%
%%%%%%%%%%%%%%%%%%%%%%%%%%%%%%%%%%%%%%%%%%%%%%%%%%%%%%%%%%%%%%%%%%%%%%
%%%%%%%%%%%%%%%%%%%%%%%%%%%%%%%%%%%%%%%%%%%%%%%%%%%%%%%%%%%%%%%%%%%%%%
%%%%%%%%%%%%%%%%%%%%%%%%%%%%%%%%%%%%%%%%%%%%%%%%%%%%%%%%%%%%%%%%%%%%%%

\section{Failures of the Hasse Principle}
\label{sec7}

The fundamental sets allows us to determine Hasse failures for small $k$ very readily. For example, direct computations reveal that the smallest positive Hasse failure occurs with $k=46$. That $k=46$ is a Hasse failure can be verified  by applying  Theorem \ref{Thm1} as follows: either $k=46$ is exceptional or there exist $3\leq x_1\leq x_2\leq x_3$ such that $x_{1}^2 + x_{2}^2 +x_{3}^2 +x_1x_2x_3=46$. The latter cannot occur since the smallest value of the polynomial is $54$. To determine if $46$ is exceptional, since it is not a sum of two squares and since 42 is not a square, it remains to check if the equation $x_{2}^2 + x_{3}^2 - x_{2}x_{3} = 45$ has any solutions with $x_2,\ x_3\ \in \mathbb{Z}$. The equation implies that $3|x_2$ and $3|x_3$, so that we consider the solvability of $y_{1}^2 + y_{2}^2 -y_{1}y_{2}=5$. This is equivalent to the solvability of $u_{1}^2 + 3u_{2}^2 = 20$, which is impossible by congruence modulo 5 or otherwise.

Let $V_{k}(\mathbb{Z})$ denote the integral points on the surface $x_{1}^2 + x_{2}^2 + x_{3}^2 -x_1x_2x_3 = k$, for $k\in \mathbb{Z}$. For $k=4+d$, the surface $V_k$ is the singular Cayley surface when reduced modulo $d$. Its features, coupled with global quadratic reciprocity, yield failures of strong approximation (mod $4d$). For example, assume that $n \to \left(\frac{4d}{n}\right)$ is a primitive Dirichlet character (mod $4d$) and let $S_d \subset \mathbb{Z}/4d\mathbb{Z}$ be the multiplicative closed set $\left\{n\ :\ \left(\frac{4d}{n}\right) =0\ \text{or}\ 1\right\}$. Then, for any ${\bf x}=(x_1,x_2,x_3) \in V_{k}(\mathbb{Z})$ one has 
\begin{equation}\label{7a}
x_{j}^2 -4 \in S_d\ (\text{mod}\ 4d)\ ,\ \text{for}\ j=1, 2, 3\ .
\end{equation}
These congruences on $x_j$ imposed by \eqref{7a} are not consequences of local considerations and so strong approximation fails for $V_{k}(\mathbb{Z})$, at least (mod $4d$). 

To see \eqref{7a}, we rewrite \eqref{1a} as
\begin{equation}\label{7b}
w^2 -4d = \left(x_{1}^2 -4\right)\left(x_{2}^2 -4\right)\ ,
\end{equation}
with $w= 2x_3 - x_1x_2$. Now, if $x_{1}^2 -4 = p_1p_2\ldots p_l$ with $p_j$ primes (possibly with repetition), then $w^2 \equiv 4d \ (\text{mod}\ p_j)$ and hence $\left(\frac{4d}{p_j}\right) =0$ or $1$. Thus $p_j \in S_d$ for each $j$, and hence so does $x_{1}^2 -4$. The same applies to $x_{2}^2 -4$ and $x_{3}^2 -4$. Quadratic reciprocity then implies that the $x_j$'s must lie in certain congruence classes (mod $4d$).

As we now show, by specializing the $k$'s and enhancing the analysis above, we can eliminate all the candidate congruence classes and produce families of Hasse failures. We turn to these and the proof of Theorem \ref{Thm2}(i) in the Introduction, which follows from Prop. \ref{HF1} below.

\begin{proposition}\label{HF1} For the following choices of $k$, $V_{k}(\mathbb{Z})$ is empty but  $V_{k}(\mathbb{Z}_p)$ is non-empty   for all primes $p$ :
\begin{enumerate}[label=(\roman*).]
\item  For $k<0$, choose $k= 4 - 2\nu^2$, with odd $\nu$ having all its prime factors congruent to $1$ or $3$ modulo 8.
\item  For $k>4$, choose $k= 4 + 2\nu^2$ with $\nu$ having all of its prime factors in the congruence classes $\{\pm 1\}$ modulo 8, and in addition with  $\nu \in \{0,\ \pm3,\ \pm4\}$ modulo 9.
\item Suppose $\ \ell\geq 13$ is a prime number with $\ell \equiv \pm 4$  (mod 9). Then choose $k = 4 + 2\ell^2$. 
\end{enumerate}
The smallest positive $k$ here is 342.
\end{proposition}
\begin{proof} \ 
Writing $k = 4 + 2\epsilon \nu^{2}$ with $\epsilon = \pm 1$, with odd $\nu$, the congruence conditions ensure that Prop. \ref{local} implies  $V_{k}(\mathbb{Z}_{p})\neq \varnothing$ for all primes $p$.

Let $(x_1,x_2,x_3)$ be a solution to 
\begin{equation}\label{2a}
x_{1}^2 + x_{2}^2 + x_{3}^2 -x_1x_2x_3 = 4 +2\epsilon \nu^2\ ,
\end{equation}
with the corresponding 
\begin{equation}\label{2b}
w^2 - 8\epsilon\nu^2 = (x_{1}^2 -4)(x_{2}^2 -4)\ ,
\end{equation}
with $w=2x_3 - x_1x_2$.

Since $\nu$ is odd, $4 \pm 2\nu^2$ is not divisible by 4, so that at least one of $x_1$, $x_2$ or $x_3$ is odd, so say $x_1$. Then $x_{1}^2 -4 \equiv 5 \equiv -3 \  (\text{mod}\ 8)$. 

\noindent \underline{Case (i)}: It follows that $x_{1}^2 - 4$ is divisible by a prime number $q \equiv -1 \ \text{or} \ -3  \  (\text{mod}\ 8)$. Since $q\nmid \nu$, it follows from \eqref{2b} that $-2$ is a quadratic residue modulo $q$, a contradiction. 

\noindent \underline{Case (ii)}: It follows that $x_{1}^2 - 4$ is divisible by a prime number $q \equiv \pm 5  \  (\text{mod}\ 8)$. Since $q\nmid \nu$, it follows from \eqref{2b} that $2$ is a quadratic residue modulo $q$, a contradiction. 

\noindent \underline{Case (iii)}: Recall that for $k\geq 5$, if $k$ is not exceptional, every solution is equivalent to one in the fundamental set $3\leq x_1\leq x_2\leq x_3$ with $x_{1}^2 + x_{2}^2 + x_{3}^2 +x_1x_2x_3= k = 4+2\ell^2$. This implies that $3\leq x_1\leq k^{\frac{1}{3}}$ and $x_1\leq x_2\leq \left(\frac{k}{x_1}\right)^{\frac{1}{2}}$. Now, the proof above requires that at least one of the variables is odd; but in fact at least 2 variables are odd (by considering the equation modulo 4). It follows that we derive a contradiction if we follow the proof above with $q\neq \ell$. On the other hand, if $q=\ell$, since two variables are odd, we can choose one, say $x_1$ satisfying $3\leq x_1\leq \sqrt{\frac{k}{3}}$ with $\ell|(x_{1}^2 -4)$. Then $\ell|(x_1 -2)$ or $\ell|(x_1 +2)$ so that $x_1 +2 \geq \ell t$ for some $t\geq 1$. Then we get
\[
\ell-2 \leq \ell t -2 \leq x_1 \leq \sqrt{\frac{4+2\ell^2}{3}}\ ,
\]
so that $3(\ell-2)^2 \leq 4 + 2\ell^2$, which implies that $\ell<13$, a contradiction.

To complete the proof, it remains to check that our choice of $k$ is not exceptional, that is there are no solutions with say $x_1$ equal to $0$, $\pm1$ or $\pm2$. If $x_1=0$, then since two variables are odd, we have $x_2$ and $x_3$ are odd with $x_{2}^2 +x_{3}^2 = 4+2\ell^2$. The left side is congruent to 2 while the right is congruent to 6 modulo 8. Next, if $x_1 =\pm 1$, we have $x_{2}^2 + x_{3}^2 -x_2x_3= 3+2\ell^2$. Completing the square gives us $(2x_1 - x_3)^2 + 3x_{3}^2 = 4(3+2\ell^2)$, so that $8\ell^2$ is a quadratic residue modulo 3. This is a fallacy since $8\ell^2 \equiv 2$. Finally, the case $x_1=\pm 2$ is trivially dealt with since it implies that 2 is a square.
\end{proof}

We continue below with variants of this construction of Hasse failures, their densities being no more than the $k$'s in Prop. \ref{HF1}, which is $K^{\frac{1}{2}}(\log{K})^{-\frac{1}{2}}$ and establishes Theorem \ref{Thm2}(i).

\begin{proposition} Suppose $\nu^2 \equiv 25 \ (\text{mod}\ 32)$ with $\nu$ having all prime factors $\equiv \pm 1 \ (\text{mod}\  12)$. Then, $V_{k}(\mathbb{Z})$ is empty with $k=4 + 12\nu^2$, but has local solutions. The smallest $\nu$ is 37, with $k=16432$.
\end{proposition}
\begin{proof}

It is obvious that with the choice of $k$, the conditions of Prop. \ref{local} are satisfied so that local solutions exist.

We first consider congruences modulo $12$, where the  squares are in $\{0,\ 1,\ 4,\ 9\}$. Suppose $(x_1,x_2,x_3)$ is a solution to 
\begin{equation}\label{12a}
x_{1}^2 + x_{2}^2 + x_{3}^2 +x_1x_2x_3 = 4 + 12\nu^2,
\end{equation}
 with $\nu$ as above. 

If $2\nmid x_1x_2x_3$, then $x_{1}^2 -4 \equiv 5\ (\text{mod}\ 12)$ or is divisible by 3 (the same holding for $x_2$ and $x_3$). From \eqref{12a} we have
\begin{equation}\label{12b}
w^2 - 48\nu^2 = (x_{1}^2 -4)(x_{2}^2 -4),
\end{equation}
so that if $x_{1}^2 -4 \equiv 5 \ (\text{mod}\ 12)$, there is a prime $p \equiv \pm 5 \ (\text{mod}\ 12)$ with $p|(x_{1}^2 -4)$. This is not possible since $p\nmid \nu$ implies that $3$ is a quadratic residue (mod $p$), a fallacy. The same holds for $x_2$ and $x_3$, so that we may assume that $x_{1}^2 \equiv x_{2}^2 \equiv x_{3}^2 \equiv 1 \ (\text{mod}\ 12)$, so that each lies in the set $\{\pm 1,\ \pm 5\}$ modulo 12. If $x_1 \equiv \pm 5$, then in $x_{1}^2 -4 = (x_1 -2)(x_2+2)$, at least one factor is congruent to $\pm 5$, so that the argument above with a prime $p$ gives a contradiction. Hence we may assume that $x_1 \equiv \pm 1 \ (\text{mod}\ 12)$, and the same for $x_2$ and $x_3$. But then 9 divides the right hand side of \eqref{12b}, a contradiction.

Next, if $2\nmid x_1x_2$, but $2|x_3$, we see that a Vieta map gives the solution $(x_1,x_2,-(x_1x_2+x_3))$ with all coordinates odd, so that the previous analysis give a contradiction.

Hence we assume $x_1$, $x_2$ and $x_3$ are all even, so that changing variables gives us the equation
\begin{equation}\label{12c}
y_{1}^2 + y_{2}^2 + y_{3}^2 + 2y_1y_2y_3 = 1 + 3\nu^2,
\end{equation} with the corresponding
\begin{equation}\label{12d}
w_{1}^2 - 3\nu^2 = (y_{1}^2 -1)(y_{2}^2 -1).
\end{equation}
If $y_1$ is odd, then $8|(y_{1}^2 -1)$ so that $3$ is a quadratic residue mod 8, a fallacy. Hence we assume all $y_1$, $y_2$ and $y_3$ are even. We now consider congruences modulo 16. We first note that $1 + 3\nu^2 \equiv 12 \ (\text{mod}\ 16)$, so that we cannot have 4 dividing each of the variables. 

Next, if $4|y_1$, $4|y_2$ and $y_3 \equiv 2 \ (\text{mod}\ 4)$, then \eqref{12c} gives us $y_{3}^2 \equiv 12 \ (\text{mod}\ 16)$, an impossibility. Similarly if $4|y_1$ but $y_2 \equiv y_3 \equiv 2 \ (\text{mod}\ 4)$, then $y_{2}^2 + y_{3}^2 \equiv 12 \ (\text{mod}\ 16)$, which we see again is impossible. Thus, we may assume that $y_1\equiv y_2 \equiv y_3 \equiv 2 \ (\text{mod}\ 4)$, in which case we write $y_1=2z_1$, $y_2=2z_2$ and $y_3=2z_3$, with $2\nmid z_{1}z_{2}z_{3}$. Then \eqref{12c} becomes
\[
z_{1}^{2} + z_{2}^{2} + z_{3}^{2} + 4z_{1}z_{2}z_{3} = 1 + 3\left(\frac{\nu^{2} -1}{4}\right).
\]
The left hand side is congruent to 7 modulo 8, while the right is congruent to 3. Hence the result follows.
\end{proof}

\begin{proposition} Suppose $\nu \equiv \pm 4 \ (\text{mod}\ 9)$ with $\nu$ having all its  prime factors congruent to $ \pm 1 \ (\text{mod}\  20)$. Then, $V_{k}(\mathbb{Z})$ is empty with $k=4 + 20\nu^2$, but has local solutions. The smallest $\nu$ is 41, with $k=33624$.
\end{proposition}
\begin{proof}
The proof is very much the same as the one above, with a small change. The squares modulo 20 lie in the set $\{0, 1, \pm 4 ,5,9 \}$  and the odd primes in $\{\pm 1,\pm 3,\pm 7,\pm 9\}$.

Write $w^2 -80\nu^2 = (x_{1}^2 -4)(x_{2}^2 -4)$. If $5|x_{1}$, there must exist a prime $p \equiv \pm 2$ modulo 5 dividing $x_{1}^{2} -4$, so that since $p \nmid \nu$, $80$ is a quadratic residue modulo $p$, which is false using quadratic reciprocity. So we may assume $5\nmid x_{1}x_{2}x_{3}$ so that $x_{j}^{2}-4$ is not $1$ modulo $20$.

If $2\nmid x_1x_2x_3$,  since $x_1$ is odd, we have $x_{1}^2 -4 \equiv -3$ or $5$ modulo 20. Assume the former. Then, if there is a prime factor $p|(x_{1}^2 -4)$, with $p \equiv \pm 3$ or $\pm 7$ (mod 20), then $w^2 \equiv 80\nu^2$ (mod $p$), so that $5$ is a quadratic residue modulo $p$; that is $p$ a quadratic residue modulo 5, which is not true. Hence $x_{1}^2 -4$ must have a prime factor $p \equiv 9$ (mod 20). But then, $x_{1}^2 -4 \equiv 3$ implies that there must be another prime factor $q \equiv \pm 3$ or $\pm 7$ all modulo 20, and that leads to a contradiction. Hence we cannot have $x_{1}^2 -4\equiv -3$ (mod 20), and the same being so for $x_2$ and $x_3$. Hence we must have $x_{1}^2 -4 \equiv x_{2}^2 -4 \equiv 5$ (mod 20), so that $w^2 -80\nu^2 = (x_{1}^2 -4)(x_{2}^2 -4)$ implies that $25|80$.

 If $2\nmid x_1x_2$, but $2|x_3$, the Vieta map gives the solution $(x_1,x_2,-(x_1x_2+x_3))$ with all coordinates odd, so that the previous analysis give a contradiction. Hence we assume $x_1$, $x_2$ and $x_3$ are all even, so that changing variables gives us the equation
\begin{equation}\label{12cc}
y_{1}^2 + y_{2}^2 + y_{3}^2 + 2y_{1}y_{2}y_{3} = 1 + 5\nu^2,
\end{equation} with the corresponding
\begin{equation}\label{12d2}
w_{0}^2 - 5\nu^2 = (y_{1}^2 -1)(y_{2}^2 -1).
\end{equation}
If $y_{1}$, $y_{2}$ and $y_{3}$ are all even, then we have a contradiction in \eqref{12cc} since $v^2 \equiv 1$ (mod 4).
If $y_{1}$ is odd, then $8|(y_{1}^2 -1)$ so that $5$ is a quadratic residue mod 8, a fallacy. The result follows.
\end{proof}
%%%%%%%%%%%%%%%%%%%%%%%%%%%%%%%%%%%%%%%%%%%%%%%%%%%%%%%%%%%%%%%%%%%%%%
%%%%%%%%%%%%%%%%%%%%%%%%%%%%%%%%%%%%%%%%%%%%%%%%%%%%%%%%%%%%%%%%%%%%%%
%%%%%%%%%%%%%%%%%%%%%%%%%%%%%%%%%%%%%%%%%%%%%%%%%%%%%%%%%%%%%%%%%%%%%%

\section{Proof of Theorem \ref{Thm2}(ii)}
\label{sec8}
The proofs for the case $k>0$ and $k<0$ are almost identical with the main difference being in the choice of our functions and the domains of the variables. We give here the details for the case $k>0$ and indicate the modification for $k<0$ in a remark below.

Let $K \rightarrow \infty$ be our main (large) parameter, and let $A$ be a secondary parameter satisfying $(\log K)^2 < A \ll_{\ve}K^{\ve}$, with $\ve >0$ sufficiently small. Let $\mathcal{A}$ be the interval $[\sqrt{A},\ A]$.  Lastly we use a parameter  $m=\prod_{p\leq L}\ p^B$, where we put $L= \frac{\log{A}}{\log\log{A}}\Phi(A)$ and $B= \frac{\log\log{A}}{\Phi(A)^2}$ with $\Phi(A) \to \infty$ with $A$ . Then, $m \sim A^{\frac{1}{\Phi(A)}}$ as $A\to \infty$. 

 For any $a\in \mathcal{A}$, put 
\begin{equation}\label{mo1}g_{a}(x_1,x_2)= x_{1}^2 + x_{2}^2 + ax_1x_2 \quad \text{and}\quad f_{a}(x_1 ,x_1 )= g_{a}(x_1 ,x_2 ) + a^2. 
\end{equation}
It will be convenient to denote by $D_{a}$ the discriminant $a^2 -4$ of the indefinite quadratic form $g_a$ above, for each $a\in \mathcal{A}$. Completing the square shows that $4g_a(x_1,x_2) = (2x_1 + ax_2)^2 - D_{a}x_{2}^2$, so that we consider the form $G_d(s_1, s_2)= s_{1}^2 - ds_{2}^2$ with $d=D_a$ (not a complete square).
We then  define the sector $\mathcal{S}_{d}$ in the plane as
\begin{equation}\label{mo1a}
\begin{split}
\mathcal{S}_{d} &= \left\{(s_1, s_2)\, :\ s_1,\ s_2 \geq 0,\ \ 0\leq G_{d}(s_1,s_2) \leq1 ,\ \  2\sqrt{d}s_2 \leq s_1 \leq 3\sqrt{d}s_2 \ \right\}\, ,\\
&= \left\{ (x_1,x_2) : \begin{split} &\quad\quad x_1,\, x_2 \geq 0,\ 0 \leq g_a(x_1,\, x_2)\leq \frac{1}{4},\\ &\frac{1}{2}\left(2\sqrt{d} - a\right)x_2 \leq x_1 \leq \frac{1}{2}\left(3\sqrt{d} -a\right)x_2\end{split}\right\} \, .
\end{split}
\end{equation}

\begin{remark} For $k<0$ we define $g_a(x_1,x_2)= x_{1}^2 + x_{2}^2 -ax_1x_2$ and define the sector $\mathcal{S}_d$ with the constants $2$ and $3$ replaced by $\frac{1}{3}$ and $\frac{1}{2}$ respectively. This then leads to some minor changes for the sector in the variable $x_1$ and $x_2$.\qed
\end{remark}
Next, we define the scaled region $\sqrt{X}\mathcal{S}_{d}= \left\{\left(\sqrt{X}s_1,\sqrt{X}s_2\right): \ (s_1,s_2)\in \mathcal{S}_{a}\right\}$. It is easily shown that
\[ \text{Vol}(\sqrt{X}\mathcal{S}_{d}) = C\frac{X}{\sqrt{d}}, \]
with $C=\frac{1}{4}\log{\frac{3}{2}}$.

For $2 \leq k \leq K$, we define
\begin{equation}
\label{mo1b}
R_{d}(k) = \#\left\{(s_1,s_2)\in \sqrt{K}\mathcal{S}_{d}\cap\mathbb{Z}^2:  \ G_{d}(s_1,s_2)=k\ \text{and}\ 2|(s_1-s_2)\ \right\}\, .
\end{equation}

\begin{lemma}
\label{Lmo0}
 For $d$ and $m$ as above we have
\[
\sum_{k\leq K}\ R_{d}(k) = \ \frac{CK}{2\sqrt{d}} +O_{\ve}\left(K^{\frac{1}{2}+\ve}\right)\, .
\]
\end{lemma}

\begin{proof}
By the definition of $\mathcal{S}_d$, we break up the sum in $s_1$ and $s_2$ into the ranges so that $\mathcal{S}_d = \mathcal{S}^{(1)}_d \cup \mathcal{S}^{(2)}_d$ with
\[
\mathcal{S}_{d}^{(1)} = \left\{ s_2\leq \sqrt{\frac{K}{8d}}\ , \ 2\sqrt{d}s_2 \leq s_1 \leq 3\sqrt{d}s_2, \ 2\vert(s_1-s_2)\right\}\, ,
\]
and
\[
 \mathcal{S}_{d}^{(2)}= \left\{\sqrt{\frac{K}{8d}}\leq s_2\leq \sqrt{\frac{K}{3d}}\ ,\  2\sqrt{d}s_2 \leq s_1 \leq \sqrt{K +ds_{2}^{2}},\ 2\vert(s_1 - s_2)\, \right\}\, .
 \]
The sums are easily evaluated.\qed
\end{proof}

\begin{lemma}
\label{Lmoa}
For $a$  and $m$ as above and for any $ \alpha_1$ and $\alpha_2$, we have
\[
\sum_{\substack{x_1 \equiv \alpha_1 (m)\\ x_2 \equiv \alpha_2 (m)\\ f_{a}(x_1 ,x_2)\leq K\\ (x_1 ,x_2 )\in \sqrt{4K}\mathcal{S}_{D_a}\cap\mathbb{Z}^2}} 1 =\ \frac{2CK}{\sqrt{D_a}m^2} +O_{\ve}\left(K^{\frac{3}{4}+\ve}\right), 
\]
with the error term uniform is all other variables. 
\end{lemma}
\begin{proof}
By \eqref{mo1} we have trivially $x_1,\ x_2 \ll \sqrt{K}$ . Assuming $0\leq \alpha_j <m$, we put $x_j = \alpha_j +ml_j$ with $1\leq l_j \ll \frac{\sqrt{K}}{m}$. Then $(x_1 ,x_2 )\in \sqrt{4K}\mathcal{S}_{D_a}\cap\mathbb{Z}^2$ and $x_2 \ll K^{\frac{1}{4}+\ve}$ gives at most $O(K^{\frac{1}{2}+\ve})$ lattice points, so that we may assume that $K^{\frac{1}{4}+\ve}\ll x_1,\  x_2 \ll \sqrt{K}$. It is then easily checked that $C_{1}'l_2 \leq l_1 \leq C_{2}'l_2$, with $C_{j}' = C_{j}\left(1 + O(K^{-\frac{1}{4}+\ve})\right)$, where we have put $C_{1}=\frac{1}{2}\left(2\sqrt{d}-a\right)$ and $C_{2}=\frac{1}{2}\left(3\sqrt{d}-a\right)$ as in \eqref{mo1a}.

Next \[\big{|} f_{a}(x_1,x_2) -m^2f_{a}(l_1,l_2)\big{|} \ll K^{\frac{1}{2}+\ve}.\] The error in replacing the condition $f_{a}(x_1,x_2) \leq K$ with the condition $m^2g_{a}(l_1,l_2)\leq K$ is at most $O(K^{\frac{1}{2} +\ve})$ since we are counting lattice points in a hyperbolic segment of width $K^{\frac{1}{2}+\ve}$, with the variables restricted as above. Thus
\[ 
\sum_{\substack{x_1 \equiv \alpha_1 (m)\\ x_2 \equiv \alpha_2 (m)\\ f_{a}(x_1 ,x_2)\leq K\\ (x_1 ,x_2 )\in \sqrt{4K}\mathcal{S}_{D_a}\cap\mathbb{Z}^2}} 1 \ = \sum_{\substack{ g_{a}(l_1 ,l_2)\leq \frac{K}{m^2}\\ (l_1 ,l_2 )\in \sqrt{\frac{4K}{m^2}}\mathcal{S}^{*}_{D_a}\cap\mathbb{Z}^2}} 1
+O_{\ve}\left(K^{\frac{1}{2}+\ve}\right), 
\]
where $\mathcal{S}^{*}$ means the constants have been perturbed by about $O(K^{-\frac{1}{4}+\ve})$, as discussed above. Completing the square shows that the last sum is over the $s_1$ and $s_2$ variables as in \eqref{mo1a}  with the  constraint that $s_1- s_2$ is even, and with the constants $2$ and $3$ defining the inequalities perturbed with the addition of $O(K^{-\frac{1}{4}+\ve})$.  Applying Lemma \ref{Lmo0} with $K$ replaced with $\frac{4K}{m^2}$ gives the result, with $C$ replaced with $C + O(K^{-\frac{1}{4} +\ve})$.

\end{proof}

\begin{Corollary}\label{Lmob}
For $a\in \mathcal{A}$ and  $k \leq K$ let
\[
r_{a}(k) = \#\left\{(x_1,x_2)\in \sqrt{4K}\mathcal{S}_{a}\cap\mathbb{Z}^2:  f_{a}(x_1,x_2) =k\right\}\, .
\]
Then
\[
\sum_{k\leq K}\ r_{a}(k) = \frac{2CK}{\sqrt{D_a}} + O_{\ve}\left(K^{\frac{1}{2}+\ve}\right)\, .
\]
\end{Corollary}
\vspace{20pt}

We now set
\begin{equation}\label{mo4}
b_{\mathcal{A}}(k) = \sum_{a\in \mathcal{A}} r_{a}(k),
\end{equation}
and we are interested in this as a function of $k$ for $1\leq k \leq K$. From Corollary \ref{Lmob}, we have
\begin{equation}\label{mo5}
\sum_{1\leq k\leq K} b_{\mathcal{A}}(k) = CK\log{A} + O(KA^{-1})\, ,
\end{equation}
so that the mean-value of $b_{\mathcal{A}}(k)$ is $C\log{A}$. Our main goal is to estimate the deviation of $b_{\mathcal{A}}(k)$ from its predicted value in terms of local masses. Let $\delta(V_k)$ denote the formal singular series for
\begin{equation}\label{mo6}
V_{k}:\quad\quad x_{1}^2 + x_{2}^2 + x_{3}^2 + x_{1}x_{2}x_{3} = k, 
\end{equation}
so that $\delta(V_k) = \prod_{p<\infty}\delta_{p}(V_k)$, with 
\[
\delta_{p}(V_k) = \lim_{\nu \to \infty}\frac{\#V_{k}(\mathbb{Z}/p^{\nu}\mathbb{Z})}{p^{\nu}}\, .
\]
These are given explicitly in the Appendices and Section 5. 
Define 
\begin{equation}\label{mo7}
\delta^{(m)}(k) = \frac{\# V_{k}(\mathbb{Z}/m\mathbb{Z})}{m^2} := \frac{r_{m}(k)}{m^2}\, .
\end{equation}
Note that $\delta^{(m)}(k)$ depends on $k$  modulo $m$.
With this, we define our variance
\begin{equation}\label{mo8}
V(K)= V(K,A,m) = \sum_{k\leq K} \left(b_{\mathcal{A}}(k) - C(\log A)\delta^{(m)}(k)\right)^2\, .
\end{equation}
We expand \eqref{mo8} as $\Sigma_1  + \Sigma_2 + \Sigma_3$. We have
\begin{equation}\label{mo9}
\begin{split}
\Sigma_3 &= C^2(\log{A})^2 \sum_{l\, ({\rm mod}\, m)}\ \delta^{(m)}(l)^2\, \sum_{\substack{k\leq K\\ k \equiv l ({\rm mod}\, m)}}1\, ,\\
&=  C^2(\log{A})^2 \left(\frac{K}{m} +O(1)\right)\sum_{l\, ({\rm mod}\, m)} \delta^{(m)}(l)^2\, ,\\
&= C^2(\log{A})^2 \left(K+O(m)\right)\delta_{m}\left(V^{(2)}\right)\, ,
\end{split}
\end{equation}
where we define
\begin{equation}\label{mo10}
V^{(2)}:  \quad \quad x_{1}^2 + x_{2}^2 + x_{3}^2 + x_{1}x_{2}x_{3} = y_{1}^2 + y_{2}^2 + y_{3}^2 + y_{1}y_{2}y_{3} , 
\end{equation}
and $\delta_{m}\left(V^{(2)}\right)$ is the singular series for $V^{(2)}$ over $\mathbb{Z}/m\mathbb{Z}$.

Next,
\begin{equation}\label{mo11}
\begin{split}
\Sigma_{2} &= -2C(\log{A})\sum_{k\leq K}\, b_{\mathcal{A}}(k)\delta^{(m)}(k) = -\log{A} \sum_{l\, ({\rm mod}\, m)} \delta^{(m)}(l) \sum_{\substack{k\leq K\\k \equiv l ({\rm mod}\, m)}} b_{\mathcal{A}}(k)\, ,\\
&= -2C\log{A} \sum_{l\, ({\rm mod}\, m)} \delta^{(m)}(l) \sum_{a\in \mathcal{A}}\, \sum_{\substack{k\leq K\\k \equiv l ({\rm mod}\, m)}}\, r_{a}(k)\, .
\end{split}
\end{equation}
Now, for each $a\in \mathcal{A}$, the last sum in \eqref{mo11} above is
\begin{equation}\label{mo12}
\sum_{\substack{k\leq K\\k \equiv l ({\rm mod}\, m)}}\, r_{a}(k) = \sum_{\substack{\alpha_1 , \alpha_2 ({\rm mod}\, m)\\ f_{a}(\alpha_1 , \alpha_2) \equiv l ({\rm mod}\, m)}}\sum_{\substack{y_1 \equiv \alpha_1\, ({\rm mod}\, m)\\ y_2 \equiv \alpha_2\, ({\rm mod}\, m)\\ f_{a}(y_1 ,y_2)\leq K\\ (y_1 ,y_2 )\in \sqrt{4K}\mathcal{S}_{a}}} 1. 
\end{equation}
Applying Lemma \ref{Lmoa} to the inner sum gives
\[
\sum_{\substack{k\leq K\\k \equiv l\, ({\rm mod}\, m)}}\ r_{a}(k) = \sum_{\substack{\alpha_1 , \alpha_2\, ({\rm mod}\, m)\\ f_{a}(\alpha_1 , \alpha_2) \equiv l\, ({\rm mod}\, m)}} \frac{2CK}{am^2}\left(1 + O(a^{-2})\right),
\]
so that
\begin{equation}\label{mo13}
\begin{split}
\sum_{\substack{ k\leq K\\ k\equiv l\, ({\rm mod}\, m)}}\ b_{\mathcal{A}}(k) &= \frac{2CK}{m^2}\sum_{\substack{\beta\, ({\rm mod}\, m)\\ \alpha_1 , \alpha_2\, ({\rm mod}\, m)\\ f_{\beta}(\alpha_1 , \alpha_2) \equiv l\, ({\rm mod}\, m)}}\left(\sum_{\substack{a\in \mathcal{A}\\a \equiv \beta\, ({\rm mod}\, m)}} \left(a^{-1} + O(a^{-3})\right)\right)\, , \\
&= \frac{2CK}{m^3}\left(\frac{1}{2}\log{A} + O(A^{-\frac{1}{2}})\right) \sum_{\substack{\beta\, ({\rm mod}\, m)\\\alpha_1 , \alpha_2\, ({\rm mod}\, m)\\ f_{\beta}(\alpha_1 , \alpha_2) \equiv l\, ({\rm mod}\, m)}} 1\, ,\\
&= \frac{2CK}{m}\left(\frac{1}{2}\log{A} + O(A^{-\frac{1}{2}})\right)\delta^{(m)}(l)\, .
\end{split}
\end{equation}
Combining \eqref{mo11} with \eqref{mo13} gives us
\begin{equation}\label{mo14}
\begin{split}
\Sigma_2 &= - 2C^2\frac{K}{m}\left( (\log{A})^2 + O(A^{-\frac{1}{2}}\log{A})\right)\sum_{l\, ({\rm mod}\, m)} \delta^{(m)}(l)^2\, ,\\
&= - 2C^2K\left( (\log{A})^2 + O(A^{-\frac{1}{2}}\log{A})\right) \delta_{m}(V^{(2)})\, .
\end{split}
\end{equation}

It remains for us to analyze the difficult case  $\Sigma_{1}$. We have 
\begin{equation}\label{mo15}
\begin{split}
\Sigma_{1}=\sum_{k\leq K}b_{\mathcal{A}}^{2}(k) &= \sum_{a_1,\, a_2 \in \mathcal{A}}\sum_{k\leq K} r_{a_1}(k)r_{a_2}(k)\, ,\\
&= \sum_{a\in \mathcal{A}}\sum_{k\leq K}r_{a}^{2}(k) + \sum_{\substack{ a_1,\, a_2 \in \mathcal{A}\\ a_1 \neq a_2}}\ \sum_{k\leq K}r_{a_1}(k)r_{a_2}(k)\, .
\end{split}
\end{equation}
The diagonal term above can be estimated from 
\begin{lemma}\label{Lmo1}\ 
\begin{enumerate}[label=(\alph*).]
\item For $R_{d}$ as in \eqref{mo1b}, we have
\[
\sum_{k\leq K} R_{d}^{2}(k) \ll \frac{K}{\sqrt{d}} + \frac{K\log{K}}{d}\tau(d)\, ,
\]
\item 
\[
\sum_{a\in \mathcal{A}}\sum_{k\leq K}r_{a}^{2}(k) \ll K\log{A}\, .
\]
\end{enumerate}
\end{lemma}
where $\tau(\ )$ is the divisor function, and all implied constants are absolute.
\begin{proof}
Since we are obtaining upper-bounds, we will discard the condition that $s_1 - s_2$ is even in the definition of $R_{d}(k)$. By abuse of notation, we denote this modified counting function by $R_{d}(k)$  in the proof. Part\,(b) follows from Part\,(a) in the same manner that Lemma \ref{Lmoa} follows from Lemma \ref{Lmo0}, and summing over $a\in \mathcal{A}$, giving
\[
\sum_{a\in \mathcal{A}}\sum_{k\leq K}r_{a}^{2}(k) \ll K\log{A}  + \frac{K\log{K}\log{A}}{\sqrt{A}} \ll K\log{A}\, ,
\]
since $ (\log{K})^2 \ll A \ll K^{\ve}$.

 For the proof of Part\,(a), we write ${\bf s}_{j} = (s_j,t_j)$ for $j=1\,, 2$ to get
\[
R_{d}^{2}(k) = \#\left\{ ({\bf s}_1,{\bf s}_2)\, :\ s_{1}^2 -dt_{1}^2 = s_{2}^2 -dt_{2}^2 = k,\  {\bf s}_{j}\in \sqrt{K}\mathcal{S}_{d},\ j=1,\, 2\, \right\}\, ,
\]
so that we have
\begin{equation}\label{mo15a}
\sum_{k\leq K} R_{d}^{2}(k) = \#\left\{ ({\bf s}_{1},{\bf s}_{2}): \ s_{1}^2 -dt_{1}^2 = s_{2}^2 -dt_{2}^2 \leq K,\ {\bf s}_{j}\in \sqrt{K}\mathcal{S}_{d},\ j=1,\, 2\, \right\}\, .
\end{equation}
Now ${\bf s}_{j}\in \sqrt{K}\mathcal{S}_d$ and $s_{j}^2 - dt_{j}^2 \leq K$ imply that
\[
s_1\, , s_2 \ll \sqrt{K}\quad \text{and} \quad t_1\, , t_2 \ll \sqrt{\frac{K}{d}}
\]
Switching the roles of $t_1$ and $t_2$ in \eqref{mo15a} shows that
\[
\sum_{k\leq K} R_{d}^{2}(k)  \leq \#\left\{ ({\bf s}_{1},{\bf s}_{2}) :\, s_{1}^2 +dt_{2}^2 = s_{2}^2 +dt_{1}^2 ,\,  s_j \ll K^{\frac{1}{2}}, t_j \ll \left(\frac{K}{d}\right)^{\frac{1}{2}},\ j=1,\, 2\, \right\} .
\]
Since the forms are now positive definite, we apply Theorem 2 of \cite{BG06}, which gives the estimate in the Lemma.\qed
\end{proof}

The inner sum in the off-diagonal term in \eqref{mo15} can be analyzed by using Kloostermann's method (see \cite{HB96} and \cite{Ni} for a modern treatment and uniformity with our parameters) to give, for $a_1 \neq a_2$
\begin{equation}\label{mo17}
\sum_{k\leq K}r_{a_1}(k)r_{a_2}(k) = \delta^{(K)}_{\infty}(a_1, a_2)\,  \delta_{\textrm{fin}}(a_1,a_2)\  + O(K^{1-\ve_{0}})\, ,
\end{equation}
for some $\ve_{0}>0$. 

Here, $\delta^{(K)}_{\infty}(a_1,a_2)$ is the singular integral and $\delta_{\textrm{fin}}(a_1,a_2)=\prod_{p<\infty}\delta_{p}(a_1, a_2)$, where $ \delta_{p}(a_1,a_2)$ is the singular series, both associated to the equation
\begin{equation}\label{mo18}
V_{a_1,a_2}:  \quad f_{a_1}(x_1,x_2) = f_{a_2}(y_1,y_2)\, ,
\end{equation}
with
\begin{equation}\label{mo20}
\delta_{p}(a_1,a_2)= \lim_{\nu \to \infty} \frac{\# V_{a_1,a_2}\left(\mathbb{Z}/p^{\nu}\mathbb{Z}\right)}{p^{3\nu}}\, .
\end{equation}
For the singular integral, let  $\sqrt{4K}\mathcal{S}_{a_1,a_2} = \sqrt{4K}\mathcal{S}_{a_1} \times \sqrt{4K}\mathcal{S}_{a_2}$ and let $\chi_{a_1,a_2}$ be its characteristic function (here we abuse notation by writing $\mathcal{S}_{a_j}$ instead of $\mathcal{S}_{D_{a_j}}$). Then,
\begin{equation}\label{mo19}
\begin{split}
\delta^{(K)}_{\infty}(a_1,a_2) &= \int_{-\infty}^{\infty}\left[\int_{\mathbb{R}^4} \chi_{a_1,a_2}({\bf x},{\bf y})\, e\left(t(g_{a_1}({\bf x})-g_{a_2}({\bf y})+a_{1}^2 -a_{2}^2)\right)\ \text{d}{\bf x}\text{d}{\bf y}\right]\ \text{d}t\, .\\
&= \lim_{\epsilon \to 0}\frac{1}{\epsilon}.\text{Vol}\left(({\bf x},{\bf y}) \in \sqrt{4K}\mathcal{S}_{a_1,a_2}:  \vert g_{a_1}({\bf x})-g_{a_2}({\bf y})+a_{1}^2 -a_{2}^2 \vert < \epsilon \right)\, ,\\
&= \frac{4C^2K}{a_1a_2}\left[1 + O(A^{-1})\right]\, .
\end{split}
\end{equation}
Hence, from \eqref{mo15} and \eqref{mo17} we have 
\begin{equation}\label{mo19a}
\Sigma_{1} = 4C^2K\left(1 + O(A^{-1})\right) \sum_{\substack{ a_1,\, a_2 \in \mathcal{A}\\ a_1 \neq a_2}}\  \frac{1}{a_1a_2}\delta_{\textrm{fin}}(a_1,a_2) + O\left( K\log{A} + K^{1 - \ve_{0}}A^2\right)\, .
\end{equation}
To analyse the main term in \eqref{mo19a}, we replace $\delta_{\textrm{fin}}(a_1,a_2)$ with $\delta^{(m)}(a_1,a_2)$, where 
\begin{equation}\label{mo22}
\delta^{(m)}(a_1,a_2) := \frac{\# V_{a_1,a_2}\left(\mathbb{Z}/m\mathbb{Z}\right)}{m^{3}}. 
\end{equation}
The error term in doing so in \eqref{mo19a} has size
\begin{equation}\label{mo23}
\ll  K \sum_{s\geq 1}\sum_{\substack{ a_1,a_2 \in \mathcal{A}\\ a_1 \neq a_2\\ \text{gcd}\left(D_{a_1},D_{a_2}\right)=s}}\  \frac{\big{\vert}\delta_{\textrm{fin}}(a_1,a_2) - \delta^{(m)}(a_1,a_2)\big{\vert}}{a_1a_2} .
\end{equation}
By Appendix B, $ \delta^{(m)}(a_1,a_2)$ is suitably close to $\delta_{\textrm{fin}}(a_1,a_2)$ unless $s= \text{gcd}\left(D_{a_1},D_{a_2}\right)$  is in the set
\[
S_{\mathcal{A},m} := \left\{s: s= p_{1}^{e_1}\ldots p_{t}^{e_t}\ \text{with either}\ e_j \geq B \ \text{for some}\ j,\ \text{or}\ p_j >L\ \text{for\ some}\  j \right\}.
\]
Moreover, for such $a_1$ and $a_2$, the difference $\big{\vert}\delta_{\textrm{fin}}(a_1,a_2) - \delta^{(m)}(a_1,a_2)\big{\vert}$ is $O\left(\tau(s)\right)$, with $\tau(.)$ the divisor function. Hence the contribution  to \eqref{mo23}  from these  is at most
\begin{equation}\label{mo24}
\sum_{s\in S_{\mathcal{A},m}}\tau(s)  \sum_{\substack{ a_1, a_2 \in \mathcal{A}\\ \text{gcd}(D_{a_1},D_{a_2})=s}}\ \frac{1}{a_1 a_2}. 
\end{equation}
Since $D_{a}\equiv 0$ (mod $k$) occurs only if $a \equiv \pm 2$ (mod $\frac{s}{2^{\alpha}}$) with $0\leq \alpha \leq 3$, the sums in \eqref{mo24} above are bounded by 
\begin{equation}\label{mo25}
(\log{A})^2\sum _{s\in S_{\mathcal{A},m}}\ \frac{\tau(s)}{s^2} \ll (\log{A})^2\ \min\left(L,2^B\right)^{-\frac{1}{2}}\, ,
\end{equation}
because $s\in S_{\mathcal{A},m}$ implies $s\geq \min(L,2^B)$ and $a_{1},\ a_{2} \neq \pm 2$.

Next, for $s \notin S_{\mathcal{A},m}$,  we write
\[
\delta_{\textrm{fin}}(a_1,a_2) = \prod_{p\leq L}\delta_{p}(a_1,a_2)\ \prod_{p>L}\delta_{p}(a_1,a_2)\, .
\]
Recall from Prop. \ref{LA4} that $\delta_{p}(a_1,a_2) = 1 + O(p^{-2})$ if $p\nmid D_{a_1}D_{a_2}\left(D_{a_1}-D_{a_2}\right)$ when $p \geq 3$.  We denote these primes by $\mathcal{P}^{(1)}$ and include $p=2$ in this set, and denote the remaining finite set of  primes by $\mathcal{P}^{(2)}$. We decompose $\mathcal{P}^{(2)}$ further into  $\mathcal{P}^{(3)} = \left\{p\geq 3:  p\vert \text{gcd}\left(D_{a_1},D_{a_2}\right)\  \right\}$ and its complement. Then we write
\[
\begin{split}
\prod_{p>L}\delta_{p}(a_1,a_2) &= \prod_{\substack{p\in \mathcal{P}^{(1)}\\p>L}}\ \delta_{p}(a_1,a_2) \prod_{\substack{ p\in \mathcal{P}^{(2)}\\p>L}}\delta_{p}(a_1,a_2)\, , \\
&= \prod_{\substack{ p\in \mathcal{P}^{(2)}\\p>L}}\delta_{p}(a_1,a_2)\, \left(1 + O\left(\frac{1}{L}\right)\right)\, .
\end{split}
\]
Since $s \notin S_{\mathcal{A},m}$, if $p>L$ we have (using Prop. \ref{LA4} again) 
\[
\log{\prod_{\substack{ p\in \mathcal{P}^{(2)}\\p>L}}\delta_{p}(a_1,a_2)} =  \sum_{\substack{p\in \mathcal{P}^{(2)}\\p>L}}  \frac{c_{p}}{p} +\ O\left( \frac{1}{L}\right)\, ,
\]
with coefficients $c_p$ satisfying $\vert c_p\vert \leq 1$. Since $a_1 \neq \pm a_2$ and $a_j \neq \pm 2$, the set $\mathcal{P}^{(2)}$ has the bound  $\text{card}(\mathcal{P}^{(2)}) \ll \frac{\log{A}}{\log\log{A}}$, as it contains those primes dividing $D_{a_1}$ or $D_{a_2}$ or $\left(D_{a_1}-D_{a_2}\right)$. Hence the sum over $\mathcal{P}^{(2)}$ above is bounded by $\frac{\log{A}}{L\log\log{A}} \ll \Phi(A)^{-1}$.  Thus, for $s \notin S_{\mathcal{A},m}$ we have
\[
\delta_{\textrm{fin}}(a_1,a_2) = \prod_{p\leq L}\delta_{p}(a_1,a_2)\times\left(1 + O\left(\frac{1}{\Phi(A)}\right)\right)\, .
\]

To analyse this further, we write $\delta^{(m}(a_1,a_2)=\prod_{p\leq L}\ \delta^{(m)}_{p}(a_1,a_2) $. Then one has 
 \begin{equation}\label{LB1}
 \delta_{p}(a_1,a_2) = 1 + \sum_{l=1}^{\infty}N_{l}(a_1,a_2) \quad \text{and} \quad  \delta^{(m)}_{p}(a_1,a_2) = 1 + \sum_{l=1}^{B}N_{l}(a_1,a_2),
 \end{equation}
  where
\begin{equation}\label{LB1a}
 N_{l}(a_1,a_2) = p^{-4l}\sideset{}{^*}\sum_{b\, ({\rm mod}\, p^l)}\, \sum_{{\bf x},\,{\bf y}\,({\rm mod}\,p^l)}\ e\left(\frac{g_{a_1}({\bf x})- g_{a_2}({\bf y}) +D_{a_1} - D_{a_2}}{p^l}b\right)\, .
  \end{equation}
 Then for $s \notin S_{\mathcal{A},m}$, one has by Prop \ref{LA40} that $ \delta_{p}(a_1,a_2) = \delta^{(m)}_{p}(a_1,a_2) + O(p^{-B})$. It  follows that the contribution to  \eqref{mo23} is 
 \[
 \begin{split}
 K \sum_{\substack{s\geq 1 \\  s \notin S_{\mathcal{A},m}}}\sum_{\substack{ a_1,\ a_2 \in \mathcal{A}\\ a_1 \neq a_2\\ \text{gcd}\left(D_{a_1},D_{a_2}\right)=s}} & \frac{\big{\vert}\delta_{\textrm{fin}}(a_1,a_2) - \delta^{(m)}(a_1,a_2)\big{\vert}}{a_1a_2}\\ & \ll K\min\left(\Phi(A),2^B\right)^{-1}\sum_{\substack{ a_1,\ a_2 \in \mathcal{A}\\ a_1 \neq a_2}}\  \frac{\big{\vert}\delta^{(m)}(a_1,a_2)\big{\vert}}{a_1a_2},\\ &\ll K\min\left(\Phi(A),2^B\right)^{-1}(\log{A})^2\, ,
\end{split}\]
using $\delta^{(m)}(a_1,a_2) \ll \tau(\text{gcd}(a_1,a_2))$. 

We choose $\Phi(A) = \frac{1}{2}\sqrt{\frac{\log\log A}{\log\log\log A}}$ so that $ \Phi(A)^2 \ll 2^B = o(L)$. Substituting into \eqref{mo19a} gives us 

\begin{equation}\label{no1}
\Sigma_1 = 4C^2K\sum_{\substack{ a_1,\ a_2 \in \mathcal{A}\\ a_1 \neq a_2}}\  \frac{\delta^{(m)}(a_1,a_2)}{a_1a_2} + O\left(K\Phi(A)^{-1} (\log{A})^2\right)\, .
\end{equation}
Since $\delta^{(m)}(a_1,a_2)$ is periodic modulo $m$, the sum in \eqref{no1} can be analyzed as  in \eqref{mo9}, giving
\begin{equation}\label{no2}
\begin{split}
\Sigma_1 &= 4C^2K\sum_{\substack{ \alpha_1 \,({\rm mod}\, m)\\ \alpha_2\, ({\rm mod}\, m) }}\,  \delta^{(m)}(\alpha_1,\alpha_2)\frac{\left(\frac{1}{2}\log{A}\right)^2}{m^2} + O\left(K\Phi(A)^{-1} \left(\log{A}\right)^2\right),\\
&= C^2K\left(\log{A}\right)^2\delta_{m}\left(V^{(2)}\right) + O\left(K\Phi(A)^{-1} \left(\log{A}\right)^2\right)\, .
\end{split}
\end{equation}
Combining \eqref{mo9}, \eqref{mo14} and \eqref{no1} into \eqref{mo8} gives us the key cancellation and hence the estimate on the variance $V(K)$.
\begin{proposition}\label{Pmo1} Let $K \rightarrow \infty$, let $\mathcal{A}$ be the interval $[\sqrt{A},\ A]$ with $A$  satisfying $(\log K)^2 < A \ll_{\ve}K^{\ve}$, with $\ve >0$ sufficiently small. Then with $\Phi(A) = \frac{1}{2}\sqrt{\frac{\log\log A}{\log\log\log A}}$, we have 
\[
\sum_{k\leq K} \left[b_{\mathcal{A}}(k) - C(\log A)\delta^{(m)}(k)\right]^2 \ll K\Phi(A)^{-1} \left(\log{A}\right)^2\  .
\]
\end{proposition}
\begin{Note}
\label{Rmo1} One can remove the auxiliary parameter $B$  in the Proposition above with
\[
\delta^{(m)}(k) = \prod_{p\leq L}\delta_{p}(k) +O(2^{-B}),
\] 
as follows. From \eqref{A8} and \eqref{mo6}, we have $\delta(V_k)=\prod_{p<\infty}\delta_{p}(k)$ with $\delta_{p}(k)=1+\sum_{l\geq 1}N_{l}(k)$, where $N_{l}(k)$ is given in \eqref{A7} and  are all evaluated in the Appendix . Similarly, one can show that $\delta^{(m)}(k)=\prod_{p\leq L}\delta^{(m)}_{p}(k)$ with $\delta^{(m)}_{p}(k)=1+\sum_{1\leq l \leq B}N_{l}(k)$ . Then, it follows from Prop. \ref{LA5} that $\delta^{(m)}_{p}(k) = \delta_{p}(k) + O(p^{-B})$. Applying Prop \ref{LA3} then gives the result. \qed
\end{Note}

%%%%%%%%%%%%%%%%%%%%%%%%%%%%%%%%%%%%%%%%%%%%%%%%%%%%%
\subsection{Lower bound for $\delta^{(m)}(k)$ for most admissible $k$'s}
\label{lowerbound}\ 

To complete the proof  of Theorem \ref{Thm2}(ii), we need to estimate, for $\ve>0$
\begin{equation}\label{no3}
\big{|}\big\{0\leq k\leq K\, :\, k \ \text{admissible},\ \delta^{(m)}(k)<\ve\big\}\big{|}\, .
\end{equation}
By Props. \ref{LA4} and \ref{Lb6} in the Appendix, and Remark \ref{Rmo1} , we can write
\begin{equation}\label{no4}
\delta^{(m)}(k) = \prod_{p\leq L}\left(1 + N_{1,p}(k) + \mathcalorig{C}_{p}(k)\right) + o(1), 
\end{equation}
where we indicate the dependence of $p$ in the definition of $N_{l}$, and  with $\mathcalorig{C}_{p}(k)$ coming from the $N_{l,p}$'s with $l\geq 2$. Since we are assuming that $k$ is admissible, we can ignore the primes $p=2$ and $p=3$ since then these local factors are bounded from below. For $p\geq 5$, the problematic case of $\mathcalorig{C}_{p}(k)$ in \eqref{no4} is, by Prop. \ref{LA5}, of the form $4^{\beta}p^{-1} + O\left(p^{-2}\right)$. So, up to $O\left(p^{-2}\right)$, which can be ignored for our purposes of bounding $\delta^{(m)}(k)$ from below, we have that
\begin{align}\label{no5}
\begin{split}
\delta^{(m)}(k) &\gg \prod_{p\leq L}\left(1 + N_{1,p}(k) + O\left(\frac{1}{p^2}\right)\right) + o(1)\, ,\\
&\gg \prod_{p\leq L}\left(1 + \frac{\chi(k-4)\left(3+\chi(k)\right)}{p} \right)\, ,
\end{split}
\end{align}
where $\chi$ is the Legendre symbol modulo $p$. Hence
\begin{align}\label{no6}
\begin{split}
\left[\delta^{(m)}(k)\right]^{-1} &\ll \prod_{p\leq L}\left(1 - \frac{\chi(k-4)\left(3+\chi(k)\right)}{p} \right), \\
&= \sum_{n\leq M}\ \frac{\mu(n)A(k,n)}{n} \, ,
\end{split}
\end{align}
where $A(k,n)= A(k,p_1)\ldots A(k,p_l)$ if $n=p_1\ldots p_l$\ , $M=\left(\prod_{p\leq L}p\right) \leq m \ll K^{\ve}$, and 
\begin{equation}\label{no7}
A(k,p) = \left\{ \begin{array}{ll}
  \chi(k-4)\left(3+\chi(k)\right)& \text{if}\ \ p\geq 5\,  , \\
 \quad\quad0 & \text{otherwise}\  . 
\end{array} \right. 
\end{equation}
Since $A(k,n)$ as a function of $k$ is periodic of period $n$, we have
\begin{equation}\label{no8}
\sum_{k\leq K}A(k,n) = \frac{K}{n}\sum_{k\, ({\rm mod}\, n)} A(k,n) + O(n)\, .
\end{equation}
By multiplicativity, the completed sum 
\[
\sum_{k\, ({\rm mod}\, n)} A(k,n) = \prod_{p|n}\left(\sum_{k\, ({\rm mod}\, p)} A(k,p)\right) = \mu(n) , 
\]
since $\sum_{k\, ({\rm mod}\, p)} \chi(k-4) =0$ and  $\sum_{k\, ({\rm mod}\, p)} \chi(k-4)\chi(k) =-1$. Hence
\[
\sum_{k\leq K}A(k,n) = \frac{\mu(n)}{n}K + O(n),
\]
so that by \eqref{no6} we have 
\[
\sum_{\substack{k\leq K\\ k\, \text{admissible}}} \left[\delta^{(m)}(k)\right]^{-1} \ll K\left(\sum_{n}\frac{1}{n^2}\right) + M \ll K. 
\]
Hence, it follows that for any $\ve >0$
\begin{equation}\label{no9}
 \big{|}\big\{0\leq k\leq K: \ k \ \text{admissible}\, ,\ \delta^{(m)}(k)<\ve\big\}\big{|} \ll \ve K. 
 \end{equation}
 Finally, combining \eqref{no9} with the variance estimate in Prop. \ref{Pmo1} gives Theorem \ref{Thm2}(ii).\qed
%%%%%%%%%%%%%%%%%%%%%%%%%%%%%%%%%%%%%%%%%%%%%%%%%%%%%%%%%%%
%%%%%%%%%%%%%%%%%%%%%%%%%%%%%%%%%%%%%%%%%%%%%%%%%%%%%%%%%%%

%%%%%%%%%%%%%%%%%%%%%%%%%%%%%%%%%%%%%%%%%%%%%%%%%%%%%%%%%%%%%%%%%%%%%%
%%%%%%%%%%%%%%%%%%%%%%%%%%%%%%%%%%%%%%%%%%%%%%%%%%%%%%%%%%%%%%%%%%%%%%

\section{Computations}
\label{sec9}
We computed all the Hasse failures (HF) for positive $k \leq K$  with $K$ about  564 million and present some data associated with this. 

By definition there are no HF's in the congruence classes $3$\,(mod\,4) and $\pm3$\,(mod\,9). We tabulate in Table \ref{table1}  the percentages in some other congruence classes: more precisely this is the quantity  $\lambda =100\times \#\{\text{ HF's in a congruence class}\}/\#\{\text{ HF's }\}$.

\begin{center}
\begin{tabular}{p{4cm}p{4cm}}
\begin{tabular}{|c|c|}
\hline
Class (mod 4)   & $\lambda \%$ \\
\hline
0&  30.6733\\
\hline
1&  28.1317\\
\hline
2&  41.195\\
\hline
\end{tabular}

&
\begin{tabular}{|c|c|}
\hline
Class (mod 3)   & $\lambda \%$ \\
\hline
0& 42.2078\\
\hline
1& 43.2872\\
\hline
2&  14.5051\\
\hline
\end{tabular}
\end{tabular}
\end{center}

\begin{center}
\begin{tabular}{p{4cm}p{4cm}}
\begin{tabular}{|c|c|}
\hline
Class (mod 9) &   $\lambda \%$ \\
\hline
0& 42.2078\\
\hline
1&  19.023\\
\hline
4&  4.462\\
\hline
7&  19.802\\
\hline
\end{tabular}
&
\begin{tabular}{|c|c|}
\hline
Class (mod 9) &   $\lambda \%$ \\
\hline
2&  4.831\\
\hline
5&  4.833\\
\hline
8&  4.841\\
\hline
\end{tabular}
\vspace{15pt}
\end{tabular}

\captionof{table}{Percentages of Hasse failures in congruence classes.}\label{table1}
\end{center}

The number of admissible $k$'s (see Sec.\,\ref{sec3}) in the interval $[1, K]$ we denote by $\mathcalorig{A}(K)$, and is  asymptotically $\frac{7}{12}K$. The admissibles consist of the exceptional $k$'s, of which there are $O(\frac{K}{\sqrt{\log{K}}})$ members, the generic $k$'s consisting of the HF's and the generic  $k$'s with $\mathfrak{h}(k)>0$. For $K\geq 5$, let $\mathcalorig{A}_{\scriptscriptstyle HF}(K)$ denote the number of HF's in the interval $[5, K]$. By the arguments in Sec.\,\ref{sec6}, we know that $\mathcalorig{A}_{\scriptscriptstyle HF}(K) \rightarrow \infty$ with $K$, and more precisely $\mathcalorig{A}_{\scriptscriptstyle HF}(K) \gg_{\epsilon} K^{\frac{1}{2} -\epsilon}$  for any $\epsilon >0$. While we do not know the exact order of $\mathcalorig{A}_{\scriptscriptstyle HF}(K)$, Theorem \ref{Thm2}(ii) shows that it is $o(K)$, and we consider this question here computationally, for which we  compare $\mathcalorig{A}_{\scriptscriptstyle HF}(K)$ with $\mathcalorig{A}(K)$. There are two possible models to consider, namely (1) $\mathcalorig{A}_{\scriptscriptstyle HF}(K) \sim \mathcalorig{A}(K)^\theta$ for some $0 < \theta <1$ or (2) $\mathcalorig{A}_{\scriptscriptstyle HF}(K) \sim \mathcalorig{A}(K)/\log{\mathcalorig{A}(K)}^\theta$.
To check these cases, we compute $\log{\mathcalorig{A}_{\scriptscriptstyle HF}(K)} / \log{\mathcalorig{A}(K)}$ with $\mathcalorig{A}(K)= \frac{7}{12}K$.

For $K \sim  564\times10^6$ put $L=\log K$, and let 
\begin{equation}\label{9.1}
 f(K)= 0.887516\, - 8.06653\ L^{-2} - 21.8923\ L^{-3}+ 2.38097\ L^{-4} + 3.35656\ L^{-5}.
\end{equation}
Using  $f(K)$ as a predictor for $\frac{\log{\mathcalorig{A}_{\scriptscriptstyle HF}(K)}}{ \log{\mathcalorig{A}(K)}}$ we compare $\mathcalorig{A}_{\scriptscriptstyle HF}(K)$ with $\mathcalorig{P}(K)$,  the integer part of $\mathcalorig{A}(K)^{f(K)}$ and tabulate also the relative percentage error $\mathcalorig{E}(K) = 100\times \frac{\mathcalorig{A}_{\scriptscriptstyle HF}(K) - \mathcalorig{P}(K)}{\mathcalorig{P}(K)}$, in Table \ref{table4},  page \pageref{table4}.  The $K$'s  given  are each a multiple of 100,800 but otherwise  is a  sample. The data was gathered by subdividing the  interval $[1,K]$ into subintervals of length 100,800 and computing $\mathcalorig{A}_{\scriptscriptstyle HF}(K)$ at the endpoints. 

We first show, in Fig. \ref{img2}, the graph of the ``percentages of the Hasse failures'', namely the quantity $100\times \frac{\mathcalorig{A}_{\scriptscriptstyle HF}(K)}{\mathcalorig{A}(K)}$, together with the prediction $\mathcalorig{P}(K)$. There were a total of $23,298,277$ HF's (with $K= 564,062,446 $) but we plot $5,595$ sample points. Next, Fig. \ref{img3} is a plot of  $\frac{\log{\mathcalorig{A}_{\scriptscriptstyle HF}(K)}}{ \log{\mathcalorig{A}(K)}}$ together with the approximation $f(K)$ in \eqref{9.1}. In both graphs, the approximation  curves  and data points are too close to show up as separate curves, and so we plot only the data points. We also indicate in Table \ref{table2}, the quality of the approximations in   Fig. \ref{img4} and Fig. \ref{img5}.

\begin{figure}[!ht]
\centering
  \includegraphics[width=.75\linewidth]{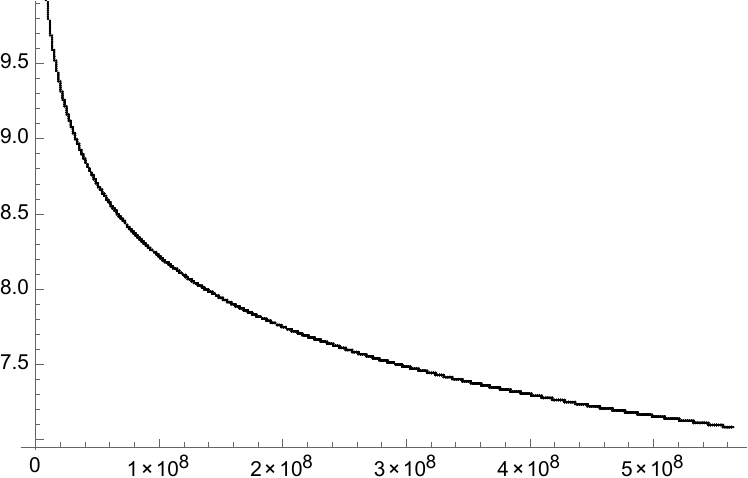}
\caption{Percentages of Hasse failures: $100\times \frac{\mathcalorig{A}_{\scriptscriptstyle HF}(K)}{\mathcalorig{A}(K)}$ .}
\label{img2}
\vspace{2cm}
\end{figure}

 \begin{figure}[!ht]
\centering
  \includegraphics[width=.75\linewidth]{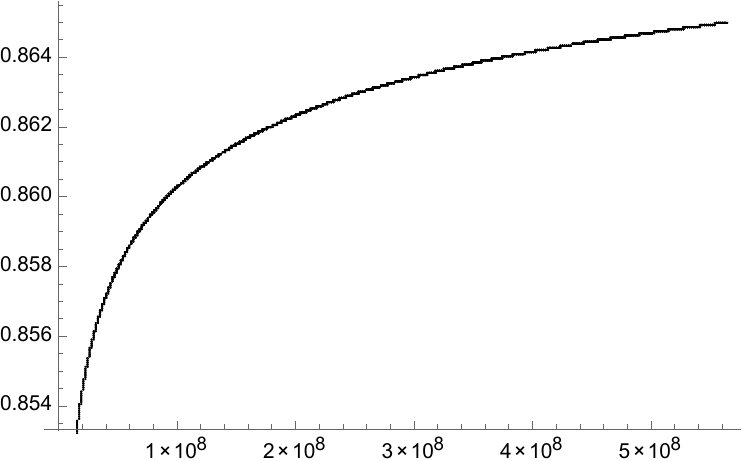}
\caption{Plot of $\frac{\log{\mathcalorig{A}_{\scriptscriptstyle HF}(K)}}{ \log{\mathcalorig{A}(K)}}$.}
 \label{img3}
 \vspace{.5cm}
\end{figure}

\begin{center}
 \begin{tabular}{p{7cm}p{6cm}}
\begin{tabular}{|r l|}
\hline
{}&{}\\
$R^2$ Goodness of Fit: &	0.99999359 \\
Correlation Coefficient: &	0.99999689 \\
Maximum Error: &	0.00015802515 \\ 
Mean Squared Error: &	$6.09792 \times 10^{-11}$ \\
Mean Absolute Error: &	$3.4565103 \times 10^{-6}$ \\
{}&{}\\
\hline
\end{tabular}
%\vspace{.5cm}
\captionsetup{width=\linewidth}
\captionof{table}{Data on approximation:\\ $\frac{\log{\mathcalorig{A}_{\scriptscriptstyle HF}(K)}}{ \log{\mathcalorig{A}(K)}} \sim f(K)$, \ $2\times 10^6 < k < 564 \times 10^6$}\label{table2}
\end{tabular}
\end{center}

\begin{figure}[!ht]
\centering
  \includegraphics[width=.75\linewidth]{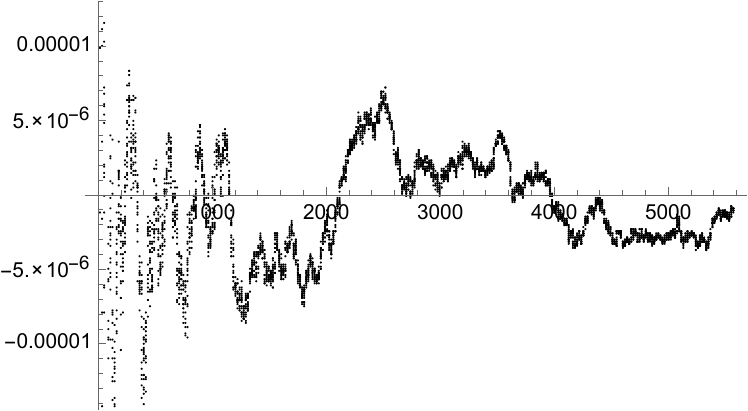}
\caption{Normalised residuals between $\frac{\log{\mathcalorig{A}_{\scriptscriptstyle HF}(K)}}{ \log{\mathcalorig{A}(K)}}$  and prediction $f(K)$: x-axis is \ $K/100800$.}
\label{img4}
\vspace{.5cm}
\end{figure}

\begin{figure}[!ht]
\centering
  \includegraphics[width=.65\linewidth]{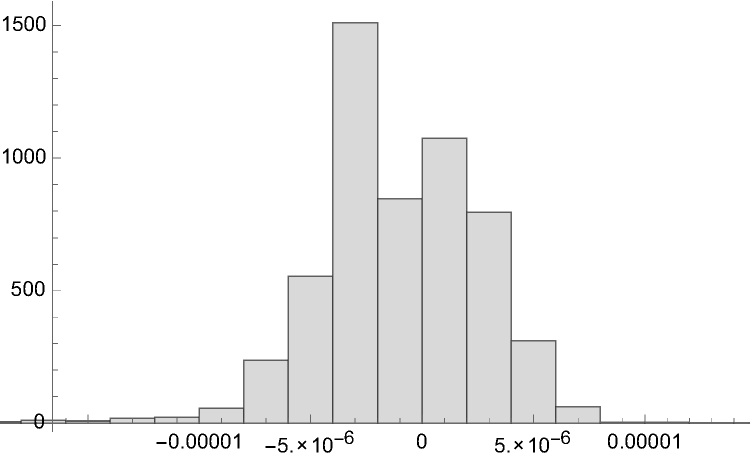}
\caption{Histogram of residuals above.}
  \label{img5}
  \vspace{.5cm}
\end{figure}

\begin{table}[!ht]
\caption{Percentages of Hasse failures.}
\vspace{10pt}
\label{table3}
\centering\small
\begin{tabular}{|c| c| }
\hline
$k \leq K$ &   ${ \%}$  \small{Hasse failures}\\
\hline
{}&{}\\
\ 100,800&12.97620\\
 10,080,000 & 9.84888 \\
 20,160,000 & 9.34943 \\
 30,240,000 & 9.05874 \\
 40,320,000 & 8.85229 \\
 50,400,000 & 8.69513 \\
 60,480,000 & 8.56721 \\
 70,560,000 & 8.45991 \\
 80,640,000 & 8.36626 \\
 90,720,000 & 8.28619 \\
 {}&{}\\
 \hline
 \end{tabular}
 \quad
\begin{tabular}{|c| c|}
\hline
 $k\leq K$ & $ \%$  Hasse failures\\
 \hline
 {}&{}\\
100,800,000 & 8.21313 \\
 110,880,000 & 8.14845 \\
 120,960,000 & 8.08844 \\
 131,040,000 & 8.03345 \\
 141,120,000 & 7.98373 \\
 151,200,000 & 7.93721 \\
 161,280,000 & 7.89398 \\
 171,360,000 & 7.85349 \\
 181,440,000 & 7.81481 \\
 191,520,000 & 7.77902 \\
 {}&{}\\
 \hline
 \end{tabular}
 %\vspace{5pt}
\label{table3a}
\begin{tabular}{|c| c|}
\hline
 $k\leq K$ & $ \%$  Hasse failures\\
 \hline
 {}&{}\\
201,600,000 & 7.74513 \\
 211,680,000 & 7.71296 \\
 221,760,000 & 7.68273 \\
 231,840,000 & 7.65369 \\
 241,920,000 & 7.62577 \\
 252,000,000 & 7.59906 \\
 262,080,000 & 7.57319 \\
 272,160,000 & 7.54791 \\
 282,240,000 & 7.52428 \\
 292,320,000 & 7.50153 \\
 {}&{}\\
 \hline
 \end{tabular}
 \quad
\begin{tabular}{|c |c|}
\hline
 $k\leq K$ & $ \%$  Hasse failures\\
 \hline
 {}&{}\\
302,400,000 & 7.47924 \\
 312,480,000 & 7.45807 \\
 322,560,000 & 7.43758 \\
 332,640,000 & 7.41768 \\
 342,720,000 & 7.39828 \\
 352,800,000 & 7.37978 \\
 362,880,000 & 7.36165 \\
 372,960,000 & 7.34364 \\
 383,040,000 & 7.32663 \\
 393,120,000 & 7.30999  \\
 {}&{}\\
 \hline
 \end{tabular}
 %\vspace{.5cm}
\begin{tabular}{|c| c| }
\hline
$k \leq K$ &   $ \%$  Hasse failures\\
\hline
{}&{}\\
 403,200,000 & 7.29355 \\
 413,280,000 & 7.27757 \\
 423,360,000 & 7.26198 \\
 433,440,000 & 7.24716 \\
 443,520,000 & 7.23274 \\
 453,600,000 & 7.21817 \\
 463,680,000 & 7.20418 \\
 473,760,000 & 7.19059 \\
 {}&{}\\
 \hline
 \end{tabular}
 \quad
\begin{tabular}{|c| c|}
\hline
 $k\leq K$ & $ \%$  Hasse failures\\
 \hline
 {}&{}\\
 483,840,000 & 7.17718 \\
 493,920,000 & 7.16409 \\
 504,000,000 & 7.15137 \\
 514,080,000 & 7.13901 \\
 524,160,000 & 7.12665 \\
 534,240,000 & 7.11460 \\
 544,320,000 & 7.10295 \\
 554,400,000 & 7.09159 \\
 {}&{}\\
 \hline
\end{tabular}
 
\end{table}

In Table \ref{table3}, we provide a sample of the percentages of the Hasse failures. The data in Table \ref{table3}  suggests that $$\mathcalorig{A}_{ \scriptscriptstyle HF}(K) \sim C\  K^{0.8875...  +\ {\text o}(1)}$$ for some constant $C$, at least for $K$ in this range. The error is smaller than $0.1 \%$ for $K \geq 10^7$ and gets better for larger values of $K$. This is illustrated in Table \ref{table4}.

We could use $f(K) = 1 - \frac{C}{\log{K}} + \ldots $ to test for the ``positive proportion'' case (which we know is false), and use $f(K) = 1 - \frac{C \log\log{K}}{\log{K}} + \ldots $ for case (2). For our range of $K$, $\log\log K$ is essentially constant and we will not be able to distinguish between these cases satisfactorily. Testing gave a candidate function whose error in approximating is closer to $1 \%$, a result much weaker than the one for case (1) above. Moreover, the graph of the normalised residuals (not shown) was far from random.

We do not come to any firm conclusions from these computations as our range of $k$'s may well be too small (since $\log K$ does not grow fast enough), and because the exponent $0.8875...$ is quite close to 1. However, it is still a bit of a surprise how well columns 2 and 3 match in Table \ref{table4}.

We next consider a slightly different calculation. Recall that we  subdivided our large interval into subintervals of length $h= 100,800$ (this value was chosen to be a nice integer slightly larger than $\sqrt{K} =23750$). It is interesting to look at the distribution of HF's within these subintervals. We plot $\frac{1}{h}\left(\mathcalorig{A}_{\scriptscriptstyle HF}((l+1)h) -\mathcalorig{A}_{\scriptscriptstyle HF}(lh)\right)  $, the average number of HF's in the $l$-th subinterval, with $lh\leq K$. The  curve in the graph is an approximation, given by $g(x)=0.2353 x^{-0.0908047} - 46.7396 x^{-0.661084}$ (there was no effort made to find ``optimal'' coefficients in the  approximation). The residuals appear random, with the histogram fitting the normal distribution $N(-0.0016, 0.0018)$ quite closely.

\begin{table}[!ht]
\caption{Comparision between Hasse failure count $\mathcalorig{A}_{\scriptscriptstyle HF}(K)$ and the model $\mathcalorig{P}(K)$.}
\vspace{10pt}
\label{table4}
\resizebox{6cm}{!}{
\begin{minipage}[t]{0.70\textwidth}
\centering
 \begin{tabular}{|cccl|}
\hline\hline
$k$ & $\mathcalorig{A}_{\scriptscriptstyle HF}(K)$& $\mathcalorig{P}(K)$ & error: $\mathcalorig{E}(K)$\%\\
\hline
 6,552,000 & 388,485 & 388,474 & { }0.00279494 \\
 13,104,000 & 738,402 & 738,476 & -0.0100959 \\
 19,656,000 & 1,074,038 & 1,074,075 & -0.00351784 \\
 26,208,000 & 1,400,385 & 1,400,458 & -0.00526837 \\
 32,760,000 & 1,720,203 & 1,720,067 & { }0.0078529 \\
 39,312,000 & 2,034,145 & 2,034,330 & -0.00913502 \\
 45,864,000 & 2,343,944 & 2,344,184 & -0.0102605 \\
 52,416,000 & 2,650,338 & 2,650,290 & { }0.0017743 \\
 58,968,000 & 2,952,994 & 2,953,142 & -0.00502773 \\
 65,520,000 & 3,253,233 & 3,253,119 & { }0.00349665 \\
 72,072,000 & 3,550,279 & 3,550,523 & -0.00689091 \\
 78,624,000 & 3,845,160 & 3,845,601 & -0.0114887 \\
 85,176,000 & 4,138,458 & 4,138,557 & -0.00241157 \\
 91,728,000 & 4,429,888 & 4,429,563 & { }0.00732315 \\
 98,280,000 & 4,718,612 & 4,718,766 & -0.00326508 \\
 104,832,000 & 5,006,091 & 5,006,291 & -0.00401399 \\
 111,384,000 & 5,292,241 & 5,292,251 & -0.000204305 \\
 117,936,000 & 5,576,772 & 5,576,742 & { }0.00052202 \\
 124,488,000 & 5,859,223 & 5,859,851 & -0.0107229 \\
 131,040,000 & 6,140,768 & 6,141,654 & -0.0144272 \\
 137,592,000 & 6,421,657 & 6,422,220 & -0.00876724 \\
 144,144,000 & 6,701,189 & 6,701,611 & -0.00630533 \\
 150,696,000 & 6,979,137 & 6,979,885 & -0.0107173 \\
 157,248,000 & 7,256,456 & 7,257,091 & -0.00876333 \\
 163,800,000 & 7,532,631 & 7,533,279 & -0.00860614 \\
 170,352,000 & 7,807,978 & 7,808,490 & -0.00656096 \\
 176,904,000 & 8,082,302 & 8,082,764 & -0.00572446 \\
 183,456,000 & 8,355,009 & 8,356,139 & -0.0135256 \\
 190,008,000 & 8,627,950 & 8,628,647 & -0.0080886 \\
 196,560,000 & 8,899,431 & 8,900,322 & -0.0100164 \\
 203,112,000 & 9,170,775 & 9,171,192 & -0.00455085 \\
 209,664,000 & 9,440,833 & 9,441,285 & -0.00478836 \\
 216,216,000 & 9,710,721 & 9,710,626 & { }0.000974181 \\
 222,768,000 & 9,979,756 & 9,979,240 & { }0.00516595 \\
 229,320,000 & 10,247,890 & 10,247,150 & { }0.00722142 \\
 235,872,000 & 10,515,262 & 10,514,376 & { }0.00842309 \\
 242,424,000 & 10,781,980 & 10,780,939 & { }0.00964957 \\
 248,976,000 & 11,047,893 & 11,046,859 & { }0.0093601 \\
 255,528,000 & 11,313,674 & 11,312,152 & { }0.0134518 \\
 262,080,000 & 11,577,887 & 11,576,836 & { }0.00907272 \\
 268,632,000 & 11,841,388 & 11,840,928 & { }0.00388283 \\
 275,184,000 & 12104,565 & 12,104,442 & { }0.00101294 \\
 \hline\hline
 \end{tabular}
 \end{minipage}}
 \hspace{.4cm}
\resizebox{6cm}{!}{
\begin{minipage}[t]{0.70\linewidth}
\begin{tabular}{|cccl|}
\hline\hline
$K$ & $\mathcalorig{A}_{\scriptscriptstyle HF}(K)$& $\mathcalorig{P}(K)$ & error: $\mathcalorig{E}(K)$\%\\
\hline
281,736,000 & 12,367,646 & 12,367,393 &  { }0.00203978 \\
 288,288,000 & 12,630,282 & 12,629,796 &  { }0.00384661 \\
 294,840,000 & 12,892,179 & 12,891,662 &  { }0.00400262  \\
 301,392,000 & 13,153,376 & 13,153,006 & { }0.00280671 \\
 307,944,000 & 13,414,178 & 13,413,839 & { }0.00252140\\
 314,496,000 & 13,674,773 & 13,674,173 & { }0.00438484 \\
 321,048,000 & 13,934,649 & 13,934,018 & { }0.00452291 \\
 327,600,000 & 14,194,163 & 14,193,386 & { }0.00547096 \\
 334,152,000 & 14,452,782 & 14,452,286 & { }0.00342712 \\
 340,704,000 & 14,711,231 & 14,710,729 & { }0.00341123 \\
 347,256,000 & 14,969,227 & 14,968,723 & { }0.00336506 \\
 353,808,000 & 15,227,250 & 15,226,278 & { }0.00638342 \\
 360,360,000 & 15,484,481 & 15,483,402 & { }0.00696817 \\
 366,912,000 & 15,740,411 & 15,740,103 & { }0.00195186 \\
 373,464,000 & 15,996,468 & 15,996,391 & { }0.00048054 \\
 380,016,000 & 16,252,525 & 16,252,271 & { }0.00155736 \\
 386,568,000 & 16,508,096 & 16,507,753 & { }0.00207456 \\
 393,120,000 & 16,763,273 & 16,762,843 & { }0.00256364 \\
 399,672,000 & 17,017,822 & 17,017,548 & { }0.00160994 \\
 406,224,000 & 17,271,602 & 17,271,874 & -0.00157807 \\
 412,776,000 & 17,525,323 & 17,525,829 & -0.00288924 \\
 419,328,000 & 17,778,565 & 17,779,418 & -0.00480173 \\
 425,880,000 & 18,031,595 & 18,032,648 & -0.00584339 \\
 432,432,000 & 18,284,841 & 18,285,525 & -0.00374203 \\
 438,984,000 & 18,537,717 & 18,538,054 & -0.00181802 \\
 445,536,000 & 18,790,012 & 18,790,240 & -0.00121647 \\
 452,088,000 & 19,041,406 & 19,042,090 & -0.00359348 \\
 458,640,000 & 19,292,457 & 19,293,608 & -0.00596736 \\
 465,192,000 & 19,543,623 & 19,544,799 & -0.00602086 \\
 471,744,000 & 19,794,451 & 19,795,669 & -0.00615548 \\
 478,296,000 & 20,045,181 & 20,046,222 & -0.00519473 \\
 484,848,000 & 20,295,253 & 20,296,462 & -0.00596103 \\
 491,400,000 & 20,545,280 & 20,546,395 & -0.00542974 \\
 497,952,000 & 20,794,925 & 20,796,024 & -0.00528918 \\
 504,504,000 & 21,044,290 & 21,045,355 & -0.00506101 \\
 511,056,000 & 21,293,257 & 21,294,390 & -0.00532190 \\
 517,608,000 & 21,541,991 & 21,543,134 & -0.00530773 \\
 524,160,000 & 21,790,444 & 21,791,591 & -0.00526622 \\
 530,712,000 & 22,038,418 & 22,039,765 & -0.00611405 \\
 537,264,000 & 22,286,350 & 22,287,659 & -0.00587756 \\
 543,816,000 & 22,534,130 & 22,535,278 & -0.00509679 \\
 550,368,000 & 22,782,046 & 22,782,624 & -0.00254092 \\
 \hline\hline
\end{tabular}%
\end{minipage}}
\end{table}

Finally we include data on the distribution of the number of orbits $\mathfrak{h}(k)$ with generic $k \leq  K$ where $K=10^7$. A sample of fundamental sets  together with the corresponding class numbers $\mathfrak{h}(k)$ obtained using  Theorem \ref{Thm1}(i) is given in Table \ref{funsettable}. The data on the distribution of these class numbers is given in Table \ref{table5} and Fig. \ref{img10}. Here, $\mathfrak{n}(h)=\mathfrak{n}_{K}(h)$ is the number of occurrences of $h=\mathfrak{h}(k)$ with $k $ running through generic integers in $[1,K]$. Our count also includes the number of Hasse failures, denoted by $\mathfrak{n}(0)$.  Since $\mathfrak{n}(0)$ grows with $K$, we normalize our counts and consider the distribution of $\frac{\mathfrak{n}(h)}{\mathfrak{n}(0)}$. We find that this quantity appears to behave like the graph of $e^{-\sqrt{h+1}}$. If so, this suggests that $\frac{\mathfrak{n}(h+1)}{\mathfrak{n}(h)} \sim 1 - \frac{1}{2}h^{-\frac{1}{2}}$ as $h \rightarrow \infty$. This is roughly consistent with the data in the second column of Table \ref{table5}, for which with $h=21$ we have $\frac{\mathfrak{n}(h+1)}{\mathfrak{n}(h)} = 0.88921$ while $1 - \frac{1}{2}h^{-\frac{1}{2}} = 0.89089$. By Lemma \ref{average}, the  average value of $\mathfrak{h}(k)$  with $0\leq k \leq K$ has size about $(\log K)^2$. Since $\mathfrak{n}(0)$ has size a power of $K$, the data (at least in this short range for $K$) suggests that the maximal value of  $\mathfrak{h}(k)$ is probably a power of  $\log K$ , or at worst $\mathfrak{h}(k) \ll_{\epsilon} k^\epsilon$, for $\epsilon >0$ (the maximum value for $\mathfrak{h}(k)$ in our data was 131).  As mentioned in the introduction, the best we know is $\mathfrak{h}(k) \ll_{\ve} k^{\frac{1}{3} + \ve}$. 

\begin{figure}[!ht]
\vspace{1cm}
\centering
  \includegraphics[width=.7\linewidth]{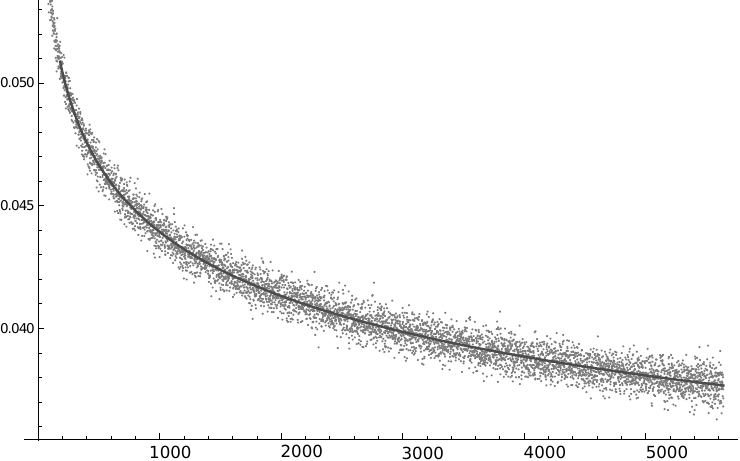}
  %\captionof{}{1}
\caption{Hasse failures in subintervals.\ }
\label{img6}
\vspace{1cm}
\end{figure}

\begin{figure}[!ht]
\centering
  \includegraphics[width=.65\linewidth]{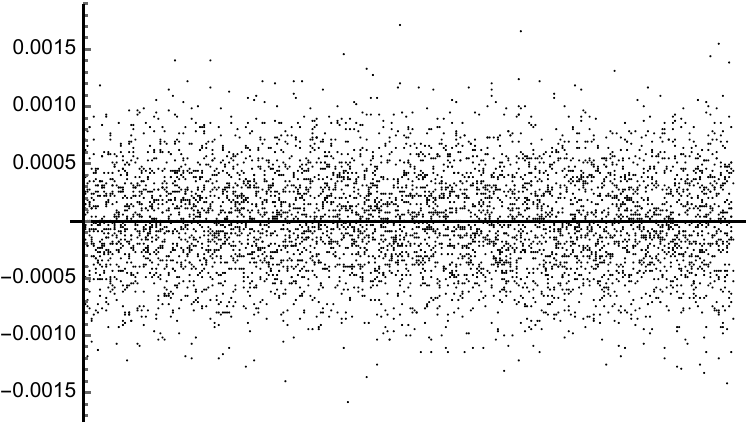}
  %\captionof{}{1}
\caption{Hasse failures in subintervals residuals.}
  \label{img7}
\vspace{2cm}
\end{figure}

\begin{figure}[!ht]
%\centering
  \includegraphics[width=.40\linewidth]{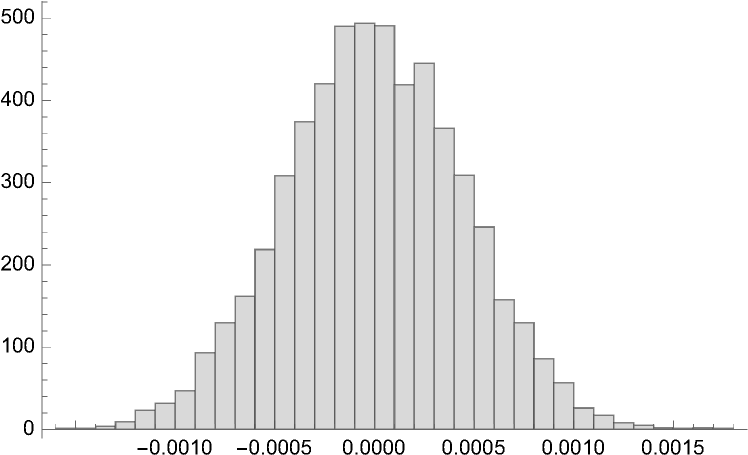}
  %\captionof{}{1}
  \label{img8}
%\caption{Hasse failures in subintervals residuals histogram.}
%\centering
\vspace{20pt}
  \includegraphics[width=.40\linewidth]{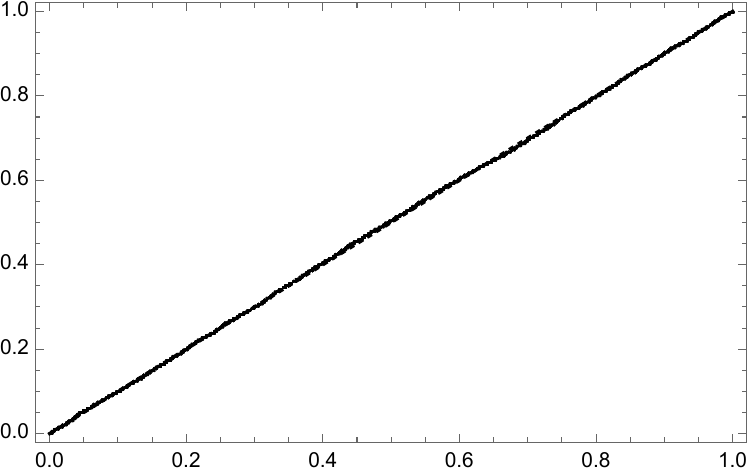}
  %\captionof{}{1}
\caption{Hasse failures in subintervals: residuals histogram and probability plot.}
  \label{img9}
\vspace{30pt}
\end{figure}
\pagebreak

\begin{center}
\begin{longtable}[!ht]{|l|l|l|}
%\captionsetup{font=small}
\caption[Funset]{\small{Sample fundamental sets, $(-u_1,u_2,u_3)$}}\label{funsettable} \\

\hline \multicolumn{1}{|c|}{$k$\quad\quad} & \multicolumn{1}{c|}{$\mathfrak{h}(k)$} & \multicolumn{1}{l|}{\textbf{\  $(u_1,u_2,u_3)$}} \\ \hline 
\endfirsthead

\multicolumn{3}{c}%
{{ \tablename\ \thetable{} \small{-- continued from previous page}}} \\
\hline \multicolumn{1}{|c|}{$k$} &
\multicolumn{1}{c|}{$\mathfrak{h}(k)$} &
\multicolumn{1}{l|}{\textbf{$\ (u_1,u_2,u_3)$}} \\ \hline
\endhead

\hline \multicolumn{3}{r}{{}} \\
\endfoot

\hline \hline
\endlastfoot

54 & 1 & (3, 3, 3) \\
70 & 1 & (3, 3, 4) \\
88 & 1 & (3, 3, 5) \\
108 & 1 & (3, 3, 6) \\
133 & 1 & (3, 4, 6) \\
154 & 1 & (3, 3, 8) \\
166 & 1 & (4, 5, 5) \\
188 & 1 &  (3, 5, 7) \\
189 & 1 & (3, 6, 6) \\
214 & 1 & (3, 4, 9) \\
236 & 1 & (5, 5, 6) \\
254 & 1 & (3, 7, 7) \\
270 & 1 & (3, 3, 12) \\
304 & 1 & (3, 3, 13) \\
329 & 2 & (3, 8, 8), (4, 4, 11) \\
341 & 1 & (4, 5, 10) \\
358 & 1 & (3, 5, 12) \\
378 & 1 & (3, 3, 15) \\
412 & 1 & (5, 6, 9) \\
414 & 1 & (3, 9, 9) \\
430 & 1 & (3, 4, 15) \\
446 & 1 & (5, 5, 11) \\
448 & 1 & (3, 6, 13) \\
460 & 2 & (3, 3, 17), (3, 9, 10) \\
473 & 2 & (3, 4, 16), (5, 8, 8) \\
494 & 2 & (4, 7, 11), (5, 5, 12) \\
502 & 1 & (4, 9, 9) \\
504 & 1 & (3, 3, 18) \\
518 & 1 & (3, 4, 17) \\
532 & 1 & (6, 6, 10) \\
540 & 1 & (3, 6, 15) \\
553 & 1 & (4, 8, 11) \\
558 & 1 & (3, 9, 12) \\
566 & 1 & (4, 5, 15) \\
616 & 1 & (4, 10, 10) \\
664 & 1 & (3, 9, 14) \\
665 & 2 & (3, 4, 20), (4, 8, 13) \\
668 & 2 & (3, 10, 13), (6, 7, 11) \\
684 & 1 & (6, 9, 9) \\
693 & 1 & (3, 6, 18) \\
700 & 1 & (3, 3, 22) \\
713 & 2 & (3, 8, 16), (5, 8, 12) \\
718 & 1 & (3, 4, 21) \\
{\ldots}&{\ldots}&{\ldots}\\
{\ldots}&{\ldots}&{\ldots}\\
9230 & 3 & (3, 28, 59), (7, 17, 52), (11, 25, 28) \\
9234 & 2 & (3, 15, 75), (9, 9, 63) \\
9253 & 3 & (3, 42, 44), (8, 9, 66), (12, 18, 35) \\
9260 & 9 & (3, 7, 86), (3, 19, 70), (3, 29, 58), (5, 19, 58), (5, 31, 42)\\ 
{}&{}& (6, 23, 47), (7, 31, 33),(9, 13, 53), (9, 22, 37) \\
9261 & 1 & (6, 15, 60) \\
9268 & 1 & (6, 32, 36) \\
9288 & 2 & (3, 30, 57), (6, 12, 66) \\
9289 & 1 & (3, 24, 64) \\
9296 & 1 & (10, 11, 55) \\
9302 & 3 & (4, 21, 61), (5, 9, 76), (11, 19, 36) \\
9304 & 5 & (3, 13, 78), (9, 14, 51), (9, 27, 31), (13, 18, 33), (14, 21, 27) \\
9308 & 3 & (5, 27, 47), (9, 11, 58), (10, 23, 33) \\
9310 & 3 & (3, 3, 92), (3, 20, 69), (4, 13, 73) \\
9313 & 1 & (4, 24, 57) \\
9317 & 2 & (4, 6, 85), (4, 34, 45) \\
9322 & 2 & (5, 24, 51), (9, 15, 49) \\
9329 & 2 & (7, 8, 72), (8, 28, 33) \\
9353 & 3 & (4, 36, 43), (8, 12, 59), (8, 29, 32) \\
9358 & 4 & (3, 21, 68), (9, 23, 36), (12, 13, 45), (12, 21, 31) \\
9368 & 3 & (3, 14, 77), (7, 21, 46), (13, 21, 29) \\
9373 & 4 & (3, 6, 88), (4, 38, 41), (11, 22, 32), (18, 18, 25) \\
9380 & 7 & (3, 34, 53), (4, 22, 60), (6, 20, 52), (8, 10, 64), (8, 24, 38) \\
{}&{}& (10, 24, 32), (15, 17, 31)\\
9388 & 3 & (6, 9, 73), (6, 17, 57), (9, 19, 42) \\
9405 & 1 & (3, 18, 72) \\
9414 & 3 & (3, 9, 84), (3, 36, 51), (9, 21, 39) \\
9416 & 2 & (4, 30, 50), (5, 29, 45) \\
9430 & 2 & (3, 15, 76), (12, 15, 41) \\
9436 & 2 & (6, 25, 45), (10, 25, 31) \\
9446 & 1 & (11, 20, 35) \\
9449 & 2 & (4, 16, 69), (8, 16, 51) \\
9450 & 1 & (3, 39, 48) \\
9454 & 11 & (3, 7, 87), (4, 11, 77), (4, 23, 59), (4, 31, 49), (7, 12, 63), \\
{}&{}& (7, 17, 53), (7, 28, 37), (11, 23, 31), (13, 13, 43), (15, 20, 27) \\
{}&{}&(17, 17, 28)\\
9468 & 1 & (3, 42, 45) \\
9470 & 4 & (3, 43, 44), (5, 7, 81), (5, 12, 71), (17, 21, 23) \\
9484 & 2 & (3, 5, 90), (9, 13, 54) \\
9493 & 5 & (3, 30, 58), (4, 27, 54), (6, 12, 67), (6, 14, 63), (6, 23, 48) \\
9494 & 1 & (7, 29, 36) \\
9500 & 8 & (3, 13, 79), (5, 9, 77), (5, 10, 75), (5, 31, 43), (6, 13, 65) \\
{}&{}&(10, 13, 51), (10, 27, 29), (13, 23, 27)\\
9504 & 3 & (3, 3, 93), (6, 21, 51), (12, 12, 48) \\
9520 & 2 & (3, 31, 57), (13, 15, 39) \\
9532 & 1 & (15, 21, 26) \\
9538 & 1 & (5, 27, 48) \\
\end{longtable}
\end{center}

\begin{table}[!ht]
\caption{Distribution of $\mathfrak{h}(k)$, generic $k\leq 10^7$.}
\vspace{10pt}
\centering\small

\begin{tabular}{|c| c| }
\hline
$\mathfrak{h}(k)$ &   $\mathfrak{n}(\mathfrak{h}(k))$ occurrences  \\
\hline
 0 & 574,778 \\
 1 & 423,094 \\
 2 & 346,019 \\
 3 & 259,787 \\
 4 & 202,111 \\
 5 & 157,726 \\
 6 & 124,744 \\
 7 & 100,431 \\
 8 & 81,243 \\
 9 & 66,794 \\
 10 & 54,942 \\
 11 & 45,898 \\
 12 & 38,719 \\
 13 & 32,886 \\
 14 & 28,001 \\
 15 & 23,954 \\
 16 & 20,930 \\
 17 & 17,932 \\
 18 & 15,970 \\
 19 & 13,748 \\
 20 & 12,105 \\
 21 & 10,434 \\
 \hline
 \end{tabular}
 \quad\quad
%\caption*{Distribution of consecutive ratios of $\mathfrak{h}(k)$, $k\leq 10^7$}
\begin{tabular}{|c| l| }
\hline
$s$ &   $ \mathfrak{n}(s+1)/\mathfrak{n}(s)$  \\
\hline
 0 & 0.7361 \\
 1 & 0.81783 \\
 2 & 0.750788 \\
 3 & 0.777987 \\
 4 & 0.780393 \\
 5 & 0.790891 \\
 6 & 0.805097 \\
 7 & 0.808943 \\
 8 & 0.822151 \\
 9 & 0.822559 \\
 10 & 0.83539 \\
 11 & 0.843588 \\
 12 & 0.84935 \\
 13 & 0.851457 \\
 14 & 0.855469 \\
 15 & 0.873758 \\
 16 & 0.856761 \\
 17 & 0.890587 \\
 18 & 0.860864 \\
 19 & 0.880492 \\
 20 & 0.861958 \\
 21 & 0.888921 \\
 \hline
 \end{tabular}
 %\vspace{.5cm}
 \label{table5}
\end{table}

We end this Section with some basic Conjectures concerning the class 
numbers $\mathfrak{h}(k)$. These are suggested by our theoretical results as well as our more refined numerical findings.

\begin{conjecture}\label{conj1} For any $\ve >0$
\[ \mathfrak{h}(k) \ll_{\ve}\ |k|^{\ve}. \]
\end{conjecture}

\begin{conjecture}\label{conj2}
The number of Hasse failures for $0\leq k \leq K$ satisfies
\[
\big{|}\left\{ 0\leq k \leq K:  \mathfrak{h}(k)=0 \ \text{and}\ k \ \text{admissible}\ \right\}\big{|} \sim C_0 K^{\theta}, 
\]
for some $C_0 >0$ and some $\frac{1}{2}< \theta <1$.

More generally, for $t\geq 1$
\[
\big{|}\left\{ 0\leq k \leq K:  \mathfrak{h}(k)=t \ \text\ \right\}\big{|} \sim C_t K^{\theta}, 
\]
with $C_t >0$.
\end{conjecture}

The values of $C_t$ above are illustrated in Fig. \ref{img10}, suggesting an exponential decay in $t$.

\begin{figure}[!ht]
\centering
  \includegraphics[width=\linewidth]{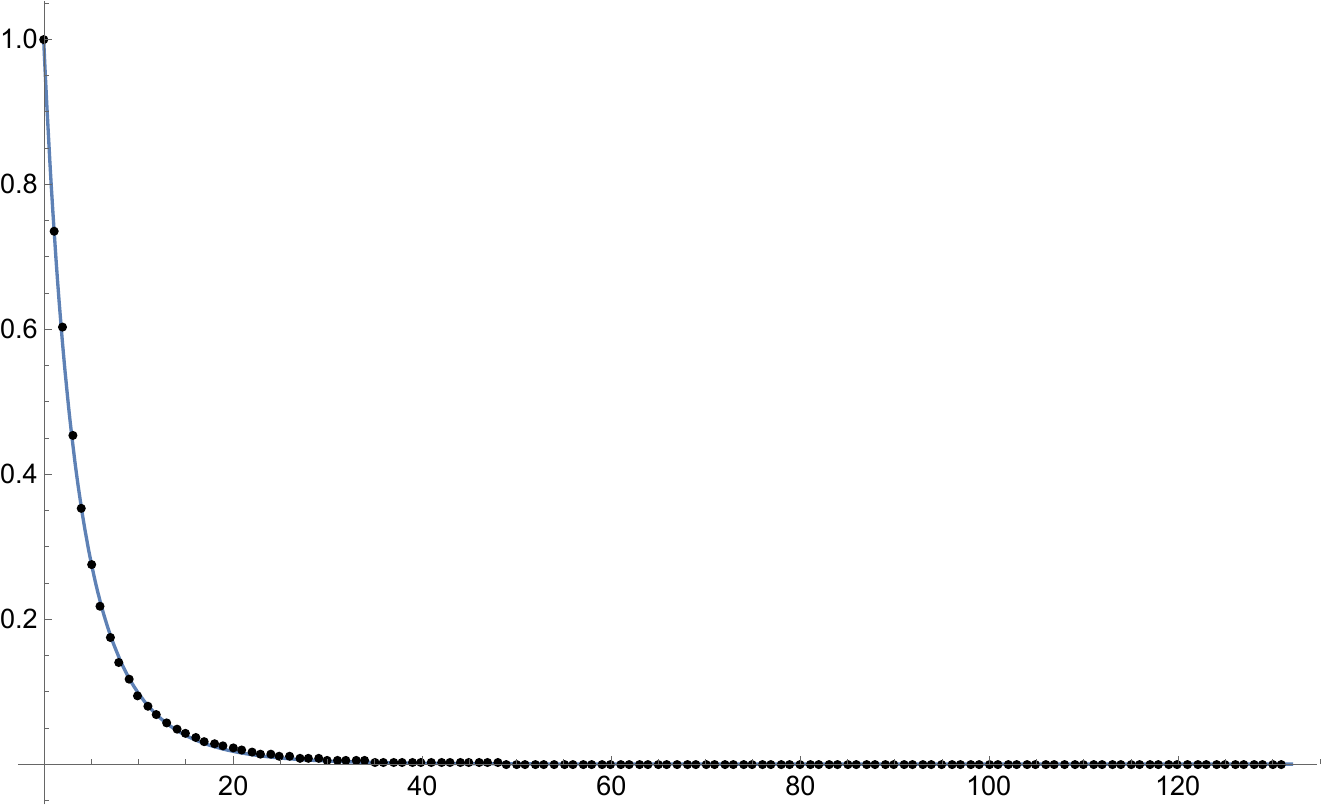}
  %\captionof{}{1}
  \begin{minipage}[b]{\textwidth}
\caption{Occurences of relative number of orbits: $\mathfrak{n}(\mathfrak{h}(k))/\mathfrak{n}(0)$,\\ generic $k\leq 10^7$. Approximation curve $\mathfrak{n}(h)=\mathfrak{n}(0) (6.86293 + 4.62621 h$ $ + 0.0576149 h^2)e^{-1.92905\sqrt{h+1}}$. }
\label{img10}
\end{minipage}
\vspace{2cm}
\end{figure}

\vspace{10pt}
\begin{figure}[!ht]
\centering
  \includegraphics[width=.8\linewidth]{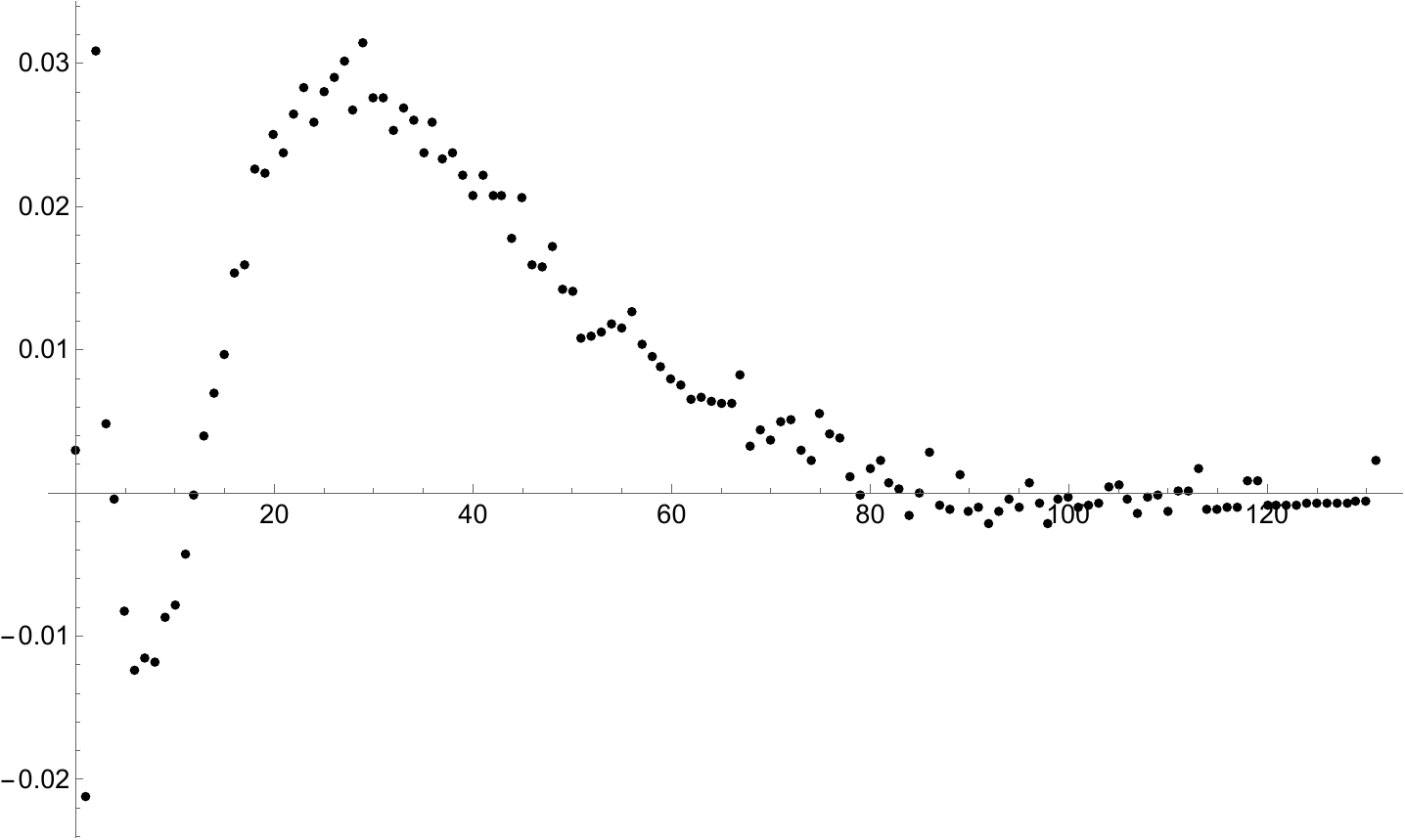}
  %\captionof{}{1}
\caption{Residuals to Fig. \ref{img10}. }
\label{img11}
%\vspace{2cm}
\end{figure}

%%%%%%%%%%%%%%%%%%%%%%%%%%%%%%%%%%%%%%%%%%%%%%%%%%%%%%%%%%%%%%%%%%%
%%%%%%%%%%%%%%%%%%%%%%%%%%%% APPENDIX %%%%%%%%%%%%%%%%%%%%%%%%%%%%%%%%%%
%\clearpage
\begin{appendices}
\section*{Appendix}\ 
The appendix consists of  (A) a discussion of invariants of  affine cubic forms  referred to in the Introduction, and (B) computation of local masses $\delta_p$, for primes $p \geq 2$ (with some details omitted); their structure is used in the proofs in Section \ref{sec8}. 

\section{Arithmetic invariants of  affine cubic forms.}
\label{invariants}
A number of invariants of $f$ as an element of the unique factorization domain $R=\mathbb{Q}[x_1,x_2,\ldots,x_n]$, enter into the study of the values assumed by such an affine cubic form $f$. The first is the  $\mathcal{h}$-invariant from \cite{D-L}: $\mathcal{h}(f)$ is the minimal integer $h$ for which 
\begin{equation}\label{inv-1}
f_0 = L_1Q_1 + L_2Q_2 + \ldots + L_hQ_h ,
\end{equation}
where the $L_j$'s are homogeneous linear and the $Q_j$'s are homogeneous quadratic members of $R$; equivalently  $n - \mathcal{h}(f)$ is the dimension of the largest $\mathbb{Q}$-linear subspace contained in $W_0 = \left\{{\bf x}: f_0({\bf x})=0 \right\}$, the linear space given by $L_1 =L_2 = \ldots = L_h =0$. Note that $\mathcal{h}(f)=1$ iff $f_0$ is reducible in $R$, and in this case $W_0$ contains a rational hypersurface.

Closely related are the $\mathbb{Q}$-invariants $\mathcal{l}(f)$ and $\mathcal{q}(f)$ defined as the dimensions of the largest $\mathbb{Q}$-affine linear subspaces $U_{\mathcal{l}}$ and $U_{\mathcal{q}}$ of $\mathbb{A}^n$ on which the restriction of $f$ to $U_{\mathcal{l}}$ is linear (non-constant) and to $U_{\mathcal{q}}$ is quadratic. So, $\mathcal{l}(f)$ and $\mathcal{q}(f)$ lie in $[0,n-1]$. Of particular interest to us is that
\begin{equation}\label{inv-2}
\mathcal{h}(f)=1 \quad \text{iff}\quad \mathcal{q}(f)=n-1.
\end{equation}

The group $\text{Aff}_n(\mathbb{Z})$ consisting of integral affine linear maps ${\bf x}\to A{\bf x} + {\bf b}$ with $A\in \text{GL}_n(\mathbb{Z})$ and ${\bf b}\in \mathbb{Z}^n$, acts on the integral cubic polynomials by a change of variable. The arithmetic invariants as well as the diophantine questions concerning $V_{k,f}(\mathbb{Z})$ are all preserved by this action. On the leading homogeneous cubic term $f_0$, the action is that of $\text{GL}_n(\mathbb{Z})$, which has been well studied in terms of its invariants. With these fixed, there are finitely many $\text{GL}_n(\mathbb{Z})$ orbits, see \cite{B-S} for a recent discussion of the case $n=3$, which is our interest. In this case the vector space of $f_0$'s is 10-dimensional and it's quotient by $SL_3$ is 2-dimensional, given by the Aronhold invariants $I$ and $J$. The vector space of $f$'s is 20-dimensional and its quotient by $\text{Aff}_3$ is 9-dimensional. The invariants for this action up to the additive constant term and at a generic point are $I(f_0)$, $J(f_0)$ together with the 6-dimensional vector space associated with the homogeneous quadratic part of $f$

%Once $f_0$ is reduced, the extra freedom of the $\text{Aff}_n(\mathbb{Z})$ action is to bring the linear term of $f$ into homogeneous form (for the generic $f$) and gives a description of the orbits of $\text{Aff}_n(\mathbb{Z})$ on cubics.

We end with some examples of affine cubic forms and their invariants.
\begin{enumerate}[wide,labelindent=0pt,label=(\arabic*).]
\item $S({\bf x}) = x_1^3 + x_2^3 + x_3^3$,\quad $\mathcal{h}(S)=3$, $\mathcal{l}(S)=\mathcal{q}(S)=0$;
\item $M({\bf x}) = x_1^2 + x_2^2 + x_3^2 - x_1x_2x_3$,\quad $\mathcal{h}(M)=1$, $\mathcal{l}(M)=0$, $\mathcal{q}(M)=2$;
\item $T({\bf x}) =   x_1x_2x_3 +x_1 + x_2 $ (perhaps the mildest perturbation  of the fully split form $x_1x_2x_3$),\quad $\mathcal{h}(M)=1$, $\mathcal{l}(M)=\mathcal{q}(M)=2$ (the restriction of $T$ to $x_3=0$ is linear). From the last it follows that $\mathfrak{v}_T(k)=\vert V_{k,T}(\mathbb{Z})\vert=\infty$; however $T$ is not perfect or even almost perfect since $V_{k,T}(\mathbb{Z})$ is not Zariski dense in $V_{k,T}$ for $k\neq 0$.
\item $P({\bf x})= x_1x_2x_3 + (x_1 -1)Q_1({\bf x}) + (x_2 -1)Q_2({\bf x})$, with $Q_1$, $Q_2$ generic quadratics. Then, $\mathcal{l}(P)=\mathcal{q}(P)=1$ (with $x_1=x_2=1$ giving the line $U_{\mathcal{l}}$). In particular, $V_{k,P}(\mathbb{Z})\neq \emptyset$ for every $k$. We expect that $P$ is full.
\end{enumerate}

%%%%%%%%%%%%%%%%%%%%%%%%%%%%%%%%%%%%%%%%%%%%%%%%%%%%%%%%%%%%%
\section{Analysis of the local masses.}
\subsection{Computation of $\delta_{p}(k)$ for odd primes.}
\label{Apx-A}\

For any integer $k$ and prime $p\geq 3$, we determine
\[
\delta_{p}(k) = \lim_{l \to \infty}\ |V_{k}(\mathbb{Z}\slash\; p^{l}\mathbb{Z})|\;p^{-2l}\ .
\]
Define

\begin{equation}\label{A7} N_{l}(k) = p^{-3l}\sideset{}{^*}\sum_{b\ (\text{mod}\ p^l)}\ \sum_{{\bf x}\ (\text{mod}\ p^l)}\ e\left(\frac{f({\bf x})-k}{p^l}b\right)\ ,
\end{equation}
where ${\bf x}=(x_{1},x_{2},x_{3})$, $f({\bf x})= x_{1}^2 + x_{2}^2 + x_{3}^2 -x_1x_2x_3$ and the asterisk denotes a sum over those $b$'s not divisible by $p$. Then one has

\begin{equation}\label{A8}
\delta_{p}(k) = 1 + \sum_{l=1}^{\infty}N_{l}(k)\ .
\end{equation}

In what follows, we analyze the case $l\geq 2$ (the case $l=1$ is determined by Lemma \ref{5.4}). For $p \geq 3$ one has

\begin{equation}\label{A9}
N_{l}(k) = p^{-3l}\sideset{}{^*}\sum_{b\ (\text{mod}\ p^l)}\ e\left(\frac{4(4-k)}{p^l}b\right) \sum_{{\bf x}}\ e\left(\frac{(2x_3 -x_1x_2)^2 -(x_{1}^2 -4)(x_{2}^2 -4)}{p^l}b\right)\ .
\end{equation}

Making a change of variable shows that the inner sum over ${\bf x}$ is

\begin{equation}\label{A10}
\sum_{u}\sum_{x_1,x_2}\ e\left(\frac{bu^2}{p^l}\right)\ e\left(\frac{-b(x_{1}^2 -4)(x_{2}^2 -4)}{p^l}\right) = S(b;p^{l})\sum_{x}\ e\left(\frac{4b(x^2 -4)}{p^l}\right)\overline{S\left(b(x^2 -4);p^{l}\right)}\ , 
\end{equation}
{}
where  for $q\geq 1$ we put
\begin{equation} S(b;q)=\sum_{r\ ({\rm mod}\ q)} e\left(\frac{br^2}{q}\right).
\end{equation}

Using properties of the Gauss sum, we get

  \begin{proposition}\label{LA5} For $p \geq 3$ we have
 \begin{enumerate}[label=(\alph*).]
 \item \[ N_{1}(k) = \chi(k-4)\left[3+\chi(k)\right]\frac{1}{p} + \frac{1}{p^2}\ ; \]
 \item  if $l\geq 3$ is odd, 
 \[
 N_{l}(k) = \left\{ \begin{array}{ll}
 4p^{-\frac{1}{2}(l +1)}\chi\left(\frac{k-4}{p^{l-1}}\right) & \text{if}\ \   p^{l-1}\vert (k-4) , \\
 p^{-\frac{1}{2}(l +1)}\chi\left(\frac{k}{p^{l-1}}\right) & \text{if}\ \   p^{l-1}\vert k , \\
\quad  0& \text{otherwise}\ ;
 \end{array} \right.
 \] 
 \item if $l\geq 2$ is even, then
 \[
 N_{l}(k) = \left\{ \begin{array}{ll}
  -p^{-\frac{l+2}{2}}\left\{4\eta_{l-1}(k-4) + \eta_{l-1}(k)\right\} & \text{if}\ \   p^{l-1}\vert\vert k(k-4) , \\
 p^{-\frac{l}{2}}\left(1 -\frac{1}{p}\right)\left\{4\eta_{l}(k-4) + \eta_{l}(k)\right\} & \text{if}\ \   p^{l}\vert k(k-4) , \\
 \quad 0& \text{otherwise},
 \end{array} \right.
 \]
 where we define $\eta_{l}(m) = 1$ if $p^l |m$ and is zero otherwise;
 \end{enumerate}
 where $\chi(b)$ denotes the Legendre symbol $\left(\frac{b}{p}\right)_{L}$.
 \end{proposition}

 To compute $\delta_{p}(k)$ for $p\geq 3$ in \eqref{A8}, we write $\delta_{p}(k) = 1 + N_{1}(k) + \mathfrak{S}_{p}(k)$. Define $\mu \geq 0$ by $p^{\mu}\vert\vert k(k-4)$. By Prop. \ref{LA5}, $\mu = 0$ implies $N_{l}(k)=0$ for $l\geq 2$, so that we have $\mathfrak{S}_{p}(k)=0$ for this case. For $\mu \geq 1$, we have $p^{\mu}\vert\vert (k- 4\beta)$, with $\beta = 0$ or $1$. Then we combine Prop. \ref{LA5} in \eqref{A8}, to  get

\begin{equation}\label{A19}
\mathfrak{S}_{p}(k)  = 4^{\beta}\times\left\{ \begin{array}{ll}
p^{-1} - p^{-\frac{1}{2}(\mu +1)} - p^{-\frac{1}{2}(\mu +3) } & \text{if}\ \ 2\nmid \mu\ ,\\
p^{-1} - p^{-\frac{\mu}{2} -1} \left(1 - \chi\left(\frac{k-4\beta}{p^\mu}\right)\right)  & \text{if}\ \ 2\vert \mu\ .
\end{array} \right.
\end{equation}

In particular, we see that if $\mu=1$, then $\mathfrak{S}_{p}(k) = - 4^{\beta}p^{-2}$ while if $\mu \geq 2$ then $\mathfrak{S}_{p}(k) = 4^{\beta}p^{-1} + O(p^{-2})$.

Combining \eqref{A19} with Prop. \ref{LA5}(a) in \eqref{A8} gives

 \begin{proposition}\label{LA3} For $p \geq 3$, suppose $p^{\mu}\vert\vert k(k-4)$ with $\mu \geq 0$. We have
 \begin{enumerate}[label=(\alph*).]
 \item if $\mu = 0$, then $\delta_{p}(k) = 1 + \chi(k-4)\left[3+\chi(k)\right]\frac{1}{p} + \frac{1}{p^2}$\ ;
 \item if $p\vert\vert k$, then $\delta_{p}(k) = 1 + 3\chi(-1)\frac{1}{p}$\ ;
 \item if $p\vert\vert (k-4)$, then $\delta_{p}(k) = 1 - \frac{3}{p^2}$\ ;
 \item if $\mu \geq 2$ and $p\vert k$, then $$\delta_{p}(k) =  1 + \left\{ \begin{array}{ll}
\left(1 + 3\chi(-1)\right)p^{-1} +p^{-2} - p^{-\frac{1}{2}(\mu +1)} - p^{-\frac{1}{2}(\mu +3) } & \text{if}\ \ 2\nmid \mu\ ,
\\
\left(1 + 3\chi(-1)\right)p^{-1} + p^{-2} - p^{-\frac{\mu}{2} -1} \left(1 - \chi\left(\frac{k}{p^\mu}\right)\right)  & \text{if}\ \ 2\vert \mu\ ;
\end{array} \right. $$
\item if $\mu \geq 2$ and $p\vert (k-4)$, then $$\delta_{p}(k) =  1 + \left\{ \begin{array}{ll}
4p^{-1} +p^{-2} - 4p^{-\frac{1}{2}(\mu +1)} - 4p^{-\frac{1}{2}(\mu +3) } & \text{if}\ \ 2\nmid \mu\ ,\\
4p^{-1} +p^{-2} - 4p^{-\frac{\mu}{2} -1} \left(1 - \chi\left(\frac{k}{p^\mu}\right)\right)  & \text{if}\ \ 2\vert \mu\ .
\end{array} \right. $$
 \end{enumerate}
 \end{proposition}

\begin{remark} The case (b) shows that $\delta_{3}(k) =0$ if $k \equiv 3$ or $6$ (mod $9$), while case (a) and (d) shows that $\delta_{3}(k)>0$ otherwise.
\end{remark}

  %%%%%%%%%%%%%%%%%%%%%%%%%%%%%%%%%%%%%%%%%%%%%%%%%%%%%%%%%%%%%%%%%%%%%
  %%%%%%%%%%%%%%%%%%%%%%%%%%%%%%%%%%%%%%%%%%%%%%%%%%%%%%%%%%%%%%%%%%%%%%
  %%%%%%%%%%%%%%%%%%%%%%%%%%%%%%%%%%%%%%%%%%%%%%%%%%%%%%%%%%%%%%%%%%%%%%%
  
  \subsection{Local factors associated with $V_{a_1,a_2}$, odd primes.}
  \label{Apx-B}\ 
  
 We next state (without details) the analogous results for the density function $\delta_{p}(a_1,a_2)$ in  \eqref{mo20} for the surface $V_{a_1,a_2}$ in \eqref{mo18}. Recalling the properties in \eqref{LB1} and \eqref{LB1a}, since $p$ is odd, completing the square gives us
  
  \begin{equation}\label{LB2}
  N_{l}(a_1,a_2) = p^{-3l}\sideset{}{^*}\sum_{b\ (\text{mod}\ p^l)}e\left(4b\frac{D_{a_1}-D_{a_2}}{p^l}\right)\overline{S(bD_{a_1};p^{l})}S(bD_{a_2};p^{l})\ .
  \end{equation}
 Again, using properties of the Gauss sums 
 gives us
  \begin{proposition}\label{LA40} Let $a_1 \neq a_2$ be fixed, and let $p\geq 3$.
  \begin{enumerate}[label=(\alph*).]
  \item Suppose $p \nmid D_{a_1}D_{a_2}\left(D_{a_1} -D_{a_2}\right)$. Then
  \[
  N_{l}(a_1,a_2)  =  \left\{ \begin{array}{ll}
- \frac{\chi \left(D_{a_1}D_{a_2}\right)}{p^2}  & \text{if}\ \ l=1 ,\\
0  & \text{otherwise}\ .
\end{array}\right .
  \]
  \item Suppose $p \nmid D_{a_1}D_{a_2}$ and $p^{\mu}\vert\vert\left(D_{a_1}-D_{a_2}\right)$ with $\mu \geq 1$. Then
  \[
  N_{l}(a_1,a_2)  =  \left\{ \begin{array}{ll}
\frac{1}{p^l}\left(1-\frac{1}{p}\right)  & \text{if}\ l\leq \mu ,\\
-p^{-\mu -2} & \text{if}\  l=\mu +1 , \\
0 & \text{otherwise}\ .
\end{array}\right .
  \]
  \item Suppose $p^{\alpha}\vert\vert D_{a_1}$ but $\ p\nmid D_{a_2}$ with $\alpha \geq 1$. Then
   \[
  N_{l}(a_1,a_2)  =  \left\{ \begin{array}{ll}
p^{-1} & \text{if}\ \ l=1 ,\\
0  & \text{otherwise}\ .
\end{array}\right .
  \]
  \item Suppose $p^{\eta_1}\vert\vert D_{a_1}$, $p^{\eta_2}\vert\vert D_{a_2}$ and $p^{\mu}\vert\vert\left(D_{a_1}-D_{a_2}\right)$ with $\eta_1$,$\eta_2$ and $\mu \geq 1$\ . Putting $\eta = \min{(\eta_1,\eta_2)}$ gives us
  \[
   N_{l}(a_1,a_2)  =  \left\{ \begin{array}{ll}
\left(1-\frac{1}{p}\right)  & \text{if}\ \ 1\leq l\leq \eta\ ,\\
p^{-1}\ & \text{if}\ \  l=\eta +1,\ \eta_1\neq \eta_2  \ ,\\
- p^{-\eta -2}\chi \left(\frac{D_{a_1}}{p^\eta}\right)\chi \left(\frac{D_{a_2}}{p^\eta}\right)& \text{if}\ \ l= \eta +1,\ \eta_1 = \eta_2 \leq \mu \ ,\\
0  & \text{otherwise}\ .
\end{array}\right .
\]
  \end{enumerate}
  \end{proposition}
  
 It then follows that
  
  \begin{proposition}\label{LA4} Let $a_1 \neq a_2$ be fixed, and let $p\geq 3$.
  \begin{enumerate}[label=(\alph*).]
  \item Suppose $p \nmid D_{a_1}D_{a_2}\left(D_{a_1} -D_{a_2}\right)$. Then
  \[
  \delta_{p}(a_1,a_2) = 1 - \frac{\chi \left(D_{a_1}D_{a_2}\right)}{p^2}\ .
  \]
  \item Suppose $p \nmid D_{a_1}D_{a_2}$ and $p^{\mu}\vert\vert\left(D_{a_1}-D_{a_2}\right)$ with $\mu \geq 1$. Then
  \[
  \delta_{p}(a_1,a_2) = \left(1 + \frac{1}{p}\right)\left(1 - \frac{1}{p^{\mu +1}}\right)\ .
  \]
  \item Suppose $p\vert D_{a_1}D_{a_2}$ but $p \nmid \left(D_{a_1}-D_{a_2}\right)$. Then
   \[
  \delta_{p}(a_1,a_2) = \left(1 + \frac{1}{p}\right)\ .
  \]
  \item Suppose $p^{\eta_1}\vert\vert D_{a_1}$, $p^{\eta_2}\vert\vert D_{a_2}$ and $p^{\mu}\vert\vert\left(D_{a_1}-D_{a_2}\right)$ with $\eta_1$,$\eta_2$ and $\mu \geq 1$\ . Putting $\eta = \min{(\eta_1,\eta_2)}$ gives us
  \[
  \delta_{p}(a_1,a_2) =  \left\{ \begin{array}{ll}
(1+\eta) - \frac{\eta -1}{p}  & \text{if}\ \ \eta_1 \neq \eta_2\ ,\\
(1+\eta) - \frac{\eta }{p}- \frac{1}{p^{2}}\chi\left(\frac{D_{a_1}}{p^\eta}\right) \chi\left(\frac{D_{a_2}}{p^\eta}\right)& \text{if}\ \ \eta_1 = \eta_2 =\mu \ ,\\
(1+\eta) - \frac{\eta -1 }{p}- \frac{1}{p^{\mu - \eta +1}}\left(1 + \frac{1}{p}\right) & \text{if}\ \ \eta_1 = \eta_2 <\mu \ .
\end{array}\right .
\]
  \end{enumerate}
  \end{proposition}
  
   \begin{remark} If $a_1 = a_2 =a$  and  $p\geq 3$, one can deduce the result for $\delta_{p}(a,a)$ from parts (c) and (d) above, with $\mu \to \infty$, giving
  \begin{enumerate}[label=(\alph*).]
  \item if $p\nmid D_a$, then $\delta_{p}(a,a) = 1+ p^{-1}$\ , and
   \item if $p^\eta \vert\vert D_a$ with $\eta \geq 1$, then $\delta_{p}(a,a) = (1+\eta) - \frac{\eta -1}{p}$\ .
  \end{enumerate}
  \end{remark}
 %%%%%%%%%%%%%%%%%%%%%%%%%%%%%%%%%%%%%%%%%%%%%%%%%%%%%%%%%%%%%%%%%%%
 %%%%%%%%%%%%%%%%%%%%%  APPENDIX C  %%%%%%%%%%%%%%%%%%%%%%%%%%%%%%%%%%%%%%% 

 \subsection{The even local factor $\delta_{2}(k)$.}
 \label{app-C}\ 
 
 Since the analysis here is a bit more delicate, we provide some additional details.
 Let $l\geq 0$ and define $F_{l}(c)=\sum_{x\ \text{mod}\ 2^l}\ e\left(\frac{cx^2}{2^l}\right)\ $. Recall the three primitive real characters modulo powers of two:  $\chi_{4} $ modulo 4, $\chi_{8}$ and $\chi_{4}\chi_{8}$ modulo 8, where 
 \[
 \chi_{4}(x) = \left(\frac{-4}{x}\right)_{\text{J}} = \left\{ \begin{array}{ll}
\ 1 & \text{if}\ \ x \equiv 1\ \text{mod}\ 4\ ,\\
 -1 & \text{if}\ \ x \equiv 3\ \text{mod}\ 4\ ,\\
\ 0 &  \ \text{otherwise}\ ,
 \end{array} \right. 
 \]
 and
 \[
 \chi_{8}(x) = \left(\frac{8}{x}\right)_{\text{J}}= \left\{ \begin{array}{ll}
\ 1 & \text{if}\ \ x \equiv \pm 1\ \text{mod}\ 8\ ,\\
 -1 & \text{if}\ \ x \equiv \pm 3\ \text{mod}\ 8\ ,\\
\ 0 &  \ \text{otherwise}\ .

 \end{array} \right.\label{ap} 
 \]
 
 For $l\geq 1$ we define $\omega_{l}(k)$ to be 1 if $2^l \vert k$ and 0 otherwise; if $l\leq 0$, we define  $\omega_{l}(k)$ to be 1 always. Given a term $\omega_{l}(k)$, we define $\hat{k}=\frac{k}{2^l}$. While $\omega_{l}(k)= \omega_{l}(-k)$, the corresponding ``hat'' function is not the same, and the appropriate choice is determined by the $\omega$-function.

 We have
 
 \begin{lemma}\label{Lb1}\ 
 Define $\theta\geq 0$ so that $2^{\theta}\vert\vert c$. We have
 \begin{enumerate}[label=(\alph*).]
 \item if $\theta \geq l$, $F_{l}(c)  = 2^l$\ ;
 \item if $\theta=l-1$, $F_{l}(c)=0$\ ;
 \item if  $l\geq 2$ and $2 \nmid c$,  then 
$ F_{l}(c) = 2^{\frac{l}{2}}\chi_{8}(c)^{l}\left[1 + \chi_{4}(c)i\right]$\ ;
 \item if $l\geq 3$ and $1 \leq \theta \leq l-2$, we have
 \[
 F_{l}(c) = 2^{\frac{l+\theta}{2}}\chi_{8}\left(\frac{c}{2^\theta}\right)^{l+\theta}\left[1 + \chi_{4}\left(\frac{c}{2^\theta}\right)i\right]\ ;
 \]
 \item for $l\geq 1$ and $q\in \mathbb{Z}$,
 \[
 \sideset{}{^*}\sum_{b\ (\text{mod}\ 2^l)}\ e\left(\frac{qb}{2^l}\right)F_{l}(b) = \omega_{l-3}(q)\ 2^{\frac{3(l-1)}{2}}\cos\left(\frac{\hat{q}+1}{4}\pi\right)[1 + (-1)^{l+\hat{q}}]\ ,
 \]
 where $\hat{q}=\frac{q}{2^{l-3}}$, and the sum over $b$ runs through odd numbers.
 \end{enumerate}
 \end{lemma}

 \begin{lemma}\label{Lb2}
 For any $b$  and $l\geq 0$, put 
 \[
 Q_{l}(b;a)=\sum_{x_1,x_2\ (\text{mod}\ 2^l)}\ e\left(b\frac{x_{1}^2 + x_{2}^2 -ax_{1}x_{2}}{2^l}\right)\ .
 \] 
 \begin{enumerate}[label=(\alph*).]
\item If $2\vert a$, then $Q_{l}(b;a)=F_{l}(b)\overline{F_{l}\left(b(\frac{a^2}{4}-1)\right)}$,
\item If $2 \nmid ab $, then $Q_{l}(b;a)= (-2)^l$.
\end{enumerate}
 \end{lemma}

 We now compute $N_{l}(k)$ given in \eqref{A7} with $p=2$, where the sum over $b$ runs through odd numbers. It will be convenient to compute $N_{l}(k)$ for some small values and we give it as
 
 \begin{lemma}\label{Lb3}\ 
 \begin{enumerate}[label=(\alph*).]
 \item  $N_{0}(k)=1$\ ;
 \item $N_{1}(k)= \frac{1}{4}(-1)^k$\ ;
 \item $N_{2}(k)= \frac{1}{4}\cos(k\frac{\pi}{2}) + \frac{3}{4}\sin(k\frac{\pi}{2})$\ ;
 \item 
 \[N_{3}(k)= 
  \left\{ \begin{array}{ll}
 \frac{3}{4}(-1)^{\frac{k+3}{4}} & \text{if}\ \ k \equiv 1 \ \text{mod}\ 4,\\

 \quad 0 & \text{otherwise}\ .
 \end{array} \right. 
 \]
 \end{enumerate}
 \end{lemma}

 For $l\geq 4$, we have
 
 \begin{equation}\label{B1}
 N_{l}(k) = 2^{-3l}\sideset{}{^*}\sum_{b\ (\text{mod}\ 2^l)}\ e\left(-\frac{kb}{2^l}\right)\sum_{a\ (\text{mod}\ 2^l)}\ e\left(\frac{ba^2}{2^l}\right)Q_{l}(b;a)\ .
 \end{equation}
 
 Using the lemmas above, we conclude
 
 \begin{lemma}\label{Lb5}\
 \begin{enumerate}[label=(\alph*).]
 \item If $k$ is odd and $l\geq 4$, $N_{l}(k)=0$\ ;
 \item if $l=4$ then $N_{4}(k)=0$ unless $4\vert k$, in which case
 \[N_{4}(k)= 
  \left\{ \begin{array}{ll}
 \frac{1}{2}\chi_{4}\left(\frac{k}{4}\right) & \text{if}\ \ 4\vert\vert k\ , \\
 \\
 \ \frac{1}{4}(-1)^{\frac{k}{8}+1} & \text{if}\ \ 8\vert k \ ;
 \end{array} \right. 
 \]
 \item if $l = 5$, $N_{5}(k)=0$ unless $8\vert k$, in which case $N_{5}(k)=\frac{3}{4}\chi_{4}\left(\frac{k}{8}\right)$\ ;
 \item if $l\geq 6$, define ${\hat k} = \frac{k}{2^{l-3}}$ or $\frac{4-k}{2^{l-3}}$. Then, $N_{l}(k)=0$ unless ${\hat k}\in \mathbb{Z}$, in which case
 \[N_{l}(k)= 
  \left\{ \begin{array}{ll}
 -2^{-\frac{l+1}{2}}\cos\left(\frac{{\hat k}+1}{4}\pi\right)\left[1 + (-1)^{l+{\hat k}}\right] & \text{if}\ \ 2^{l-3}\vert k\ , \\
 \\
2^{\min(3,\;l-5) -\frac{l+1}{2}}\cos\left(\frac{{\hat k}+1}{4}\pi\right)\left[1 + (-1)^{l+{\hat k}}\right] & \text{if}\ \ 2^{l-3}\vert (k-4) \ .
\end{array} \right. 
 \] 
 \end{enumerate}
 \end{lemma}
  \begin{remark} Note that for odd $w$, $\cos\left(\frac{w+1}{4}\pi\right)= \frac{1}{2}\chi_{8}(w)\left(1-\chi_{4}(w)\right)=\frac{1}{2}\left( \left(\frac{8}{w}\right)_{\text{J}} - \left(\frac{-8}{w}\right)_{\text{J}}\right)$.
 \end{remark} 
 Combining Lemmas \ref{Lb3} and \ref{Lb5} gives us
 
 \begin{proposition}\label{Lb6} Suppose $k \neq 0$ or $4$. Let $\delta_{2}(k)$ denote the mass at $p=2$. Then $\delta_{2}(k) = 0 $ only when $k \equiv 3 \ ({\text mod}\ 4)$. Otherwise $\delta_{2}(k)\geq \frac{3}{4}$. More precisely,  
 \begin{enumerate}[label=(\arabic*).]
 \item If $k$ is odd then 
 \[\delta_{2}(k)= \frac{3}{4}\left(1+\chi_{4}(k)\right)\left(2-\chi_{8}(k)\right)\ ;
 \]
 \item if $2\vert \vert k$, then $\delta_{2}(k)= 1 $\ ;
 \vspace{2pt}
 \item if $4\vert\vert k$, define $\eta \geq 3$ with $2^{\eta}\vert\vert (k-4)$, and put $4-k = 2^\eta w$ with $w$ odd.
 \begin{enumerate}[label=(\alph*).]
 \item if $\eta \geq 6$ is even, $\delta_{2}(k) = \frac{13}{4} - 2^{-\frac{\eta -6}{2}} - \left(\frac{-4}{w}\right)_{{\rm J}}2^{-\frac{\eta -4}{2}} + \left( \left(\frac{8}{w}\right)_{{\rm J}} - \left(\frac{-8}{w}\right)_{{\rm J}}\right)2^{-\frac{\eta -2}{2}}$\ ,
 \item if $\eta\geq 7$ is odd, $\delta_{2}(k) =\frac{13}{4} - 2^{-\frac{\eta -6}{2}} + (-1)^{w}2^{-\frac{\eta -5}{2}} $\ ,
 \item if $\eta=3$, $\delta_{2}(k) =1$\ ,
 \item if $\eta=4$, $\delta_{2}(k) = 2 + \frac{1}{4}\left( \left(\frac{8}{w}\right)_{{\rm J}} - \left(\frac{-8}{w}\right)_{{\rm J}}\right)$\ ,
 \item if $\eta=5$, $\delta_{2}(k) = 2 + \frac{1}{4}(-1)^w$\ ;
 \end{enumerate}
 \vspace{2pt}
 \item if $8\vert k$, define $\eta \geq 3$ with $2^{\eta}\vert\vert k$, and put $k = 2^\eta w$ with $w$ odd.
  \begin{enumerate}[label=(\alph*).]
  \item if $\eta \geq 6$ is even, $\delta_{2}(k) = \frac{5}{2} - 2^{-\frac{\eta -6}{2}} + \left(\frac{-4}{w}\right)_{\text{J}}2^{-\frac{\eta -4}{2}} - \left( \left(\frac{8}{w}\right)_{{\rm J}} - \left(\frac{-8}{w}\right)_{{\rm J}}\right)2^{-\frac{\eta -2}{2}}$\ ,
 \item if $\eta\geq 7$ is odd, $\delta_{2}(k) =\frac{5}{2} - 2^{-\frac{\eta -6}{2}} - (-1)^{w}2^{-\frac{\eta -5}{2}} $\ ,
 \item if $\eta=3$, $\delta_{2}(k) =\frac{5}{2}$\ ,
 \item if $\eta=4$, $\delta_{2}(k) = \frac{5}{4} - \frac{1}{4}\left( \left(\frac{8}{w}\right)_{{\rm J}} - \left(\frac{-8}{w}\right)_{{\rm J}}\right)$\ ,
 \item if $\eta=5$, $\delta_{2}(k) = \frac{5}{4} -  \frac{1}{4}(-1)^w$\ ;
  \end{enumerate}
 \end{enumerate}
 \end{proposition}

%%%%%%%%%%%%%%%%%%%%%%%%%%%%%%%%%%%%%%%%%%%%%%%%%%%%%%%%%%%%%%%%%%%
%%%%%%%%%%%%%%%%%%%%%%%%%%%%%%%%%%%%%%%%%%%%%%%%%%%%%%%%%%%%%%%%%%%
%%%%%%%%%%%%%%%%%%%%%%%%%%%%%%%%%%%%%%%%%%%%%%%%%%%%%%%%%%%%%%%%%%%

 \subsection{Local factors associated with $V_{a_1,a_2}$ for $p=2$.}\ 
 
The analog of \eqref{LB2} is 
 \begin{equation}\label{ED1}
 N_{l}(a_1,a_2) = 2^{-4l}\sideset{}{^*}\sum_{b\ (\text{mod}\ 2^l)}e\left(b\frac{a_{1}^2 - a_{2}^2}{2^l}\right)Q_{l}(b,a_1)\overline{Q_{l}(b,a_2)}\ ,
  \end{equation}
 with $\delta_{2}(a_1,a_2) = 1+ \sum_{l=1}^{\infty}\ N_{l}(a_1,a_2)$. In what follows we have $l\geq 1$.

 %%%%%%%%%%%%%%%%%%%%%%%%%%%%%%%%%%%%%%%%%%%%%%%%%%%%%%%%%%
 \subsubsection{}
 \label{tau1}
 Suppose $2\nmid a_1a_2$ and $2^{\eta}\vert\vert (D_{a_1} -D_{a_2})$ with $\eta \geq 3$. Then, by Lemma \ref{Lb2} we have $Q_{l}(b,a_1)\overline{Q_{l}(b,a_2)} = 2^{2l}$.

 Hence we get 
 
 \begin{equation}\label{ED3}
 N_{l}(a_1,a_2) = \left\{ \begin{array}{ll}
 2^{-l-1} & \text{if}\ \ 1\leq l\leq \eta\ , \\
-2^{-l-1} & \text{if}\ \ l=\eta +1 \ ,\\
0 & \text{if} \ \ l\geq \eta + 2 \ ,
\end{array} \right. 
\end{equation}
 so that $\delta_{2}(a_1,a_2) = \frac{3}{2}\left(1-  2^{-\eta -1}\right)$ .
 
 %%%%%%%%%%%%%%%%%%%%%%%%%%%%%%%%%%%%%%%%%%%%%%%%%%%%%%%%%%
 \subsubsection{}
 \label{tau2}
 Next suppose $2\vert a_1$ and $2\nmid a_2$, so that $\eta=0$. Put $A_1 = \frac{a_{1}^2}{4}-1$ and $2^\theta \vert\vert A_1$, with $\theta = 0$ or $\theta \geq 3$. By Lemmas \ref{Lb2}, \ref{Lb1} and \eqref{ED1}, we have $N_{l}(a_1,a_2) = 0$ if $l=1$ or $ l=\theta +1$. Otherwise, with $l\geq2$ we have 
 
  \begin{equation}\label{ED4}
 N_{l}(a_1,a_2) = \frac{(-1)^l}{2^{3l}}\sideset{}{^*}\sum_{b\ (\text{mod}\ 2^l)}e\left(b\frac{\sigma}{2^l}\right)F_{l}(b)F_{l}(bA_1)\ ,
 \end{equation}
 where $\sigma = a_{1}^2 - a_{2}^2$ is odd.
 
 If $2\leq l \leq \theta$, Lemma \ref{Lb1} shows that 
 \[
 N_{l}(a_1,a_2) = (-1)^l2^{-\frac{3}{2}l}\sideset{}{^*}\sum_{b\ (\text{mod}\ 2^l)}e\left(b\frac{\sigma}{2^l}\right)\chi_{8}(b)\left(1 +\chi_{4}(b)i\right)\ .
 \]
 
 We now use 
 
 \begin{lemma}\label{LD1} Suppose $2^{\mu}\vert\vert \sigma$ with $\mu \geq 0$. Then,
 \begin{enumerate}[label=(\arabic*).]
 \item If $l\geq 2$,
 \[\sideset{}{^*}\sum_{b\ (\text{mod}\ 2^l)}e\left(b\frac{\sigma}{2^l}\right)\chi_{4}(b) = \left\{ \begin{array}{ll}
 2^{\mu +1}\chi_{4}\left(\frac{\sigma}{2^{\mu}}\right)i & \text{if}\ \ l =2 +\mu \ , \\
0 & \text{otherwise} \ ;
\end{array}\right. \]
\item If $l\geq 3$, and $\mathfrak{a}=0$ or $1$,
\[  \sideset{}{^*}\sum_{b\ (\text{mod}\ 2^l)}e\left(b\frac{\sigma}{2^l}\right)\chi_4(b)^{\mathfrak{a}}\chi_{8}(b) = \left\{ \begin{array}{ll}
 2^{\mu + \frac{3}{2}}\chi_4\left(\frac{\sigma}{2^{\mu}}\right)^{\mathfrak{a}}\chi_{8}\left(\frac{\sigma}{2^{\mu}}\right)i^{\mathfrak{a}} & \text{if}\ \ l =3+\mu\ , \\
0 & \text{otherwise} \ ,
\end{array}\right. \]
 \end{enumerate}
 \end{lemma}
 
 If $\theta \geq 3$, then for $2\leq l \leq \theta$, we have
  \begin{equation}\label{ED5}
 N_{l}(a_1,a_2) = \left\{ \begin{array}{ll}
 \frac{1}{4} & \text{if}\ \ l =2\ \text{or} \ 3  \ , \\
  0 & \text{if}\ \ l \geq 4  \ . \\
\end{array}\right.
 \end{equation}
 Here we have used the fact that $\sigma = a_{1}^2 -a_{2}^2=(2u)^2 -v^2$ with $u$ and $v$ both odd.
 
 For $l\geq \theta +2$, with $\theta \geq 0$, we get in \eqref{ED4}
 
 \[
 N_{l}(a_1,a_2) = \frac{(-1)^l}{2^{2l -\theta}}\ \sideset{}{^*}\sum_{b\ (\text{mod}\ 2^l)}e\left(b\frac{\sigma}{2^l}\right)
 \chi_{8}(b)^{\theta}\left[\left(1 - \chi_{4}\left(\frac{A_1}{2^{\theta}}\right)\right)+ i\chi(b)\left(1 - \chi_{4}\left(\frac{A_1}{2^{\theta}}\right)\right)\right]
 \]
 
 Applying Lemma \ref{LD1} shows that $N_{l}(a_1,a_2)=0$ for all $l\geq \theta+2$. 
 
 Thus, if $\theta=0$, then $\delta_{2}(a_1,a_2)=1$, while for $\theta\geq 3$ we get $\delta_{2}(a_1,a_2)= \frac{3}{2}$.
 
 %%%%%%%%%%%%%%%%%%%%%%%%%%%%%%%%%%%%%%%%%%%%%%%%%%%%%%%%%%%
 
 \subsubsection{}
    Assume $a_1$ and $a_2$ are both even and put $A_j = \left(\frac{a_{j}}{2}\right)^2 -1$, so that $A_j \equiv 0,\ 3$ modulo 4. Put $2^{\theta_j}\vert\vert A_j$ with $\theta_j \geq 0$ but $\theta_j \neq 1,\ 2$ and $C_j = A_j 2^{-\theta_j}$ odd. We will assume that $\theta_1 \leq\theta_2$. Then $a_{1}^2 -a_{2}^2 = 4(A_1-A_2)$ so that $\eta = 2+t$, say, with $2^t\vert\vert(A_1 -A_2)$ with $t\neq 1$. We have
 
  \begin{equation}\label{ED6}
 N_{l}(a_1,a_2) = 2^{-4l}\sideset{}{^*}\sum_{b\ (\text{mod}\ 2^l)}e\left(b\frac{4\sigma'}{2^l}\right)|F_{l}(b)|^2F_{l}(bA_1)\overline{F_{l}(bA_2)}\ ,
 \end{equation}
where we put $\sigma' = A_1-A_2$.
 
Note that $N_l(a_1,a_2)=0$ if $l=1$ or $l= \theta_j +1$, so we assume $l\geq 2$.
 
 \begin{enumerate}[label=(\Roman*).]
 \item If $2\leq l\leq \theta_1$, then $F_l(bA_j)= 2^l$. Using $|F_l(b)|^2 = 2^{l+1}$ and \eqref{ED3} shows that $N_{l}(a_1,a_2)=1$. 
 
\item If $\theta_{1}+2\leq l \leq \theta_2$, using $F_l(bA_2)= 2^l$ and Lemma \ref{Lb1}(d) gives us $$N_{l}(a_1,a_2) =  2^{-\frac{3}{2}l +\frac{1}{2}\theta_1 +1}\chi_{8}(C_1)^{l+\theta_1}\mathfrak{S}_{l}(a_1,a_2)\ ,$$ where
 
  \begin{equation}\label{ED7}
 \mathfrak{S}_{l}(a_1,a_2)= \sideset{}{^*}\sum_{b\ (\text{mod}\ 2^l)}e\left(b\frac{4\sigma'}{2^l}\right)\chi_{8}(b)^{l+\theta_1}\left[1 + i\chi_{4}(bC_1)\right]\ .
 \end{equation}
 Applying \eqref{ED3} and Lemma \ref{LD1} shows that $\mathfrak{S}_{l}(a_1,a_2)=0$ except for the cases $\mathfrak{S}_{\theta_1 +2}(a_1,a_2) = 2^{\theta_1 +1}$ and $\mathfrak{S}_{\theta_1 +4}(a_1,a_2)= -2^{\theta_3 +1}$. 
 
 Thus in this range, if  $\theta_1 +2 \leq \theta_2$ then $N_{\theta_1 +2}(a_1,a_2)= \frac{1}{2}$; if $\theta_1 +4\leq \theta_2$, then  $N_{\theta_1 +4}(a_1,a_2)= -\frac{1}{4}$ and $N_{l}(a_1,a_2)= 0$ for all other $l$.
 
 \item For $l \geq \theta_2 +2$ we get $$N_{l}(a_1,a_2) =  2^{-2l +1+\frac{1}{2}(\theta_1 +\theta_2)}\chi_{8}(C_1)^{l+\theta_1}\chi_{8}(C_2)^{l+\theta_2}\mathfrak{S}_{l}(a_1,a_2)\ ,$$ with
 
  \[
 \mathfrak{S}_{l}(a_1,a_2)= \sideset{}{^*}\sum_{b\ (\text{mod}\ 2^l)}e\left(b\frac{4\sigma'}{2^l}\right)\chi_{8}(b)^{\theta_1+\theta_2}\left[\left(1 + \chi_{4}(C_1C_2)\right) + i\chi_4(b)\left(\chi_4(C_1)-\chi_4(C_2)\right)\right]\ .
 \]
 
 Since $C_1$ and $C_2$ are both odd, we have $C_1 \equiv (-1)^{\mathfrak{a}} C_2$ modulo $4$ with $\mathfrak{a}= 0,\ 1$. Hence 
 
 \begin{equation}\label{ED8}
 \mathfrak{S}_{l}(a_1,a_2)= 2i^{\mathfrak{a}} \sideset{}{^*}\sum_{b\ (\text{mod}\ 2^l)}e\left(b\frac{4\sigma'}{2^l}\right)\chi_4(b)^{\mathfrak{a}}\chi_{8}(b)^{\theta_1+\theta_2}\ .
 \end{equation}

 If $\theta_1 \neq \theta_2$, then $t=\theta_1$ and $l\geq \theta_2 + 2 \geq 3$. We apply \eqref{ED3} and Lemma \ref{LD1} with $\mu = \theta_1 +2$ so that $N_{l}(a_1,a_2)=0$ except possibly when $l= \theta_1 +3$, $\theta_1 +4$ or $\theta_1 +5$. 
 
 If $l= \theta_1 +3$ then necessarily $\theta_2=\theta_1 +1$, in which case we get $N_{l}(a_1,a_2)=0$.
 
 If $l= \theta_1 +4$ then  $\theta_2=\theta_1 +1$ or $\theta_2=\theta_1 +2$. If the former, then $N_{l}(a_1,a_2)=0$. For the latter  we get $\mathfrak{S}_{l}(a_1,a_2)= -\left(1 - \chi_{4}(C_1C_2)\right)2^{\theta_1 +3}$ so that $N_{l}(a_1,a_2)=-\frac{1}{4}\left(1 - \chi_{4}(C_1C_2)\right)$.
 
 If $l= \theta_1 +5$ then  $\theta_2=\theta_1 +1$, $\theta_1 +2$ or $\theta_1 +3$. If $\theta_2=\theta_1 +2$, then $N_{l}(a_1,a_2)=0$. If $\theta_2=\theta_1 +1$, then $\mathfrak{S}_{l}(a_1,a_2)= 2^{l-\frac{1}{2}}\chi_{4}(C_1)\chi_4\chi_8(C_2)$ while if $\theta_2=\theta_1 +3$, then $\mathfrak{S}_{l}(a_1,a_2)= 2^{l-\frac{1}{2}}\chi_{4}\chi_8(C_1)\chi_4(C_2)$. In these latter cases, we get $N_l(a_1,a_2)= 2^{-4} \chi_{4}\chi_8(C_1)\chi_4(C_2)$ if $\theta_2=\theta_1 +1$ and $N_l(a_1,a_2)= 2^{-3}\chi_{4}(C_1)\chi_4\chi_8(C_2)$ if $\theta_2=\theta_1 +3$ .
 
 Next, suppose $\theta_1=\theta_2 =\theta$ with $t\geq \theta$ and $l\geq \theta +2$. Then, from \eqref{ED8}, we have $N_{l}(a_1,a_2)=0$ if $l\geq t+5$. For the remaining cases we have
 \begin{enumerate}[label=($\alph*$).]
 \item if $l=t+4$, $\mathfrak{S}_{l}(a_1,a_2)=-\chi_{4}\left(\frac{A_1-A_2}{2^t}\right)\left[\chi_4(C_1)-\chi_4(C_2)\right]2^{l-1}$\ ;
 \item if $l=t+3$, then $\mathfrak{S}_{l}(a_1,a_2)= -\left(1+\chi_4(C_1C_2)\right)2^{l-1}$ \ ;
 \item if $\theta +2 \leq l \leq t+2$,  then $\mathfrak{S}_{l}(a_1,a_2)=\left(1+\chi_4(C_1C_2)\right)2^{l-1}$ \ .
 \end{enumerate}
 \end{enumerate}
 Combining give us the following
 
 \begin{proposition}\label{DD1} Let $a_1 \neq \pm a_2$ and $a_j \neq \pm 2$. 
 \begin{enumerate}[label=(\arabic*).]
 \item Suppose $2\nmid a_1a_2$ and $2^{\eta}\vert\vert (D_{a_1} -D_{a_2})$ with $\eta \geq 3$. Then
 \[
 N_{l}(a_1,a_2) = \left\{ \begin{array}{ll}
 2^{-l-1} & \text{if}\ \ 1\leq l\leq \eta\ , \\
-2^{-l-1} & \text{if}\ \ l=\eta +1 \ ,\\
0 & \text{if} \ \ l\geq \eta + 2 \ ;
\end{array} \right. 
\]
\item Suppose $2\vert a_1$ and $2\nmid a_2$, and let $2^{2+\theta} \vert\vert D_{a_1}$. Then
\[
 N_{l}(a_1,a_2) = \left\{ \begin{array}{ll}
 \frac{1}{4} & \text{if}\ \ \theta\geq 3 \ \text{and}\ l\in\{2,3\}\ , \\
0 & \text{otherwise}\ \  \ ;
\end{array} \right. 
\]
\item For $j=1,\ 2$ suppose $2\vert a_j$ , and put $A_j = \frac{1}{4}D_{a_j}$, $C_j = A_j2^{-\theta_j}$ with $2^{\theta_j}\vert\vert A_j$ and assume $\theta_1\leq \theta_2$. Also suppose $2^t \vert\vert(A_1 -A_2)$ so that $t\geq \theta_1$. We have
\begin{enumerate}[label=(\roman*).]
\item $N_{l}(a_1,a_2)=0$ for $l=1$, $\theta_1 +1$ and $\theta_2 +1$\ .
\item For $2\leq l \leq \theta_1$, $N_{l}(a_1,a_2)=1$\ .
\item For $l \geq \theta_1 +2$, and $\theta_1 \neq \theta_2$ we have  $N_{l}(a_1,a_2)=0$ except for the following cases:
\[
N_{l}(a_1,a_2) = \left\{ \begin{array}{ll}
 2^{-4}\chi_4\chi_8(C_1)\chi_4(C_2) & \text{if}\ \ l=\theta_1 +5\ \text{and}\ \theta_2=\theta_1 +1 \ , \\
 2^{-1} & \text{if}\ \ l=\theta_1 +2\ \text{and}\ \theta_2\geq \theta_1 +2 \ , \\
 -2^{-2}\left(1 - \chi_4(C_1C_2)\right) & \text{if}\ \ l=\theta_1 +4\ \text{and}\ \theta_2=\theta_1 +2 \ , \\
 2^{-3}\chi_4(C_1)\chi_4\chi_8(C_2) & \text{if}\ \ l=\theta_1 +5\ \text{and}\ \theta_2=\theta_1 +3 \ , \\
-2^{-2}& \text{if}\ \ l=\theta_1 +4 \ \ \text{and}\ \theta_2 \geq \theta_1 +4 \ .
\end{array} \right. 
\]
\item If $\theta_1=\theta_2=\theta$, then
\[
N_{l}(a_1,a_2) = \left\{ \begin{array}{ll}
 -2^{-(l-\theta )}\chi_8(C_1C_2)^{l+\theta }\left(1+\chi_4(C_1C_2)\right) & \text{if}\ \  \theta \leq  l-2\leq t\  , \\
 -2^{-(t-\theta +3)}\chi_8(C_1C_2)^{t+\theta +1}\left(1+\chi_4(C_1C_2)\right) & \text{if}\ \ l=t +3\  , \\
-2^{-(t-\theta +4)}\chi_8(C_1C_2)^{t+\theta}\chi_{4}\left(\frac{C_1-C_2}{2^{t-\theta}}\right)\left[\chi_4(C_1)-\chi_4(C_2)\right] & \text{if}\ \  l=t +4\  , \\
 0 & \text{if}\ \  l\geq t+5 \ .
\end{array} \right. 
\]
\end{enumerate}
 \end{enumerate}
 \end{proposition}
 
 \begin{Corollary}\label{DD2} For $a_1 \neq \pm a_2$ and $a_j \neq \pm 2$, let $2^{\theta}\vert\vert {\rm gcd}(D_{a_1},D_{a_2})$. Then $\delta_{2}(a_1,a_2) = \theta + O(1)$ and $\delta_{2}(a_1,a_2) - \delta_{2}^{(m)}(a_1,a_2) = O\left(2^{-B}\right)$, where the implied constants are absolute.
 \end{Corollary}
\vskip 0.5in

\end{appendices}
%\begin{thebibliography}{9}
\footnotesize
\setlength{\parskip=0.0pt}
\setlength{\lineskip=0.0pt} 
\bibliographystyle{amsalpha}
\raggedright
\bibliography{references}

\providecommand{\bysame}{\leavevmode\hbox to3em{\hrulefill}\thinspace}
\providecommand{\MR}{\relax\ifhmode\unskip\space\fi MR }
\providecommand{\MRhref}[2]{%
  \href{http://www.ams.org/mathscinet-getitem?mr=#1}{#2}
}
\providecommand{\href}[2]{#2}
\begin{thebibliography}{CTWX20}

\bibitem[Aur15]{Auroux}
D.~Auroux, \emph{Factorizations in $ {SL} (2,\mathbb{Z})$ and simple examples
  of inequivalent {S}tein fillings}, J. of Symplectic Geometry \textbf{13}
  (2015), no.~2, 261--277.

\bibitem[Bar94]{Baragar}
A.~Baragar, \emph{Integral solutions of {M}arkov-{H}urwitz equations}, J. of
  Number Theory \textbf{49} (1994), no.~1, 27 -- 44.

\bibitem[Beu95]{Be}
F.~Beukers, \emph{Ternary form equations}, J. Number Theory \textbf{54} (1995),
  113--133.

\bibitem[BF11]{BF11}
J.~Bourgain and E.~Fuchs, \emph{A proof of the positive density conjecture for
  integer {A}pollonian circle packings}, J. Amer. Math. Soc. \textbf{24}
  (2011), 945--967.

\bibitem[BG06]{BG06}
V.~Blomer and A.~Granville, \emph{Estimates for representation numbers of
  quadratic forms}, Duke Math. J. \textbf{135 No. 2} (2006), 261--302.

\bibitem[BGS16a]{BGS17}
J.~{Bourgain}, A.~{Gamburd}, and P.~{Sarnak}, \emph{{Markoff surfaces and
  strong approximation: 1}}, arXiv:math/1607.01530 (2016).

\bibitem[BGS16b]{BGS16}
J.~Bourgain, A.~Gamburd, and P.~Sarnak, \emph{Markoff triples and strong
  approximation}, Comptes Rendus Math. \textbf{354} (2016), no.~2, 131 -- 135.

\bibitem[Bha04]{Bha04}
M.~Bhargava, \emph{Higher composition laws {I}: A new view on {G}auss
  composition, and quadratic generalizations}, Annals of Math. \textbf{159}
  (2004), no.~1, 217--250.

\bibitem[BHB09]{B-HB}
T.~D. Browning and D.~R. Heath-Brown, \emph{Integral points on cubic
  hypersurfaces}, Analytic Number Theory: Essays in honour of Klaus Roth, CUP,
  2009, pp.~75--90.

\bibitem[BN16]{bn16}
T.~D. Browning and R.~Newton, \emph{The proportion of failures of the {H}asse
  norm principle}, Mathematika \textbf{62} (2016), no.~2, 337--347.

\bibitem[Boo19]{Booker19}
A.~R. Booker, \emph{Cracking the problem with 33}, Res. number theory
  \textbf{5} (2019), no.~26.

\bibitem[Bro15]{bro15}
T.~D. Browning, \emph{A survey of applications of the circle method to rational
  points}, London Math. Soc. Lecture Note Series, pp.~89--113, Cambridge
  University Press, 2015.

\bibitem[Bro17]{Br}
\bysame, \emph{How often does the {H}asse principle hold?}, Proceedings of the
  AMS Summer Institute in Algebraic Geometry, Salt Lake City, to appear., 2017.

\bibitem[BS15]{B-S}
M.~Bhargava and A.~Shankar, \emph{Ternary cubic forms having bounded
  invariants, and the existence of a positive proportion of elliptic curves
  having rank 0}, Annals of Math. \textbf{181} (2015), no.~2, 587--621.

\bibitem[Cas85]{Cas85}
J.~W.~S. Cassels, \emph{A note on the {D}iophantine equation $x^3 + y^3 + z^3 =
  3$}, Math. of Comp. \textbf{44} (1985), 265--266.

\bibitem[CG66]{CG66}
J.~W.~S. Cassels and M.~J.~T. Guy, \emph{On the {H}asse principle for cubic
  surfaces}, Mathematika \textbf{13} (1966), no.~2, 111--120.

\bibitem[CL09]{CL}
S.~Cantat and F.~Loray, \emph{Dynamics on character varieties and {M}algrange
  irreducibility of {P}ainlev{\'e} {VI} equation}, Annales de l'institut
  Fourier \textbf{59} (2009), no.~7, 2927--2978.

\bibitem[CTW12]{CThW12}
J.~Colliot-Th{\'e}l{\`e}ne and O.~Wittenberg, \emph{Groupe de {B}rauer et
  points entiers de deux familles de surfaces cubiques affines}, Amer. J. of
  Math. \textbf{134} (2012), no.~No. 5, 1303--1327.

\bibitem[CTWX20]{CTWX}
J.~L. Colliot-Th\'el\`ene, D.~Wei, and F.~Xu, \emph{{B}rauer-{M}anin
  obstruction for {M}arkoff surfaces}, Annali della Scuola Normale Superiore di
  Pisa \textbf{vol. XXI} (2020), 1257--1313.

\bibitem[CV94]{CV94}
W.~Conn and L.~N. Vaserstein, \emph{{On sums of three integral cubes}},
  Contemp. Math., vol. 166, Amer. Math. Soc., 1994, pp.~285--294.

\bibitem[CZ06]{CZ}
P.~Corvaja and U.~Zannier, \emph{On the greatest prime factor of {M}arkov
  pairs}, Rendiconti del Seminario Mat. della Universit{\`a} di Padova
  \textbf{116} (2006), 253--260 (eng).

\bibitem[Dav39]{Dav}
H.~Davenport, \emph{On {W}aring's {P}roblem for {C}ubes}, Acta Math. (1939),
  no.~71, 123--143.

\bibitem[DL64]{D-L}
H.~Davenport and D.~J. Lewis, \emph{Non-homogeneous cubic equations}, J. London
  Math. Soc. \textbf{s1-39} (1964), no.~1, 657--671.

\bibitem[FZ14]{CFZ}
C.~Fuchs and U.~Zannier, \emph{Integral points on curves: {S}iegel's theorem
  after {S}iegel's proof}, On Some Applications of Diophantine Approximations
  (Pisa), Scuola Normale Superiore, 2014, pp.~139--157.

\bibitem[GMS22]{GMS}
A.~Ghosh, C.~Meiri, and P.~Sarnak, \emph{Commutators in ${SL}_2$ and {M}arkoff
  surfaces {I}}, New Zealand J. of Math. \textbf{52} (2022), 773–819.

\bibitem[Gol03]{Gol03}
W.~M. Goldman, \emph{The modular group action on real characters of a one-holed
  torus.}, Geometry \& Topology \textbf{7} (2003), 443--486.

\bibitem[HB92]{HB92}
D.~R. Heath-Brown, \emph{The density of zeros of forms for which weak
  approximation fails}, Math. of Comp. \textbf{59} (1992), no.~200, 613--623.

\bibitem[HB96]{HB96}
\bysame, \emph{A new form of the circle method, and its application to
  quadratic forms.}, J. f{\"u}r die reine und angewandte Math. \textbf{481}
  (1996), 149--206.

\bibitem[Hoo16]{Hoo16}
C.~Hooley, \emph{On the representation of numbers by quaternary and quinary
  cubic forms: I}, Acta Arithmetica \textbf{173} (2016), no.~1, 19--39.

\bibitem[Hur07]{Hur07}
A.~Hurwitz, \emph{{\"U}ber eine aufgabe der unbestimmten analysis}, Archiv.
  Math. Phys. \textbf{3} (1907), 185--196.

\bibitem[Leh56]{Le}
D.~H. Lehmer, \emph{On the {D}iophantine equation $x^3+y^3+z^3= 1$}, J. London
  Math. Soc. \textbf{s1-31} (1956), no.~3, 275--280.

\bibitem[LM20]{LM}
D.~Loughran and V.~Mitankin, \emph{Integral {H}asse principle and strong
  approximation for {M}arkoff surfaces}, IMRN /imrn/rnz114 \textbf{10} (2020),
  1093.

\bibitem[LW54]{LW}
S.~Lang and A.~Weil, \emph{Number of points of varieties in finite fields},
  American J. of Math. \textbf{76} (1954), no.~4, 819--827.

\bibitem[Mar79]{Markoff1}
A.~Markoff, \emph{Sur les formes quadratiques binaires ind{\'e}finies}, Math.
  Annalen \textbf{15} (1879), 381--406.

\bibitem[Mar80]{Markoff2}
\bysame, \emph{Sur les formes quadratiques binaires ind{\'e}finies. (second
  m{\'e}moire)}, Math. Annalen \textbf{17} (1880), 379--399.

\bibitem[Mor53]{Mor53}
L.~J. Mordell, \emph{Integer solutions of the equation $x^2+ y^2 + z^2 +
  2xyz=n$}, J. Lond. Math. Soc. \textbf{s1-28} (1953), no.~4, 500--510.

\bibitem[Mor69]{MorBook}
\bysame, \emph{Diophantine equations}, Acad. Press, London, New York, 1969.

\bibitem[Nie10]{Ni}
N.~Niedermowwe, \emph{The circle method with weights for the representation of
  integers by quadratic forms}, J. of Math. Sciences \textbf{171} (2010),
  no.~6, 753--764.

\bibitem[Odo77]{Odo75}
R.~W.~K. Odoni, \emph{A new equidistribution property of norms of ideals in
  given classes}, Acta Arithmetica \textbf{33(1)} (1977), 53--63.

\bibitem[Sch76]{SchBook}
W.~M. Schmidt, \emph{Equations over finite fields: An elementary approach},
  Lectures Notes in Mathematics, vol. 536, Springer, New York, 1976.

\bibitem[Sch87]{Sch}
\bysame, \emph{Thue equations with few coefficients}, Trans. Amer. Math. Soc.
  \textbf{303} (1987), no.~1, 241--255.

\bibitem[Sie29]{Siegel}
C.~L. Siegel, \emph{{\"U}ber einige {A}nwendungen diophantischer
  {A}pproximationen}, Abh. Preuss. Akad. Wiss. Phys.-Math. Kl. \textbf{1}
  (1929), 209--226.

\bibitem[Thu09]{Thue}
A.~Thue, \emph{{\"U}ber {A}nn{\"a}herungswerte algebraischer {Z}ahlen}, J.
  f{\"u}r die reine und angewandte Math. \textbf{135} (1909), 284--305.

\bibitem[VW02]{VW}
R.~C. Vaughan and T.~D. Wooley, \emph{Waring{'}s problem: a survey}, {N}umber
  {T}heory for the {M}illennium. {III}, A. K. Peters, 2002, pp.~301--340.

\bibitem[Wha20]{Whang}
J.~P. Whang, \emph{Nonlinear descent on moduli of local systems}, Israel Jour.
  of Math. \textbf{240} (2020), 935--1004.

\bibitem[{Wik}19]{wiki19}
{Wikipedia contributors}, \emph{Sums of three cubes --- {Wikipedia}{,} the free
  encyclopedia}, 2019, [Online; accessed 8-September-2019].

\end{thebibliography}

%\end{thebibliography}

\end{document}